\documentclass[11pt,a4paper]{amsart}
\usepackage[foot]{amsaddr} 

\usepackage[english]{babel}
\usepackage[T1]{fontenc}

\usepackage[left=1in, right=1in, top=1in, bottom=1in]{geometry}

\usepackage{amssymb,amsmath,amsfonts,mathrsfs,bm,bbm,dsfont}

\usepackage[numbers,sort]{natbib}
\usepackage{bibentry} 

\usepackage{enumitem}
\setlist[itemize]{label= \raisebox{0.25ex}{\scalebox{0.5}{$\blacksquare$}}}

\usepackage{algorithm}
\usepackage{algpseudocode}

\newcommand{\SetAlgoLined}{}
\newcommand{\KwIn}[1]{\Statex \textbf{Input:} #1}
\newcommand{\KwInputs}[1]{\Statex \textbf{Input:} #1}
\newcommand{\KwOut}[1]{\Statex \textbf{Output:}  #1}
\newcommand{\KwRet}[1]{\State \textbf{return}  #1}

\newcommand{\KwDepend}[1]{\Statex \textbf{Depends on:} #1}

\usepackage{thmtools}
\usepackage{thm-restate}
\usepackage{apxproof}

\usepackage{hyperref}
\usepackage{cleveref}

\usepackage{bookmark}

\usepackage{graphicx}
\usepackage{xcolor}
\usepackage{framed}

\usepackage[center]{caption}
\usepackage{subcaption}
\usepackage{array}
\newcolumntype{L}[1]{>{\raggedright\let\newline\\\arraybackslash\hspace{0pt}}m{#1}}
\newcolumntype{C}[1]{>{\centering\let\newline\\\arraybackslash\hspace{0pt}}m{#1}}
\newcolumntype{R}[1]{>{\raggedleft\let\newline\\\arraybackslash\hspace{0pt}}m{#1}}

\usepackage{multicol,multirow}

\usepackage{mathtools}
\usepackage{anyfontsize}

\definecolor{vlightgrey}{gray}{0.975}

\usepackage{xspace}

\newtheorem{theorem}{Theorem}[section]
\newtheorem{lemma}[theorem]{Lemma}
\newtheorem{prop}[theorem]{Proposition}
\newtheorem{corol}[theorem]{Corollary}
\theoremstyle{definition}
\newtheorem{assum}[theorem]{Assumption}
\theoremstyle{definition}
\newtheorem{ex0}[theorem]{Example}
\theoremstyle{remark}
\newtheorem{rem}[theorem]{Remark}
\theoremstyle{definition}
\newtheorem{dfn}[theorem]{Definition}

\newtheorem*{thm*}{Theorem}

\newcounter{cnt}
\newenvironment{workflow}[1]{\refstepcounter{cnt}
	\vspace{1em}\noindent
	\begin{minipage}{\textwidth}
		\rule{\textwidth}{1pt}
		\textbf{Workflow~\thecnt}: #1 \vspace{-1ex}\\ 
		\rule{\textwidth}{1pt}\vspace{-3ex}
	}{\vspace{-1ex}\rule{\textwidth}{1pt}\vspace{1.5em}\end{minipage}}
\Crefname{cnt}{Workflow}{Workflows}

\newenvironment{TraitV}[3]{%
	\MakeFramed {\advance\hsize-\width}}%
{\endMakeFramed}

\definecolor{colorExLine}{rgb}{0.5372549, 0.9294118, 0.57647065}

\definecolor{colorAssumLine}{rgb}{0.961,0.306,0.259}
\newenvironment{assumb}[1]
{
	\MakeFramed {\advance\hsize-\width}
	\restatable{assum}{#1}
}
{
	\endrestatable\endMakeFramed
}

\newsavebox{\overlongequation}

\newcommand{\bi}{[\![}
\newcommand{\ei}{]\!]}

\newcommand{\e}{\mathbb{E}}
\newcommand{\cov}{\mathrm{Cov}}
\newcommand{\var}{\mathrm{Var}}

\newcommand{\C}{\mathbb{C}}
\newcommand{\R}{\mathbb{R}}
\newcommand{\N}{\mathbb{N}}

\DeclareFontFamily{U}{mathx}{\hyphenchar\font45}
\DeclareFontShape{U}{mathx}{m}{n}{<-> mathx10}{}
\DeclareSymbolFont{mathx}{U}{mathx}{m}{n}
\DeclareMathAccent{\widebar}{0}{mathx}{"73}
\newcommand{\Diag}{\mathop{\mathrm{Diag}}}

\newcommand{\q}[1]{``#1''}%

\newcommand{\peq}{.} 
\newcommand{\veq}{,} 

\newcommand{\speq}{.} 
\newcommand{\sveq}{,} 

\newcommand{\kl}{Karhunen--Lo\`{e}ve\xspace}

\newcommand{\di}{\mathrm{d}}
\newcommand{\dd}{\,\mathrm{d}}

\newcommand{\psds}{power spectral densities\xspace}
\newcommand{\psd}{power spectral density\xspace}

\usepackage[title]{appendix}

\newcommand{\sm}{Supplementary Materials\xspace}

\definecolor{darkgreen}{rgb}{0,.6,0}

\Crefname{prop}{Proposition}{Propositions}
\Crefname{corol}{Corollary}{Corollaries}
\Crefname{assum}{Assumption}{Assumptions}
\Crefname{rem}{Remark}{Remarks}
\Crefname{algorithm}{Algorithm}{Algorithms}
\Crefname{ex0}{Example}{Examples}
\Crefname{dfn}{Definition}{Definitions}

\begin{document}


\title[Galerkin--Chebyshev approximation of Gaussian random fields]{Galerkin--Chebyshev approximation of Gaussian random fields on compact Riemannian manifolds}
\author[A.~Lang]{Annika Lang} \address[Annika Lang]{\newline Department of Mathematical Sciences
	\newline Chalmers University of Technology \& University of Gothenburg
	\newline S--412 96 G\"oteborg, Sweden. } \email[Annika Lang]{annika.lang@chalmers.se}

\author[M.~Pereira]{Mike Pereira} \address[Mike Pereira]{\newline Centre for Geosciences and Geoengineering \newline
	 Mines Paris, PSL University \newline
	 77300 Fontainebleau, France} \email[Mike Pereira]{mike.pereira@minesparis.psl.eu}

\maketitle


\begin{abstract}
	A new numerical approximation method for a class of Gaussian random fields on compact connected oriented Riemannian manifolds is introduced. This class of random fields is characterized by the Laplace--Beltrami operator on the manifold. A Galerkin approximation is combined with a polynomial approximation using Chebyshev series. This so-called Galerkin--Chebyshev approximation scheme yields efficient and generic sampling algorithms for Gaussian random fields on manifolds. Strong and weak orders of convergence for the Galerkin approximation and strong convergence orders for the Galerkin--Chebyshev approximation are shown and confirmed through numerical experiments.
\end{abstract}

\medskip

\begin{quotation}
	\begin{footnotesize}
		\begin{description}[font=\normalfont\itshape]
			\item[Keywords]  
			\item[Mathematics Subject Classification] .
		\end{description}
	\end{footnotesize}
\end{quotation}

\medskip
\thanks{
	\textsc{Acknowledgement}.This work was partially supported by the Swedish Research Council (VR) through grant no.\ 2020-04170, by the Wallenberg AI, Autonomous Systems and Software Program (WASP) funded by the Knut and Alice Wallenberg Foundation, by the Chalmers AI Research Centre (CHAIR), and by the Simons Foundation Award No.\ 663281 granted to the Institute of Mathematics of the Polish Academy of Sciences for the years 2021--2023. The authors thank Christoph Schwab for his helpful comments.}


\section{Introduction}

Models for random fields defined on manifolds are of key importance in many application areas such as environmental sciences, geosciences and cosmological data analysis \cite{marinucci2011random}. While one area of interest is dealing with actual data that lies on surfaces and doing inference based on these data, we focus in this work on the primarily needed modeling and sampling of these random fields. More specifically, we propose a generic approach to define and numerically approximate a particular class of Gaussian random fields on (compact) Riemannian manifolds in a computationally efficient manner.

The main contributions of this work are the following. First, we propose a general approach to model and discretize a class of Gaussian random fields~$\mathcal{Z}$ defined on compact connected oriented Riemannian manifolds $\mathcal{M}$ via functions of the Laplace--Beltrami operator $-\Delta_{\mathcal{M}}$ of the manifold. We define the random field~$\mathcal{Z}$ through a series expansion, and derive a finite-dimensional approximation~$\mathcal{Z}_n$ on any finite-dimensional function space~$V_n$, e.g.\ a finite element space and not necessarily the spectral representation of the series expansion. To do so, we use (functions of) the Galerkin approximation of~$-\Delta_{\mathcal{M}}$ on~$V_n$. This approximation of the field allows us to give a closed form for the covariance matrix of the coefficients in basis representation of $\mathcal{Z}_n$, and hence an explicit way to sample these correlated random coefficients. Secondly, we propose an approximation of the discretized field $\mathcal{Z}_n$ based on Chebyshev polynomials which allows to sample these coefficients in a computationally efficient manner. Finally, we show convergence in mean-square and in the covariance of~$\mathcal Z_n$ to~$\mathcal Z$ and give the associated convergence rates. We also derive a convergence result for the root-mean-squared error induced by the Chebyshev approximation.

This approach, which we call Galerkin--Chebyshev approximation, provides efficient and scalable algorithms for computing samples of the discretized field. For instance, when defining the discretized field using a linear finite element space of dimension $n$, we obtain sampling costs that scale linearly with $n$ and with the order of the considered Chebyshev polynomial approximation, and storage costs that scale linearly with $n$. In particular, computational costs of {essentially} $\mathcal{O}(\epsilon^{-2/\rho})$ are then required to sample, with accuracy $\epsilon>0$, Gaussian random fields with a Matérn covariance function on a two-dimensional manifold (where $\rho$ denotes the rate at which the root-mean-squared error between the random field and its discretization converges to zero). 

So far the focus of the literature for random fields on manifolds has been on the sphere.
Extensive literature on the definition, properties, and efficient use of random fields on the sphere is available (see \cite{marinucci2011random} for a review). A first simulation approach aims at characterizing valid covariance functions on the sphere that model the correlation between two points using the arc length distance separating them \cite{gneiting2013strictly,Huang2011}. A second approach relies on the fact that stationary Gaussian random fields on the sphere have a basis expansion with respect to the spherical harmonic functions \cite{jones1963stochastic}. The resulting \kl expansion is used to derive simulation methods and to characterize the covariance structure of the resulting fields \cite{marinucci2011random,lang2015isotropic,Lantuejoul2019,Emery2019,CGLP20}.
Finally, models have also been proposed to deal with both space-time data \cite{porcu2016spatio} and anisotropy \cite{estrade2019covariance} on the sphere. Discretization methods that do not rely on \kl expansions are, for instance, using the existence of Parseval frames on the sphere~\cite{bachmayr2020multilevel} or relying on a regular discretization of the sphere, Markov properties, and fast Fourier transforms~\cite{creasey2018fast}.

However, the work done for random fields on a sphere hardly generalizes to other spatial domains, as they heavily rely  on the intrinsic properties of the sphere as a surface, and on the spherical harmonics. If now random fields on more general manifolds are of interest, Adler and Taylor \cite{adler2009random} provide a review of the theory used to define them, primarily focused on their geometry and excursion sets. The goal of this work is to propose and analyze a second approach, which generalizes the expansion approach on the sphere, and results in efficient algorithms for sampling Gaussian random fields on a manifold. Examples of samples of the resulting fields on different manifolds are shown in \Cref{fig:ex_sim} and show the flexibility of the approach, since it can be applied to widely different domains.

\begin{figure}
	\begin{subfigure}{0.33\textwidth}
		\centering
		\includegraphics[height=0.18\textheight]{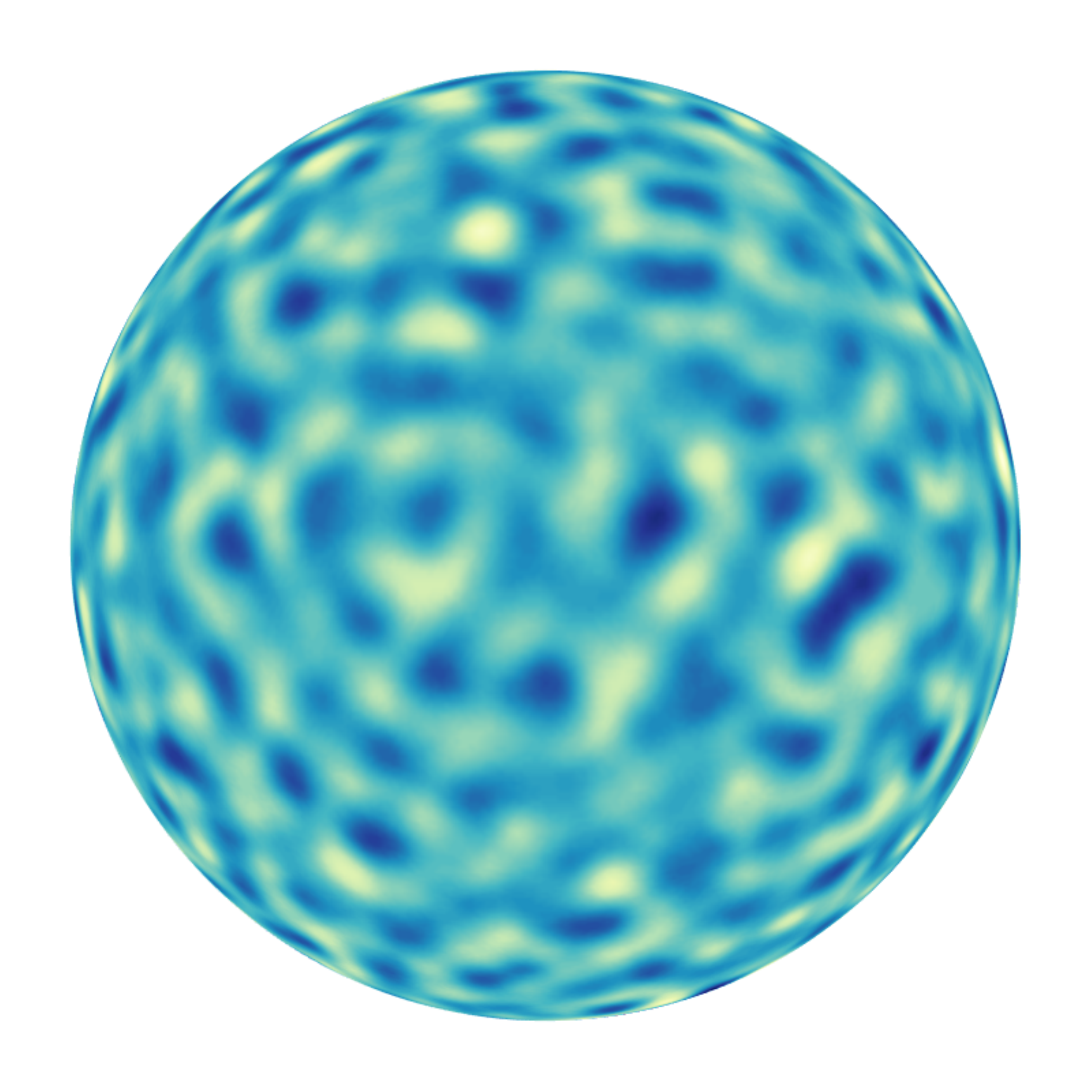}
	\end{subfigure}\hfill
	\begin{subfigure}{0.33\textwidth}
		\centering
		\includegraphics[height=0.18\textheight]{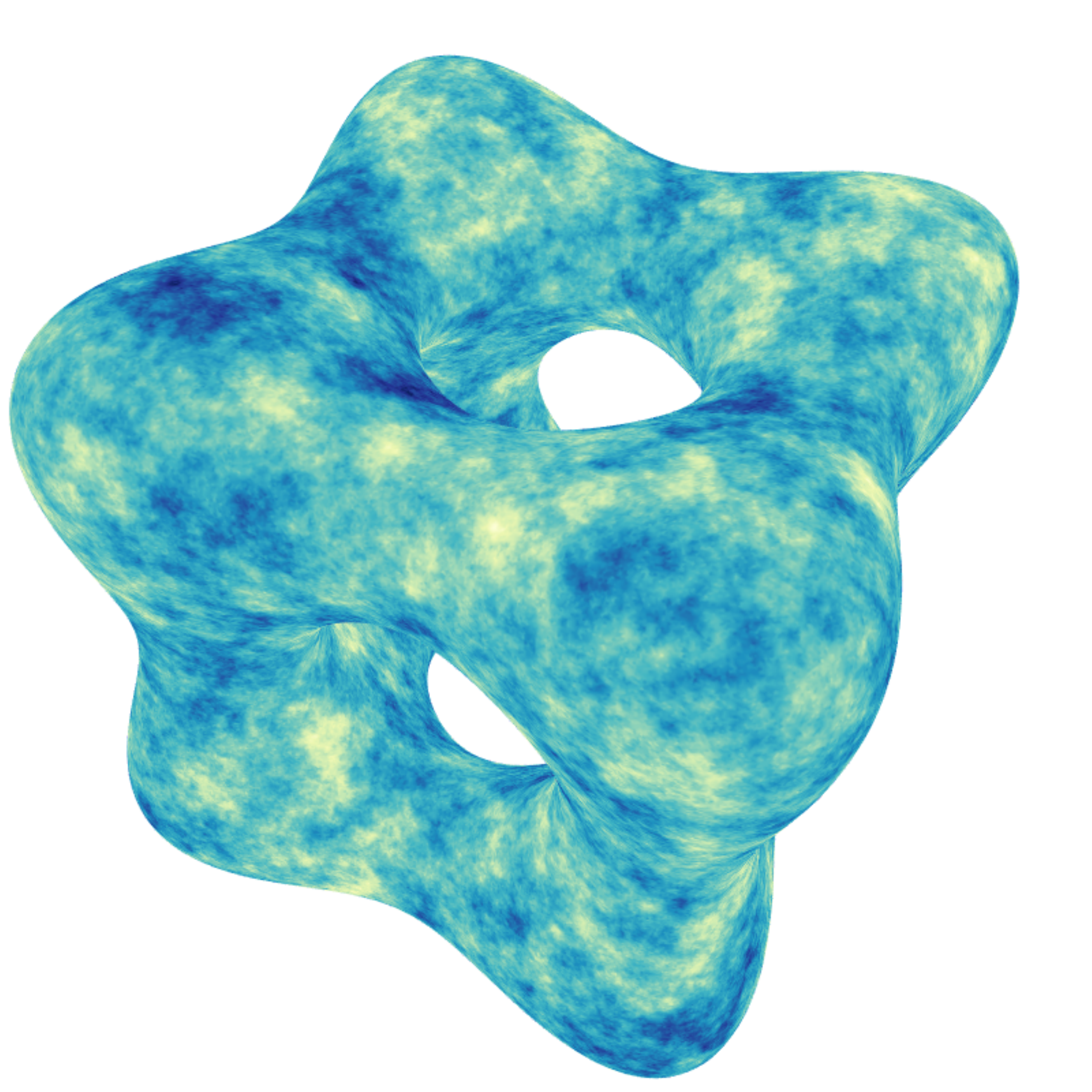}
	\end{subfigure}\hfill
	\begin{subfigure}{0.33\textwidth}
		\centering
		\includegraphics[height=0.18\textheight]{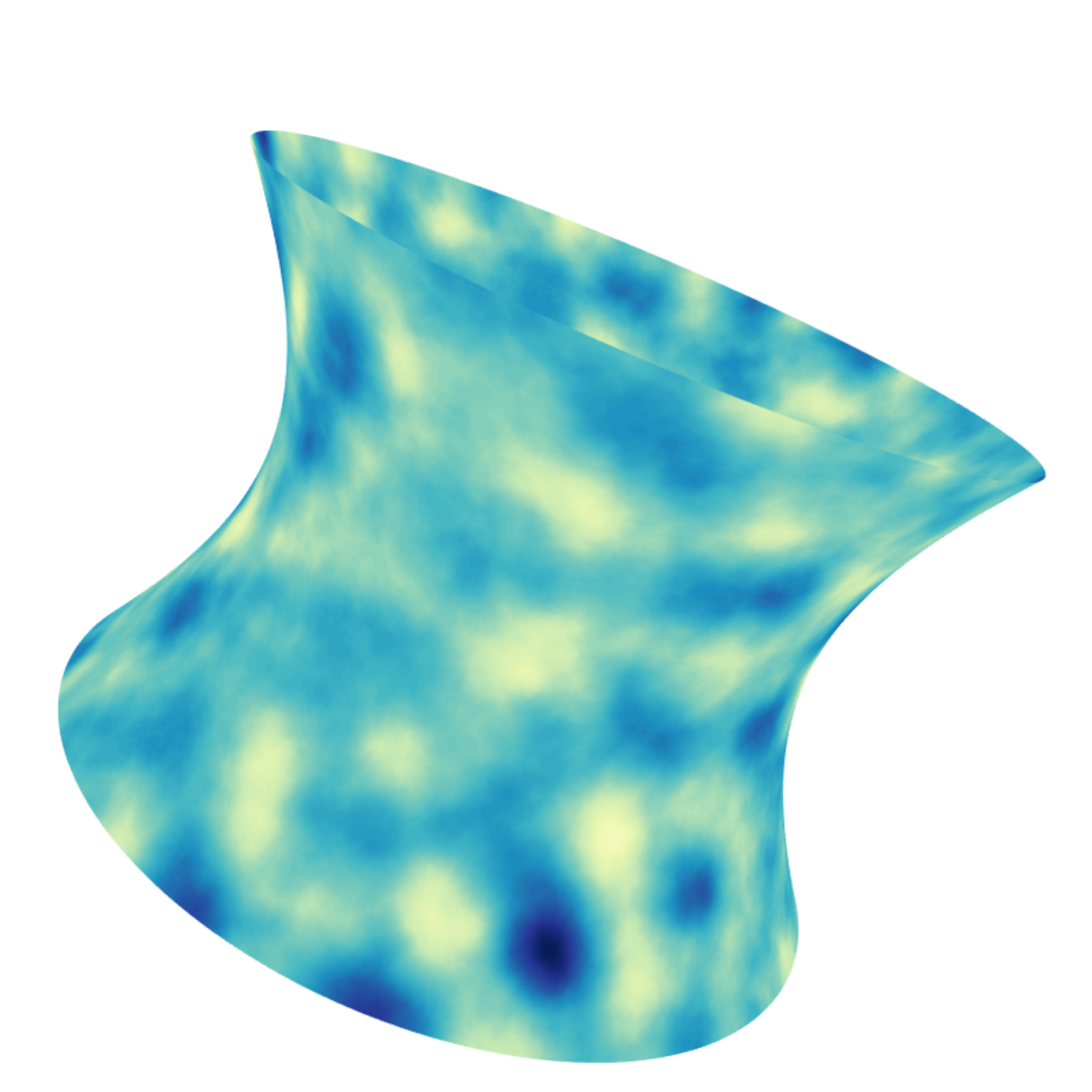}
	\end{subfigure}
	\caption{Simulations of Gaussian random fields on various (compact Riemannian) manifolds.}
	\label{fig:ex_sim}
\end{figure}

Our approach extends previous methods proposed for the numerical approximation of Gaussian random fields defined on manifolds. Several authors worked on the approximation of Gaussian random fields seen as solutions to stochastic partial differential equations (SPDEs), and in particular Whittle--Mat\'ern fields which were popularized by \citet{lindgren2011explicit}. A quadrature approximation allowed them to derive numerical approximations of such fields defined on bounded Euclidean domains \cite{BKK20,bolin2018weak} and even compact metric spaces~\cite{herrmann2020multilevel}. This approach requires to solve multiple (large but sparse) linear systems in order to generate samples of the random fields, and work has been done to find suitable and efficient preconditioners to tackle them~\cite{harbrecht2021multilevel}. In contrast, our approach  does not rely on the fact that the random field is the solution of some SPDE (since we do not require the function of~$-\Delta_{\mathcal{M}}$ to be invertible), but still includes Whittle--Mat\'ern fields as a particular case.  Also, the use of a Chebyshev polynomial approximation allows in some cases to avoid solving any linear system while generating samples.

The idea of using functions of the Laplacian to model Gaussian random fields on manifolds was recently investigated by \citet{borovitskiy2020mat} and \citet{borovitskiy21a}. 
Contrary to~\citet{borovitskiy2020mat}, our approach does not require an explicit approximation of the eigenvalues and eigenfunctions of the Laplace--Beltrami operator. Besides, we propose a convergence analysis, both in mean-square and covariance, of the approximations we propose. This analysis extends to the approximations in~\cite{borovitskiy2020mat}, as they can be seen as a particular instance of our more general framework. Finally, our work provides a theoretical justification for the use of functions of Laplacian matrices to model Gaussian fields on graphs, as proposed in \citet{borovitskiy21a}. Indeed, such matrices arise naturally when examining the discretization of random fields~\cite{pereira2019generalized}.

The outline of this paper is as follows. In \Cref{sec:back}, we present some background material on functional analysis on Riemannian manifolds and the class of Gaussian random fields considered in this work.  \Cref{sec:discr_genrf} is devoted to the Galerkin approximation of these random fields. Then, in \Cref{sec:cheb}, we introduce the Chebyshev polynomial approximation used to numerically compute the weights of the Galerkin-discretized random fields. In \Cref{sec:err} we expose the convergence analysis of the Galerkin and Chebyshev approximations and give the corresponding error estimates, and in \Cref{sec:compl} we present an analysis of the computational complexity and storage required to generate samples of a random field using its Galerkin--Chebyshev approximation. Finally, in \Cref{sec:num}, we confirm error estimates through numerical experiments on the sphere and a hyperboloid.

Throughout the paper, we denote by $\bm I$ the identity matrix and for any $a, b \in \N_0$  we write $\bi a, b\ei = \lbrace a, \dots, b\rbrace$ if $a\le b$, and adopt the convention $\bi a, b\ei = \emptyset$ if $a>b$. The entries of a vector $\bm u\in\R^n$ are denoted by $u_1,\dots,u_n$, and the entries of a matrix $\bm A\in\R^{n\times n}$ are denoted by $A_{ij}$, $1\le i,j\le n$. If $\bm X$ is a Gaussian vector with mean $\mu$ and covariance matrix $\bm\Sigma$, we write $\bm X \sim \mathcal{N}(\bm \mu, \bm\Sigma)$. Finally, for any two functions $f$ and $g$ depending on some argument $x\in\R$, and for $a\in\lbrace 0, +\infty\rbrace$, we write $f(x)=\mathcal{O}(g(x))$ if $f$ is asymptotically bounded by $g$ as $x\rightarrow a$, i.e. if there exists some constant $M_a$ independent of $x$ such that $\vert f(x)\vert \le M_a \vert g(x) \vert$ when $x\rightarrow a$.

\section{Functional analysis background and random fields on manifolds}
\label{sec:back}

\subsection{Laplace--Beltrami operator on a compact Riemannian manifold}

We first introduce a few notions of Riemannian geometry, and refer the interested reader to \cite{berard2006spectral,jost2008riemannian,lablee2015spectral}  and the references therein for a more in-depth introduction on the subject.

Let $(\mathcal{M},g)$ be a compact connected oriented Riemannian manifold of dimension $d\ge 1$, such that  $\mathcal{M}$ has either a smooth boundary $\partial\mathcal{M}$ or no boundary at all ($\partial\mathcal{M}=\emptyset$). 
A function $f: \mathcal{M} \rightarrow \R$ is called \emph{smooth} if for any coordinate patch $(U, \phi)$ (where $U\subset \mathcal{M}$ and $\phi : U \rightarrow \R^d$ defines local coordinates on $U$), the function $f\circ\phi^{-1}$ is a smooth function from $\R^d$ to $\R$. Let then $C^{\infty}(\mathcal{M})$ be the set of smooth functions from $\mathcal{M}$ to $\R$.
The gradient operator $\nabla_{\mathcal{M}}$ acting on functions of $C^{\infty}(\mathcal{M})$ associates to each $f\in C^{\infty}(\mathcal{M})$ the vector field $\nabla_{\mathcal{M}}f$ described in local coordinates $(x^1, \dots, x^d)$ by
\begin{equation*}
\nabla_{\mathcal{M}}f =\sum_{i=1}^d\sum_{j=1}^d g^{ij}\;\frac{\partial(f\circ\phi^{-1})}{\partial x^i}\; \frac{\partial\;}{\partial x^j},
\end{equation*}
where $\phi$ denotes the local chart associated with the coordinates and $(g^{ij})_{1\le i,j\le d}$ is the inverse of the metric tensor $g=(g_{ij})_{1\le i,j\le d}$. Similarly, the Laplace--Beltrami operator $-\Delta_{\mathcal{M}}$ acting on functions of $C^{\infty}(\mathcal{M})$ associates to each $f\in C^{\infty}(\mathcal{M})$ the function $-\Delta_{\mathcal{M}}f$ described  by
\begin{equation*}
-\Delta_{\mathcal{M}}f=-\frac{1}{\sqrt{\vert g\vert}}\, \sum_{i=1}^d\sum_{j=1}^d \frac{\partial\;}{\partial x^i}\bigg(\sqrt{\vert g\vert}\, g^{ij}\,\frac{\partial(f\circ\phi^{-1})}{\partial x^j}\bigg) \veq
\end{equation*}
where $\vert g\vert$ is the determinant of the metric tensor $g$.  Note in particular that both definitions are independent of the choice of local charts and associated local coordinates.

Let $\di v_g$ denote the canonical measure of $(\mathcal{M}, g)$, which is given by
\begin{equation*}
\di v_g=\sqrt{\vert g\vert} \,\di x^1\cdots\di x^d,
\end{equation*}
where $\di x^1\cdots\di x^d$ denotes the standard Lebesgue measure on~$\R^d$. We denote by $H=L^2(\mathcal{M}, g)$ the space of square-integrable functions on $(\mathcal{M},g)$, which is defined as
\begin{equation*}
H=L^2(\mathcal{M}, g)=\big\lbrace f : \mathcal{M} \rightarrow \R \text{ measurable  : } \int_{\mathcal{M}} \vert f\vert^2 \;\di v_g < +\infty\big\rbrace.
\end{equation*}
In particular, $H$ is a Hilbert space when equipped with the inner product $(\cdot, \cdot)_0$ defined by
\begin{equation*}
(f_1, f_2)_0=\int_{\mathcal{M}} f_1f_2 \;\di v_g, \quad f_1, f_2 \in H,
\end{equation*}
and we denote by $\Vert\cdot\Vert_0$ the norm associated with this inner product.


Consider the eigenvalue problem 
\begin{equation*}
-\Delta_{\mathcal{M}} \phi = \lambda \phi, \quad  \phi \in C^{\infty}(\mathcal{M})\backslash\lbrace 0\rbrace, \quad \lambda \in\R
\end{equation*}
with Dirichlet or (homogeneous) Neumann boundary conditions whenever $\partial\mathcal{M}\neq \emptyset$.
A standard result of spectral theory  \cite[Theorem 4.3.1]{lablee2015spectral} states that this problem admits solutions in the form of a set of eigenpairs $(\lambda_k, e_k)_{k\in\N}$, where $\lambda_k\ge 0$ and such that each eigenvalue has a finite multiplicity, the eigenspaces corresponding to distinct eigenvalues are $H$-orthogonal, and the direct sum of the eigenspaces is dense in $H$.
Hence this theorem provides a decomposition of any function $f\in H$ into an orthonormal basis $\lbrace e_k \rbrace_{k\in\N}$  of eigenfunctions of $-\Delta_{\mathcal{M}}$, as

\begin{equation*}
f = \sum_{k\in\N} ( e_k, f)_0 \,e_k \veq
\end{equation*}
where the equality is understood in the $L^2$-sense. 

Without loss of generality, we assume in the remainder of this paper that the eigenpairs of $-\Delta_{\mathcal{M}}$ are ordered so that $0\le \lambda_{1} \le \lambda_{2}\le \cdots$. In particular we have $\lambda_{1}=0$ whenever $\partial\mathcal{M}=\emptyset$ or Neumann boundary conditions are considered, and $\lambda_{1}>0$ when Dirichlet boundary conditions are considered \cite[Proposition 4.5.6]{lablee2015spectral}. Hence, in this work, the multiplicity $M_0$ of the eigenvalue $0$ satisfies $M_0 \in \lbrace 0, 1\rbrace$. The following can be stated about the growth rate of the eigenvalues.
\begin{prop}[Weyl's asymptotic law]
	\label{prop:weyl}
	For  $\alpha =2/d$, there exist constants $c_\lambda>0$ and $C_\lambda>0$ such that all non-negative eigenvalues $\lbrace \lambda_j\rbrace_{j\in\N}$ satisfy
	$$c_\lambda j^\alpha \le \lambda_j \le C_\lambda j^\alpha.$$
\end{prop}
This property is a direct consequence of Weyl's asymptotic formula which holds for connected compact Riemannian manifolds of dimension $d$ and states that the constants $c_\lambda$ and $C_\lambda$ depend on $d$ and on the volume of the manifold \cite[Theorem 7.6.4]{lablee2015spectral}.

\subsection{Function spaces on a compact Riemannian manifold}\label{sec:funcsp}

The Sobolev space $H^1$ is defined as the completion of $C^{\infty}(\mathcal{M})$ with respect to the norm $\Vert \cdot\Vert_{H^1}$ defined by
\begin{equation*}
\Vert f \Vert_{H^1}^2 = \Vert f\Vert_0^2 + \Vert \nabla_{\mathcal{M}} f\Vert_0^2, \quad f\in C^{\infty}(\mathcal{M}).
\end{equation*}
This space is a Hilbert space when equipped with the inner product $(\cdot, \cdot)_{H^1}$ defined by
\begin{equation*}
(f_1, f_2)_{H^1}=(f_1, f_2)_0 + (\nabla_{\mathcal{M}} f_1, \nabla_{\mathcal{M}} f_2)_0, \quad f_1, f_2 \in H^1.
\end{equation*}
In particular, the definition of the gradient operator is here extended to functions of $H^1$ using a density argument. More generally, Sobolev spaces of fractional order $H^\sigma$, $\sigma> 0$, can be defined on compact Riemannian manifolds by stating that $f\in H^\sigma$ when, for any coordinate patch $(U, \phi)$, and any function $\psi$ with compact support in $U$, the function $(f\psi)\circ \phi^{-1}$ belongs to the Sobolev space $H^\sigma(\R^d)$ as usually defined on $\R^d$ \cite[Chapter 4, Section 3]{taylor1996partial}. For $\sigma=1$, this last characterization coincides with our used definition of~$H^1$.
Finally, let $\sigma\ge 0$ and let $\mathcal{F} \subset H$ be the space of finite linear combinations of the eigenfunctions $\lbrace e_k\rbrace_{k \in\N}$ of $-\Delta_{\mathcal{M}}$.
Following the definition of spaces of generalized functions on manifolds introduced by Taylor \cite[Chapter 5, Section A]{taylor1996partial}, let $\dot{H}^\sigma$ be the completion of $\mathcal{F}$ under the norm $\Vert \cdot\Vert_{\sigma}$ defined by
\begin{equation*}
\Vert f\Vert_{\sigma}^2
=\sum_{k\in\bi 1, M_0\ei}  \vert ( f, e_k)_{0} \vert^2 + \sum_{k>M_0} \lambda_k^\sigma \vert ( f, e_k)_{0}\vert^2,
\end{equation*}
where by convention the first sum vanishes if $M_0=0$. In particular, we have $\dot{H}^0=H$ and more generally, $\dot{H}^\sigma$ is a Hilbert space when equipped with the inner product $(\cdot, \cdot)_\sigma$ defined by
\begin{equation}
(f_1, f_2)_\sigma = \sum_{k\in\bi 1, M_0\ei} ( f_1, e_k)_{0}( f_2, e_k)_{0} + \sum_{k>M_0} \lambda_k^\sigma ( f_1, e_k)_{0}( f_2, e_k)_{0} , \quad f_1, f_2\in \dot{H}^\sigma.
\label{eq:ip}
\end{equation} 

\begin{rem}
	When manifolds without boundary are considered, the definition of $\dot{H}^\sigma$ given above is equivalent to the definition of the fractional Sobolev space of order~$\sigma$ through Bessel potentials (used for instance by \citet{strichartz1983analysis} or \citet{herrmann2018numerical}). Indeed, recall that the latter is defined as the subspace of $H$ composed of functions $f\in H$ satisfying $ \Vert f \Vert_\sigma' < +\infty$, where $ \Vert \cdot \Vert_\sigma'$ is the norm defined by
	\begin{equation*}
	\Vert f\Vert_\sigma'=\bigg(\sum_{k\in\mathbb{N}} (1+\lambda_k)^\sigma \vert ( f, e_k)_{0}\vert^2\bigg)^{1/2}, \quad f\in \dot{H}^\sigma.
	\end{equation*}
	Equivalence follows from the equivalence of the norms $\Vert \cdot \Vert_\sigma$ and $\Vert \cdot\Vert_\sigma'$: we have
	\begin{equation*}
	\Vert f\Vert_\sigma \le \Vert f \Vert_\sigma' \le \big(1+\lambda_{M_0+1}^{-1}\big)^\sigma \Vert f\Vert_\sigma, \quad f\in \dot{H}^\sigma.
	\end{equation*}
	When manifolds with boundary are considered, and $\sigma>0$, $\dot{H}^\sigma$ can be seen as a subspace of a fractional Sobolev space composed of functions satisfying the same boundary conditions as the ones considered for the eigenvalue problem of the Laplace--Beltrami operator \cite[Chapter 5, Section A]{taylor1996partial}.
\end{rem}								

For $\sigma<0$, we define $\dot{H}^\sigma$ to be the dual space of $\dot{H}^{-\sigma}$: these spaces are Hilbert spaces when endowed with the inner product~\eqref{eq:ip}, and their elements are seen as distributions \cite{strichartz1983analysis}.


\subsection{Functions of the Laplacian}
\label{sec:func_lap}

We now introduce a class of operators acting on $H$, called functions of the Laplacian. These operators are classically used to express solutions of some differential equations and to prove Weyl's asymptotic formula \cite{bouclet2012introduction}. To define functions of the Laplacian, we first introduce the notion of \emph{\psd}.

\begin{dfn}\label{def:psd}
	A \emph{\psd}  is a function $\gamma : [0,+\infty) \rightarrow \mathbb{R}$ with the following properties. First, there exists some $\nu\in\N$ for which $\gamma$ is  $\nu$ times differentiable, with continuous derivatives up to order $(\nu-1)$ and a derivative of order $\nu$ of bounded variation.
	Second, $\gamma(\lambda) \rightarrow 0$ as $\lambda \rightarrow \infty$. And finally, 
	there exist constants $L_\gamma, C_{\gamma}', \beta >0$ such that for all $\lambda \ge L_\gamma$, the first derivative $\gamma'$ of $\gamma$ satisfies
\begin{equation*}
\vert \gamma'(\lambda)\vert \le C_{\gamma}' \vert \lambda\vert ^{-(1+\beta)}.
\label{eq:assum_gamma_b}
\end{equation*}
Note in particular that these last two conditions imply that there exists $C_{\gamma}>0$ such that 
\begin{equation*}
	\left\vert \gamma(\lambda) \right\vert \le C_\gamma \vert\lambda\vert^{-\beta}, \quad \lambda\ge L_\gamma.
\end{equation*}
\end{dfn}

In particular,  the \psd considered in this work should satisfy the relation given in the next assumption.

\begin{assumb}{}
	The \psd considered in this work satisfy the relation 
		\begin{equation*}
		2\alpha \beta -1 > 0
	\end{equation*}
	where $\beta>0$ is defined in \Cref{def:psd} and $\alpha>0$ is defined in \Cref{prop:weyl}.
	\label{assum:alphabeta}
\end{assumb}

This assumption allows us to define the notion of functions of Laplacian as a endomorphism of $H$. Indeed, 
given a \psd $\gamma$ satisfying \Cref{assum:alphabeta}, we define the \textit{function of the Laplacian} $\gamma(-\Delta_{\mathcal{M}})$ associated with $\gamma$ as the operator $\gamma(-\Delta_{\mathcal{M}}) : H \rightarrow H$ given by:
\begin{equation*}
\gamma(-\Delta_{\mathcal{M}})f=\sum_{k\in\N}\gamma(\lambda_k)( f, e_k)_0 \,e_k, \quad f\in H.
\label{eq:def_func_lap}
\end{equation*}
The next proposition extends the domain of this operator.
\begin{prop}
	Let \Cref{assum:alphabeta} be satisfied. For any $\sigma\in\R$,  the function of the Laplacian $\gamma(-\Delta_{\mathcal{M}})$ can be extended to an operator (also denoted $\gamma(-\Delta_{\mathcal{M}})$ with a slight abuse of notation)
	\begin{equation*}
	\gamma(-\Delta_{\mathcal{M}}) : \dot{H}^{\sigma} \rightarrow \dot{H}^{\sigma+2\beta},
	\end{equation*}
where $\alpha>0$ and $\beta>0$ are defined respectively in \Cref{prop:weyl} and \Cref{def:psd}.
	\label{prop:dom_def_flap}
\end{prop}

\begin{proof}
	Let $\sigma\in\R$ and $f\in \dot{H}^{\sigma}$. 
	
	\begin{equation*}
	\begin{aligned}
	\Vert \gamma(-\Delta_{\mathcal{M}})f \Vert_{\sigma+2\beta}^2
	&=\sum_{k\in\bi 1, M_0\ei}  \vert \gamma(\lambda_k)( f, e_k)_0 \vert^2 +\sum_{k>M_0} \lambda_k^{\sigma+2\beta} \vert \gamma(\lambda_k)( f, e_k)_0 \vert^2 \\
	&= \vert \gamma(0) \vert^2\sum_{k\in\bi 1, M_0\ei}  \vert ( f, e_k)_0 \vert^2 +\sum_{k>M_0} \lambda_k^{\sigma}\vert\lambda_k^{\beta}\gamma(\lambda_k)\vert^2 \vert ( f, e_k)_0 \vert^2.
	\end{aligned}
	\end{equation*}
	
	Following \Cref{def:psd}, and since $\lambda_{k} \rightarrow +\infty$ as $k\rightarrow +\infty$, we set
	\begin{equation}
	R_\gamma=\max\lbrace \vert\gamma(0)\vert, C_\gamma,\vert\lambda_{M_0+1}^{\beta}\gamma(\lambda_{M_0+1})\vert, \dots, \vert\lambda_{K_\gamma}^{\beta}\gamma(\lambda_{K_\gamma})\vert \rbrace,
	\label{eq:K_gamma}
	\end{equation}
	where $K_\gamma = \sup\lbrace k \in \mathbb{N} : \lambda_{k} < L_{\gamma}\rbrace$.
	We then obtain that $\gamma(-\Delta_{\mathcal{M}})f\in \dot{H}^{\sigma+2\beta}$ since
	\begin{equation*}
	\pushQED{\qed} 
	\begin{aligned}
	\Vert \gamma(-\Delta_{\mathcal{M}})f \Vert_{\sigma+2\beta}^2
	&\le R_\gamma^2\bigg(\sum_{k\in\bi 1, M_0\ei}  \vert ( f, e_k)_0 \vert^2 +\sum_{k>M_0} \lambda_k^{\sigma}\vert ( f, e_k)_0 \vert^2\bigg)
	= R_\gamma^2\Vert f\Vert_{\sigma}^2<+\infty.
	\end{aligned}\qedhere
	\popQED
	\end{equation*}
\end{proof}

Note in particular that \Cref{prop:dom_def_flap} implies that, for all $\sigma\ge -2\beta$, $\gamma(-\Delta_{\mathcal{M}})$ maps $\dot{H}^\sigma$ into (a subspace of)~$H$. 


\subsection{Random fields on a Riemannian manifold}
\label{sec:fem_grf}

Let us start by introducing some notation. Let $(\Omega, \mathcal{A}, \mathbb{P})$ be a complete probability space.  Let $Q$ denote some arbitrary Hilbert space (with inner product $(\cdot, \cdot)_Q$ and associated norm $\Vert \cdot\Vert_{Q}$). We denote by $L^2(\Omega; Q)$ the set of all $Q$-valued random variables defined on $(\Omega, \mathcal{A},\mathbb{P})$  satisfying, for any $\mathcal{Z} \in L^2(\Omega; Q)$, $\e[\mathcal{Z}]=0$ and  $\e[\Vert \mathcal{Z}\Vert_{Q}^2]<+\infty$. In particular, this implies that any $\mathcal{Z}\in L^2(\Omega; Q)$ is almost surely in $Q$. Finally, note that $L^2(\Omega; Q)$ is a Hilbert space when equipped with the inner product $( \cdot, \cdot)_{L^2(\Omega; Q)}$ (and associated norm $\Vert \cdot\Vert_{L^2(\Omega; Q)}$) defined by
\begin{equation*}
(\mathcal{Z}, \mathcal{Z}')_{L^2(\Omega; Q)}=\e\left[(\mathcal{Z}, \mathcal{Z}')_Q\right], \quad \mathcal{Z}, \mathcal{Z}'\in L^2(\Omega; Q).
\end{equation*}


We now define the notion of Gaussian white noise on the manifold $\mathcal{M}$.
Let $\lbrace W_j \rbrace_{j\in\mathbb{N}}$ be a sequence of independent, standard Gaussian random variables.
The linear functional $\mathcal{W}$ defined over $H$ by
\begin{equation}
\mathcal{W}: \varphi \in H \mapsto \langle \mathcal{W}, \varphi\rangle=\sum_{j\in\mathbb{N}}  W_j(\varphi, e_j)_{0}
\label{eq:def_wn}
\end{equation}
is called \textit{Gaussian white noise} on $\mathcal{M}$. Note that for any $\varphi\in H$, the series $\langle \mathcal{W}, \varphi\rangle$ converges in quadratic mean since $\e\left[\langle \mathcal{W}, \varphi\rangle\right]=0$ and by independence of the variables $\lbrace W_k\rbrace_{k\in\N}$, 
\begin{equation*}
\e\big[\vert \langle \mathcal{W}, \varphi\rangle\vert^2\big]
=\e\bigg[ \sum_{j\in\mathbb{N}}\sum_{k\in\mathbb{N}}   W_j(\varphi, e_j)_{0}{W_k(\varphi, e_k)_{0}}\bigg]
=\sum_{j\in\mathbb{N}} \vert(\varphi, e_j)_{0}\vert^2
=\Vert\varphi\Vert_0^2 <+\infty \peq
\end{equation*}
In particular, $\mathcal{W}$ satisfies, for any $\varphi \in H$, $\e\left[\langle \mathcal{W},\varphi\rangle\right]=0$, and for any $\varphi_1, \varphi_2 \in H$,
\begin{equation*}
\cov\left[\langle \mathcal{W},\varphi_1\rangle, \langle \mathcal{W},\varphi_2\rangle\right]=(\varphi_1, \varphi_2)_0 \peq
\label{eq:cov_func_wn}
\end{equation*}
The next proposition details the domain of definition and regularity of $\mathcal{W}$.

\begin{prop}
	For any $\epsilon >0$, $\mathcal{W}\in L^2(\Omega; \dot{H}^{-(\alpha^{-1} + \epsilon)})$, where $\alpha>0$ is given in  \Cref{prop:weyl}.
	\label{prop:reg_w}
\end{prop}

\begin{proof}
	Let $\epsilon>0$ and $N\in\mathbb{N}$. Consider the truncated white noise $\mathcal{W}_N$ defined by
	\begin{equation*}
	\mathcal{W}_N : \varphi \in H \mapsto \langle \mathcal{W}_N, \varphi\rangle=\sum_{j=1}^N  W_j(\varphi, e_j)_{0}.
	\end{equation*}
	By definition of $M_0$,
	\begin{equation*}
	\begin{aligned}
	\e\big[\Vert \mathcal{W}_N\Vert_{-(\alpha^{-1}+\epsilon)}^2\big]
	&=\e\bigg[ \sum_{k\in\bi 1, M_0\ei}\vert W_k\vert^2 +\sum_{k=M_0+1}^N \lambda_k^{-(\alpha^{-1}+\epsilon)}  \vert W_k \vert^2\bigg] =M_0 +\sum_{k=M_0+1}^N \lambda_k^{-(\alpha^{-1}+\epsilon)},
	\end{aligned}
	\end{equation*}
	which gives, using \Cref{prop:weyl},
	\begin{equation*}
	\begin{aligned}
	\e\big[\Vert \mathcal{W}_N\Vert_{-(\alpha^{-1}+\epsilon)}^2\big]
	\le M_0 +c_{\lambda}^{-(\alpha^{-1}+\epsilon)} \sum_{k=1}^N k^{-(1+\epsilon\alpha)} \le M_0 +c_{\lambda}^{-(\alpha^{-1}+\epsilon)}\zeta(1+\epsilon\alpha),
	\end{aligned}
	\end{equation*}
	where $\zeta$ denotes the Riemann zeta function satisfying $\zeta(1+\epsilon\alpha)<\infty$ since $\epsilon\alpha >0$.
	Taking the limit $N\rightarrow \infty$ implies that $\e[\Vert \mathcal{W}\Vert_{-(\alpha^{-1}+\epsilon)}^2]<\infty$, which proves the claim.\qed
\end{proof}


We now introduce a class of random fields defined using the white noise $\mathcal{W}$ and functions of the Laplacian. Let $\gamma$ be a \psd satisfying \Cref{assum:alphabeta} be satisfied and let $\mathcal{Z}$ be the random field defined by
\begin{equation}
\mathcal{Z}=\gamma(-\Delta_{\mathcal{M}})\mathcal{W} \peq
\label{eq:def_z}
\end{equation}
By Propositions~\ref{prop:dom_def_flap} and~\ref{prop:reg_w}, for any $\epsilon>0$, $\mathcal{Z}$ is (a.s.) an element of $\dot{H}^{2\beta-(\alpha^{-1}+\epsilon)}$. 
The next proposition links $\mathcal{Z}$ to $H$-valued random variables. 

\begin{prop}\label{prop:def_gegf_sum_eigf}
	Let $\gamma$ be a \psd satisfying \Cref{assum:alphabeta} and let $\mathcal{Z}$ be defined by~\eqref{eq:def_z}. Then, $\mathcal{Z}\in L^2(\Omega; H)$ and $\mathcal{Z}$ can be decomposed as
	\begin{equation*}
	\mathcal{Z}=\sum\limits_{k\in\mathbb{N}} W_k\gamma(\lambda_k) e_k \veq
	\end{equation*} 
	where the weights $\lbrace W_j \rbrace_{j\in\mathbb{N}}$ define a white noise as in~\eqref{eq:def_wn}.
\end{prop}

\begin{proof}
	Since \Cref{assum:alphabeta} is satisfied, Propositions \ref{prop:dom_def_flap} and \ref{prop:reg_w} give that $\mathcal{Z}$ is in $H$ (almost surely). Recall that, by definition of functions of the Laplacian,
	\begin{equation*}
	\mathcal{Z}=\gamma(-\Delta_{\mathcal{M}})\mathcal{W}=\sum_{k\in\N}\gamma(\lambda_k)(\mathcal{W},  e_k)_0\, e_k=\sum_{k\in\N}\gamma(\lambda_k)W_k \,e_k.
	\end{equation*}
	By linearity, we then have
	\begin{equation*}
	\e[\mathcal{Z}]=\gamma(-\Delta_{\mathcal{M}})\e[\mathcal{W}]=0,
	\end{equation*}
	and following \Cref{def:psd} and \Cref{prop:weyl},
	\begin{equation*}
	\begin{aligned}
	\e[\Vert\mathcal{Z}\Vert_0^2]
	=\sum_{k\in\N}\vert\gamma(\lambda_k)\vert^2 
	&\le R_\gamma^2\bigg(M_0+ \sum_{k>M_0} \lambda_k^{-2\beta}\bigg)
	\le R_\gamma^2\big(M_0 + c_\lambda^{-2\beta}\zeta(2\beta\alpha)\big),
	\end{aligned}
	\end{equation*}
	where $\zeta(2\beta\alpha) <\infty$ since $2\alpha\beta >1$, and $R_\gamma$ is defined in~\eqref{eq:K_gamma}.
	Hence $\e[\Vert\mathcal{Z}\Vert_0^2]<\infty$ and therefore $\mathcal{Z}\in L^2(\Omega; H)$.\qed
\end{proof}

The class of Gaussian random fields described in this section can be seen as an extension to arbitrary compact connected oriented Riemannian manifolds of the class of isotropic random fields on the sphere described in \cite{lang2015isotropic}. In this last case, the eigenfunctions $\lbrace e_k\rbrace_{k\in\N}$  of the Laplace--Beltrami operator are the spherical harmonics, and the \psd $\gamma$ defines the angular power spectrum of the field. In this sense, the decomposition introduced in \Cref{prop:def_gegf_sum_eigf} can be seen as the \kl expansion of a Gaussian random field on a compact connected oriented Riemannian manifold.

In the particular case where the \psd $\gamma$ takes the form
\begin{equation}
\gamma(\lambda)=\vert \kappa^2 + \lambda\vert^{-\beta}, \quad \lambda \ge 0,
\label{eq:gamma_beta}
\end{equation}
for some parameters $\kappa>0$ and $\beta>1/(2\alpha)=d/4$, the resulting field $\mathcal{Z}$ is a solution to the fractional elliptic SPDE 
\begin{equation}
(\kappa^2-\Delta_{\mathcal{M}})^{\beta}\mathcal{Z}=\mathcal{W}.
\label{eq:spde}
\end{equation}
As such, $\mathcal{Z}$ is an instance of a Whittle--Mat\'ern random field on a manifold, as introduced in \cite{lindgren2011explicit} for compact Riemannian manifolds. This class of random fields was studied in \cite{jansson2021surface} for the particular case where the manifold is a sphere, and in \cite{herrmann2020multilevel,harbrecht2021multilevel} for compact Riemannian manifolds.


More generally, the random fields defined by~\eqref{eq:def_z}  are particular instances of regular zero-mean generalized Gaussian fields (GeGF) as defined in \citep[Section 3.2.1]{lototsky2017stochastic}. 
To a field $\mathcal{Z}$ defined by~\eqref{eq:def_z}, we can associate the continuous linear functional $f\in H \mapsto (\mathcal{Z}, f)_0$, which corresponds to a GeGF  with a covariance operator $K : H \rightarrow H$ given by $K=\gamma^2(-\Delta_{\mathcal{M}})$ (where by definition the covariance operator is defined as $\mathbb{E}[(\mathcal{Z}, f)_0(\mathcal{Z}, f')_0]=(K(f),f')_0$). The fact that this GeGF is regular stems directly from the fact that, under the assumptions used in \Cref{prop:def_gegf_sum_eigf}, the operator $\gamma^2(-\Delta_{\mathcal{M}})$ is nuclear. Conversely, since $-\Delta_{\mathcal{M}}$ and  $\gamma^2(-\Delta_{\mathcal{M}})$ have the same eigenfunctions, and since the function $\gamma^2$ maps the eigenvalues of  $-\Delta_{\mathcal{M}}$ to those of  $\gamma^2(-\Delta_{\mathcal{M}})$, any regular GeGF with covariance operator $\gamma^2(-\Delta_{\mathcal{M}})$ can be decomposed as in \Cref{prop:def_gegf_sum_eigf} (cf.  \citep[Theorem 3.2.15]{lototsky2017stochastic} and its proof).

\section{Discretization of Gaussian random fields}
\label{sec:discr_genrf}

We now aim at computing numerical approximations of the random fields $\mathcal{Z}$ defined in~\eqref{eq:def_z} using a discretization of the Laplace--Beltrami operator. The discretization we propose is based on a Galerkin approximation, and can be seen as an extension of the approach in \cite{BKK20}. It leads to an approximation by a weighted sum of  basis functions defined on the manifold.

For $n\ge 1$, let $\lbrace \psi_k \rbrace_{1\le k \le n}$ be a family of linearly independent functions of~$\dot{H}^1$ and denote by $V_n \subset \dot{H}^1$ its linear span.
In particular, $V_n$ is a $n$-dimensional subspace of $\dot{H}^1$, and we assume that the constant functions are in $V_n$.
Examples that are included in our framework are spectral methods, where $V_n$ is spanned by finitely many eigenfunctions of~$-\Delta_{\mathcal{M}}$, boundary element methods \cite{SSch11}, and with an extra approximation step surface finite elements \cite{dziuk2013finite}.

\subsection{Galerkin discretization of the Laplace--Beltrami operator}
\label{sec:rg_discr}

We first introduce a discretization $-\Delta_n$ of the Laplace--Beltrami operator over $V_n$ by a \emph{Galerkin approximation} \cite[Chapter 4]{axelsson2001finite}.
For any $\varphi\in V_n$, we set $-\Delta_n\varphi$ to be the element of $V_n$ satisfying for all $v\in V_n$
\begin{equation*}
(-\Delta_n \varphi, v)_0
= \left(\nabla_{\mathcal{M}}\varphi, \nabla_{\mathcal{M}}v\right)_0,
\end{equation*}
which uniquely defines $-\Delta_n: V_n \rightarrow V_n$. In particular, if $\lbrace f_{k} \rbrace_{1\le k \le  n}$ denotes any orthonormal basis of $(V_n, \Vert \cdot\Vert_0)$, this operator satisfies
\begin{equation}
-\Delta_n\varphi=\sum\limits_{k=1}^n \left( \nabla_{\mathcal{M}}f_k, \nabla_{\mathcal{M}} \varphi\right)_0f_{k}, \quad \varphi \in V_n.
\label{eq:proj_lap}
\end{equation}

Let $\bm C$ and $\bm R$ be the matrices called (in the context of finite element methods) \emph{mass matrix} and \emph{stiffness matrix} respectively, and defined by
\begin{equation}
\begin{aligned}
\bm C &= \left[( \psi_k, \psi_l)_0\right]_{1\le k,l\le n}, \quad
\bm R & =\left[(\nabla_{\mathcal{M}}\psi_k, \nabla_{\mathcal{M}}\psi_l)_0\right]_{1\le k,l\le n} \peq
\end{aligned}
\label{eq:def_CG}
\end{equation}
As defined, $\bm C$ is a symmetric positive definite  matrix and $\bm R$ is a symmetric positive semi-definite matrix (cf. \Cref{prop:prop_CG} of the \sm). Consequently, the \emph{generalized eigenvalue problem} (GEP) defined by the matrix pencil $(\bm R, \bm C)$, which consists in finding all so-called eigenvalues $\lambda \in\R$ and eigenvectors $\bm w\in\R^n \backslash \lbrace\bm 0\rbrace$ such that
\begin{equation*}
\bm R \bm w = \lambda \bm C \bm w \veq
\end{equation*}
admits a solution consisting of $n$ nonnegative eigenvalues and $n$ eigenvectors mutually orthogonal with respect to the inner product $(\cdot, \cdot)_{\bm C}$ (and  norm $\Vert\cdot\Vert_{\bm C}$) defined by (see \cite[Theorem 15.3.3]{parlett1998symmetric}).
\begin{equation*}
(\bm x, \bm y)_{\bm C}
= \bm y^T \bm C \bm x, \quad \bm x, \bm y\in\R^n \peq
\end{equation*}
We observe further that since $\bm C$ is symmetric and positive definite, $\sqrt{\bm C}\in\R^{n\times n}$ satisfying $\sqrt{\bm C}(\sqrt{\bm C})^T=\bm C$ exists and is invertible. Therefore denoting by $\|\cdot\|_2$ the Euclidean norm, we obtain $\|\cdot\|_{\bm C} = \|(\sqrt{\bm C})^T\cdot\|_2$ and an isometry between $(\R^n,\|\cdot\|_{\bm C})$ and $(\R^n,\|\cdot\|_2)$ via the linear bijection $F : \R^n \rightarrow \R^n$ defined by $F(\bm x) = (\sqrt{\bm C})^{T}\bm x$.

The next result links the GEP to the operator $-\Delta_n$, and is proven in Appendix~\ref{appen:diag_lapl_discrg}.

\begin{theorem}
	The operator $-\Delta_n$ is diagonalizable and its eigenvalues are those of the GEP defined by the matrix pencil $(\bm R, \bm C)$.
	In particular, $E_0 : \mathbb{R}^n \rightarrow V_n$, defined by
	\begin{equation*}
	E_0(\bm u) = \sum\limits_{k=1}^n u_k\psi_k, 
	\quad \bm u\in\mathbb{R}^n \veq
	\end{equation*}
	is an isomorphism that maps the eigenvectors of $(\bm R, \bm C)$ to  eigenfunctions of $-\Delta_n$, and an isometry between $(\R^n,\|\cdot\|_C)$ and $(V_n, \Vert \cdot\Vert_0)$.
	\label{prop:diag_lapl_discrg}
\end{theorem}

We continue with a corollary that will be useful later on.

\begin{corol}
	The eigenvalues of $-\Delta_n$ are those of the matrix 
	\begin{equation*}
	\bm S = \big(\sqrt{\bm C}\big)^{-1}\bm R\big(\sqrt{\bm C}\big)^{-T} \veq
	\label{eq:def_S}
	\end{equation*}
	and the mapping $E: \mathbb{R}^n \rightarrow V_n$, defined by
	\begin{equation*}
	E(\bm v) = \sum\limits_{k=1}^n \left[ \big(\sqrt{\bm C}\big)^{-T}\bm v\right]_k\psi_k, 
	\quad \bm v\in\mathbb{R}^n \veq
	\label{eq:def_isom_eig}
	\end{equation*}
	is an isomorphism that maps the eigenvectors of $\bm S$ to the eigenfunctions of $-\Delta_n$ and an isometry between $(\R^n,\Vert\cdot\Vert_{2})$ and $(V_n,\Vert \cdot\Vert_0)$.
	\label{prop:diag_lapl_discr}
	
\end{corol}

\begin{proof}
	Note first that $\bm S$ is well-defined and symmetric positive semi-definite by the properties of~$\bm C$ and recall the bijection~$F$ given by $F(\bm x) = (\sqrt{\bm C})^{T}\bm x$.
	%
	Let $(\lambda, \bm w)$ be an eigenpair of $(\bm R, \bm C)$ and set $\bm v=F^{-1}(\bm w)$, then \begin{equation*}
	\bm S \bm v = \big(\sqrt{\bm C}\big)^{-1} \bm R \bm w
	= \lambda  \big(\sqrt{\bm C}\big)^{-1}\bm C  \bm w
	= \lambda  \big(\sqrt{\bm C}\big)^{T} \bm w=\lambda \bm v,
	\end{equation*}
	and therefore $(\lambda, \bm v)$ is an eigenpair of $\bm S$. Hence $F$ maps the eigenvectors of $(\bm R, \bm C)$ to those of $\bm S$, and conversely $F^{-1}$ maps the eigenvectors of $\bm S$ to those of $(\bm R, \bm C)$. Noting that $E=E_0\circ F^{-1}$ and applying \Cref{prop:diag_lapl_discrg} concludes the proof.\qed
\end{proof}

We denote by $\lbrace \lambda_{k}^{(n)}\rbrace_{1\le k\le n} \subset \mathbb{R}_+$ the eigenvalues of the matrix $\bm S$ (cf. \Cref{prop:diag_lapl_discr}), ordered in non-decreasing order. Let $\lbrace \bm v_{k}\rbrace_{1\le k\le n}\subset \R^n$ be a set of eigenvectors associated with these eigenvalues, and chosen to form an orthonormal basis of $\R^n$. Hence, if $\bm V$ is the matrix whose columns are $( \bm v_1, \dots, \bm v_n)$, we have $\bm V^T\bm V=\bm V\bm V^T=\bm I$ and
\begin{equation*}
\bm S =\bm V 
\Diag(\lambda_{1}^{(n)}, \dots,\lambda_{n}^{(n)})
\bm V^T \veq
\label{eq:s_diag}
\end{equation*}
where $\Diag(\lambda_{1}^{(n)}, \dots,\lambda_{n}^{(n)})$ denotes the diagonal matrix whose entries are $\lambda_{1}^{(n)}, \dots,\lambda_{n}^{(n)}$.
Given that $E$ defined in \Cref{prop:diag_lapl_discr} is a linear isometry, it maps orthonormal sequences in $(\R^n,\Vert\cdot\Vert_{2})$ to orthonormal sequences in $(V_n,\Vert \cdot\Vert_0)$. Hence, the set  $\lbrace e_{k}^{(n)}\rbrace_{1\le k\le n}\subset V_n$, where
\begin{equation*}
e_{k}^{(n)}=E(\bm v_k), \quad  k\in\bi 1, n\ei\veq
\end{equation*}
is an orthonormal family of functions of $V_n$. 
Moreover, given that $E$ is linear and bijective, $\lbrace E(\bm v_k)\rbrace_{1\le k\le n}$ is a basis of $V_n$. Consequently, $\lbrace e_{k}^{(n)}\rbrace_{1\le k\le n}$ defines an orthonormal basis of $V_n$ composed of eigenfunctions of $-\Delta_n$.

Consider a \psd $\gamma$ satisfying \Cref{assum:alphabeta}. Following the definition of the discretized operator $-\Delta_n$ and analogously to the definition of the operator $\gamma(-\Delta_{\mathcal{M}})$, the discretization of the operator $\gamma(-\Delta_\mathcal{M})$ on $V_n$ is defined as the endomorphism $\gamma(-\Delta_n)$ of $V_n$ given by
\begin{equation}
\begin{aligned}
\gamma(-\Delta_n)\varphi = \sum\limits_{k=1}^n\gamma(\lambda_{k}^{(n)})( \varphi, e_{k}^{(n)})_0 \; e_{k}^{(n)}, \quad \varphi \in V_n \peq
\end{aligned}
\label{eq:discr_pseudo_diff}
\end{equation}
Note that this definition does not depend on the choice of orthonormal basis (cf. \Cref{prop:onb} of the \sm).

\subsection{Galerkin discretization of Gaussian random fields}

Let $\mathcal{W}_n$ be the $V_n$-valued random variable defined by
\begin{equation}
\mathcal{W}_n=\sum\limits_{k=1}^n W_k e_{k}^{(n)} \veq
\label{eq:wn_discr_ortho}
\end{equation}
where $W_1, \dots, W_n$ are independent standard Gaussian random variables. Then, $\mathcal{W}_n$ is called \textit{white noise on $V_n$}
and satisfies, for any $\varphi, \varphi_1, \varphi_2 \in V_n$, $\e[( \mathcal{W}_n, \varphi)_0]=0$ and $$\cov[( \mathcal{W}_n, \varphi_1)_0,( \mathcal{W}_n, \varphi_2)_0]=( \varphi_1, \varphi_2)_0.$$
It can be expressed in the basis functions $\lbrace \psi_k\rbrace_{1\le k \le  n}$ of~$V_n$, as stated in the next proposition which leads to an expression of the white noise using a basis that does not have to be orthonormal or an eigenbasis of~$-\Delta_{n}$.
\begin{prop}
	Let $\mathcal{W}_n$ be a white noise on $V_n$. Then, $\mathcal{W}_n$ can be written as
	\begin{equation*}
	\mathcal{W}_n=\sum_{k=1}^n \tilde{W}_k\psi_k \veq
	\end{equation*}
	where $\tilde{\bm W}=(\tilde{W}_1, \dots, \tilde{W}_n)^T$ is a centered Gaussian vector with covariance matrix $\bm C^{-1}$.
	\label{prop:wnc}
\end{prop}

\begin{proof}
	Let $\bm W=(W_1, \cdots, W_n)^T$ be the vector containing the random weights defining $\mathcal{W}_n$ in~\eqref{eq:wn_discr_ortho}. Using the linearity of $E$ in \Cref{prop:diag_lapl_discr}, $\mathcal{W}_n\in V_n$  can be written as
	\begin{equation*}
	\mathcal{W}_n=\sum_{k=1}^n  W_k E(\bm v_k)=E\bigg(\sum_{k=1}^n   W_k\bm v_k\bigg)=E( \bm V\bm W ) \veq
	\end{equation*}
	where $\bm W \sim \mathcal{N}(\bm 0,\bm I)$. But also, denoting by $\tilde{\bm W}=(\tilde{W}_1, \dots, \tilde{W}_n)^T$ the vector containing the coordinates of $\mathcal{W}_n$ in the basis $\lbrace \psi_k\rbrace_{1\le k\le n}$ of $V_n$, we get from \Cref{eq:def_isom_eig}, 
	\begin{equation*}
	\mathcal{W}_n=\sum_{k=1}^n  \tilde{W}_k \psi_i=E\big((\sqrt{\bm C})^{T}\tilde{\bm W} \big).
	\end{equation*}
	Hence, since $E$ is bijective, we get $\tilde{\bm W} =(\sqrt{\bm C})^{-T}\bm V{\bm W} $ which proves the result.\qed
\end{proof}

Inspired by the definition of the $H$-valued random field $\mathcal{Z}$ in~\eqref{eq:def_z}, we introduce its Galerkin discretization  $\mathcal{Z}_n$ as the $V_n$-valued random field  defined by
\begin{equation}
\mathcal{Z}_n=\gamma(-\Delta_n)\mathcal{W}_n=\sum_{k=1}^n \gamma(\lambda_{k}^{(n)}) W_k e_{k}^{(n)} \veq
\label{eq:def_Zn}
\end{equation}
where $W_1, \dots, W_n$ are independent standard Gaussian random variables.
Expressing $\mathcal{Z}_n$ in the basis functions $\lbrace \psi_k\rbrace_{1\le k \le  n}$ can then be done straightforwardly using the next theorem, leading to a first method to generate approximations of~$\mathcal{Z}$.

\begin{theorem}
	The discretized field $\mathcal{Z}_n$ can be decomposed in the basis $\lbrace\psi_k\rbrace_{1\le k\le n}$ as
	\begin{equation}
	\mathcal{Z}_n=\sum_{k=1}^n Z_k \psi_k \veq
	\label{eq:fem_discr}
	\end{equation}
	where $\bm Z=(Z_1, \dots, Z_n)^T$ is a centered Gaussian vector with covariance matrix given by
	\begin{equation}
	\var[\bm Z]=\big(\sqrt{\bm C}\big)^{-T}\,\gamma^2(\bm S)\,\big(\sqrt{\bm C}\big)^{-1}
	\label{eq:cov_weights}
	\end{equation}
	with 
	\begin{equation*}
	\gamma^2(\bm S)
	=
	\bm V
	\Diag\left(\gamma\big(\lambda_{1}^{(n)}\big)^2 , \dots, \gamma\big(\lambda_{n}^{(n)}\big)^2\right)
	\bm V^T \peq
	\end{equation*}
	\label{th:cov_weights}
\end{theorem}

\begin{proof}
	Notice that $\mathcal{Z}_n$ is $V_n$-valued, hence there exists some random vector $\bm Z\in \R^n$ such that $\mathcal{Z}_n=\sum_{k=1}^n Z_k \psi_k$. Following \Cref{prop:diag_lapl_discr}, we obtain $\mathcal{Z}_n=E((\sqrt{\bm C})^{T}\bm Z)$.
	But following instead the definition of $\mathcal{W}_n$ in~\eqref{eq:wn_discr_ortho} and the linearity of $E$, we get
	\begin{align*}
	\mathcal{Z}_n
	=\sum_{k=1}^n \gamma(\lambda_{k}^{(n)}) W_k E(\bm v_k)
	=E\bigg(\sum_{k=1}^n \gamma(\lambda_{k}^{(n)}) W_k \bm v_k\bigg)
	& =E(\bm V
	\Diag\left(\gamma\big(\lambda_{1}^{(n)}\big) , \dots, \gamma\big(\lambda_{n}^{(n)}\big)\right)
	\bm W) \veq
	\end{align*}
	where $\bm W=(W_1, \dots, W_n)^T \sim \mathcal{N}(\bm 0, \bm I)$. 
	Therefore, given that $E$ is bijective, $$\bm Z=\big(\sqrt{\bm C}\big)^{-T}\bm V\Diag\big(\gamma\big(\lambda_{1}^{(n)}\big) , \dots, \gamma\big(\lambda_{n}^{(n)}\big)\big)\bm W \veq$$
	which proves the result.\qed
\end{proof}

\Cref{th:cov_weights} provides an explicit expression for the covariance matrix of the weights of $V_n$-valued random variables. Consequently, generating realizations of such random functions can  be done by simulating a centered Gaussian random vector of weights with covariance matrix~\eqref{eq:cov_weights} and then building the weighted sum~\eqref{eq:fem_discr}.

A particular case, investigated in~\cite{borovitskiy2020mat},  is when $V_n$ is spanned by the set of eigenfunctions associated with the first $n$ eigenvalues (sorted in non-decreasing order and counted with their multiplicities) of the Laplace--Beltrami operator. Then, the discretized random field $\mathcal{Z}_n$ corresponds to a truncation of order $n$ of the series in \Cref{prop:def_gegf_sum_eigf} that defines the random field~$\mathcal{Z}$. Hence, we have a direct extension to Riemannian manifolds of the spectral methods used to sample isotropic random fields with spectral density $\gamma^2$ on a bounded domain of $\R^d$~\cite{chiles1999geost} or a sphere~\cite{lang2015isotropic}. 
In practice though, for arbitrary compact, connected and  oriented Riemannian manifolds, the eigenfunctions of the Laplace--Beltrami operator are not readily available and must be computed numerically, rendering such spectral methods potentially cumbersome. But since the only requirement on $V_n$ was for this space to be a finite-dimensional subspace of $\dot{H}^1$,  \Cref{th:cov_weights}  is applicable to more general choices of approximation spaces~$V_n$. 

\section{Chebyshev approximation of the discretized random field}
\label{sec:cheb}\label{sec:motiv_cheb}

Since the weights of the discretized random field characterized in \Cref{th:cov_weights} form a centered Gaussian random vector, they are entirely characterized by their covariance matrix. We show how the particular form of this covariance matrix can be used to propose efficient sampling methods.

Let $\bm Z$ be the centered Gaussian random vector generating~$\mathcal{Z}_n$ in \Cref{th:cov_weights}. Then, $\bm Z$ can be expressed as the solution to the linear system
\begin{equation*}
\big(\sqrt{\bm C}\big)^{T}\bm Z = \bm X,
\end{equation*}
where $\bm X$ is a centered Gaussian random vector with covariance matrix $\gamma^2(\bm S)$. In this section, we review ways of generating the right-hand side of this linear system.

A rather straightforward way to generate samples of $\bm X$ would be to compute the product
\begin{equation}
\bm X = \sqrt{\gamma^2(\bm S)} \bm W
\label{eq:sim_vec}
\end{equation}
where $\bm W \sim \mathcal{N}(\bm 0, \bm I)$ and $\sqrt{\gamma^2(\bm S)}$ is a square-root of $\gamma^2(\bm S)$, i.e., a matrix satisfying $\gamma^2(\bm S)=\sqrt{\gamma^2(\bm S)}\big(\sqrt{\gamma^2(\bm S)}\big)^T$.
Suitable choices are the Cholesky factorization of $\gamma^2(\bm S)$ and the matrix $\gamma(\bm S)$. 
However these choices would entail to fully diagonalize the matrix~$\bm S$ since they rely on matrix functions. This requires a workload of $\mathcal{O}(n^3)$ operations and a storage space of $\mathcal{O}(n^2)$. To reduce these high costs, we propose to use a polynomial approximation of the square-root based on Chebyshev series instead.

Let $\bm X$ be a sample of the weights obtained through the relation
\begin{equation}
\bm X =\gamma(\bm S)\bm W
\label{eq:prod_gs}
\end{equation}
where $\bm W \sim \mathcal{N}(\bm 0, \bm I)$.  Note that in the particular case where $\gamma = P$ is a polynomial of degree~$K$ with coefficients $a_0, \dots,a_{K}\in\R$, we have
\begin{equation*}
\begin{aligned}
\bm X 
&=\bm V
\Diag\big(P(\lambda_{1}^{(n)}), \dots, P(\lambda_{n}^{(n)})\big)
\bm V^T\bm W
=\sum_{k=0}^Ka_k\bm S^k\bm W \peq
\end{aligned}
\end{equation*}
This means in particular that the product $P(\bm S)\bm W$ can be computed iteratively, while requiring at each iteration only a single product between $\bm S$ and a vector. Hence, no diagonalization of the matrix is needed in this case.
Building on this idea, we propose to approximate, for a general function $\gamma$, the vector $\bm X$ in \eqref{eq:prod_gs} by the vector $\widehat{\bm X}$ defined by
\begin{equation*}
\widehat{\bm X} =P_{\gamma,K}(\bm S)\bm W \sveq
\end{equation*}
where $P_{\gamma,K}$ is a polynomial approximation of degree $K\in\N$ of $\gamma$, over an interval containing all the eigenvalues of $\bm S$. In particular, since $\bm S$ is positive semi-definite, we consider this interval to be $[0, \lambda_{\max}]$ where $\lambda_{\max}$ is some upper bound of the greatest eigenvalue of $\bm S$.

We choose the basis of Chebyshev polynomials (of the first kind) to compute the expression of the approximating polynomial $P_{\gamma,K}$. These polynomials  are the family  $\lbrace T_k\rbrace_{k\in\mathbb{N}_0}$ of polynomials defined over $[-1,1]$ by:
\begin{equation}
	T_k(\cos\theta)=\cos(k\theta), \quad \theta\in [-\pi, \pi], \quad k\in\N_0, 
	\label{eq:cheb_cos}
\end{equation}
or equivalently via the recurrence relation:
\begin{equation}
	T_0(t) =1, \quad
	T_1(t) =t, \quad
	T_{k+1}(t) =2t\,T_k(t)-T_{k-1}(t) \quad k\geq 1 \peq
	\label{eq:cheb_rec}
\end{equation}
Note in particular that for any $k\in\N_0$, $T_k$ is a polynomial of degree $k$ and that for any $t\in [-1, 1]$, $\vert T_k(t) \vert \le 1$.
A remarkable property of Chebyshev polynomials is that they form a set of orthogonal functions of the space $L^2_c([-1,1])$ defined by
\begin{equation*}
	L^2_c([-1,1])=\bigg\lbrace f : [-1, 1] \rightarrow \R \text{ such that } \int_{-1}^1 f(t)^2\frac{\dd t}{\sqrt{1-t^2}} < +\infty\bigg\rbrace 
\end{equation*}
and equipped with the inner product $\langle \cdot, \cdot \rangle_c$ defined by
\begin{equation*}
	\langle f, g\rangle_c = \int_{-1}^1 f(t)g(t)\frac{\dd t}{\sqrt{1-t^2}} \peq
\end{equation*}
As such, the truncated Chebyshev series of order $K\ge 0$ of any function $f\in L^2_c([-1, 1])$ is the polynomial of degree (at most) $K$ given by
\begin{equation}
	\mathcal{S}_K[f](t) = \frac{1}{2}c_0T_0(t)+\sum\limits_{k=1}^K c_k T_k(t), \quad t\in [-1,1] \veq
	\label{eq:cheb_sum}
\end{equation}
where the coefficients $c_k$ are defined by
\begin{equation}
	c_k = \frac{2}{\pi}\langle f, T_k\rangle_c, \quad k\ge 0 \peq
	\label{eq:coef_cheb_sum}
\end{equation}
 Truncated Chebyshev series of continuous functions are pointwise convergent in the $L^2_c$-sense \cite[Theorem 5.6]{mason2002chebyshev}, and for \psds they are uniformly convergent (cf. Appendix \ref{sec:cheb_pres} for more details). This motivates their use to approximate a \psd $\gamma$. Besides, using truncated Chebyshev series also guarantees:
\begin{itemize}
	\item the fact that at any order of approximation $K$, the polynomial $P_{\gamma,K}$ is near optimal in the sense that
	\begin{equation*}
		\Vert B_{\gamma}^* - \gamma\Vert_{\infty} \le \Vert P_{\gamma,K} - \gamma\Vert_{\infty} \le (1+\Lambda_K)\Vert B_{\gamma}^* - \gamma\Vert_{\infty},
	\end{equation*}
	where $B_{\gamma}^*$ is the best polynomial approximation of $\gamma$ of order $K$ and 
	$$\Lambda_K= (4/\pi^2)\log(K)+C+\mathcal{O}(K^{-1}),$$
	where $C\approx 1.27$ is the so-called Lebesgue constant of the approximation \cite[Chapter 5, Section 5]{mason2002chebyshev};
	
	\item the fact that the coefficients of the polynomial in the Chebyshev basis of polynomials can be computed very efficiently using the Fast Fourier Transform (FFT) algorithm \cite{cooley1965algorithm}, with a complexity that can be bounded by $\mathcal{O}(K\log K)$ to compute $K$ coefficients (see \cite[Section B.4.4]{pereira2019generalized} for an algorithm).
\end{itemize}

Since Chebyshev polynomials are defined on $[-1, 1]$, the interval of approximation $[0, \lambda_{\max}]$ must be mapped onto $[-1, 1]$ and vice versa, which is done with the linear change of variable $\theta : [-1,1] \rightarrow [0,  \lambda_{\max}]$, given by $\theta(t)=0.5\lambda_{\max}(1+t)$, $t\in [-1, 1]$.
The function $\tilde{\gamma} : [-1, 1] \rightarrow \R$ given by
\begin{equation}
\tilde{\gamma}(t)=\gamma\left(\theta(t)\right) , \quad
t\in [-1, 1]  \veq
\label{eq:chang_gamma}
\end{equation}
can then be approximated by a truncated Chebyshev series of order $K$, and the polynomial $P_{\gamma,K}$ approximating $\gamma$ on $[0,\lambda_{\max}]$ takes the form
\begin{equation}
P_{\gamma,K}(\lambda)
=\mathcal{S}_K[\tilde\gamma]\left(\theta^{-1}(\lambda)\right)
=\mathcal{S}_K[\tilde\gamma]\left(2\lambda_{\max}^{-1}\lambda-1\right), \quad \lambda\in[0, \lambda_{\max}] \veq
\label{eq:p_gamma}
\end{equation}
where $\mathcal{S}_K[\tilde{\gamma}]$ is the truncation of order $K$ of the Chebyshev series of $\tilde{\gamma}$.

Ultimately, the approximation  $\widehat{\mathcal{Z}}_{n,K}$ of the discretized field $\mathcal{Z}_n$ that results from the polynomial approximation introduced in this subsection takes the form
\begin{equation*}
\widehat{\mathcal{Z}}_{n,K} = \sum_{k=1}^n \widehat{Z}_k \psi_k ,
\end{equation*}
where the random weights $\widehat{\bm Z} = (\widehat{Z}_1, \dots, \widehat{Z}_n)^T$ are given by
\begin{equation}
\widehat{\bm Z}
=\big(\sqrt{\bm C}\big)^{-T}P_{\gamma,K}(\bm S)\bm W
=\big(\sqrt{\bm C}\big)^{-T}\sum_{k=0}^K c_k T_k(2\lambda_{\max}^{-1}\bm S-\bm I)\bm W
\label{eq:approx_cheb_zn}
\end{equation}
with $\bm W \sim \mathcal{N}(\bm 0, \bm I)$ and $c_0, \dots, c_K$ denote the first $K$ coefficients of the Chebyshev series of~$\tilde{\gamma}$. We call $\widehat{\mathcal{Z}}_{n,K}$ a \textit{Galerkin--Chebyshev approximation} of discretization order $n\in\N$ and polynomial order $K\in\N$ of the Gaussian random field $\mathcal{Z}$.


\section{Convergence analysis}
\label{sec:err}

The goal of this section is to derive the overall error between the random field $\mathcal{Z}$, as defined in~\eqref{eq:def_z}, and its Galerkin--Chebyshev approximation $\widehat{\mathcal{Z}}_{n,K}$ associated with a functional discretization space~$V_n$ of dimension $n$ and a Chebyshev polynomial approximation of order $K$ of the \psd. To derive this error, we assume for simplicity that the upper bound $\lambda_{\max}$ of the eigenvalues of the stiffness matrix $\bm S$ (on which the Chebyshev polynomial approximation is defined) is equal to the maximal eigenvalues of $\bm S$, i.e., $\lambda_{\max} = \lambda_n^{(n)}$. 

To prove convergence result between $\mathcal{Z}$ and $\widehat{\mathcal{Z}}_{n,K}$,  we need an additional assumption on the space $V_n$, or more precisely on the approximating properties of the discretized operator $-\Delta_n$ that this space yields. We assume the following link between the eigenpairs of~$-\Delta_n$ and those of~$-\Delta_{\mathcal{M}}$ (arranged in non-decreasing order).

\begin{assumb}{}
	Let $\alpha>0$ be defined in \Cref{prop:weyl}. There exist constants $N_0, C_1, C_2>0$, $l_\lambda \in (0,1]$, and exponents $r, s > 0$ and $q\ge 1$, satisfying the inequality
	\begin{equation}
		\alpha q \le \min\lbrace 2s, r+\alpha\rbrace,
		\label{eq:exp}
	\end{equation} 
	such that for all $n \ge N_0$ and $k\in [\![M_0+1, n]\!]$, 
	\begin{equation}
	\begin{aligned}
	\vert\lambda_{k}^{(n)} - \lambda_k \vert &\le ~C_1\lambda_k^q n^{-r}, \quad
	\Vert e_{k}^{(n)} - e_k\Vert_{0}^2 &\le ~C_2 \lambda_k^q n^{-2s}
	\end{aligned},
	\label{eq:ineqConv}
	\end{equation}
	and {
		\begin{equation}
		\lambda_{k}^{(n)} \ge l_{\lambda} \lambda_k .
		\label{eq:weylb}
		\end{equation}
	}
	\label{assum:Lh}
\end{assumb}
\vspace{-1em}

\begin{rem}\label{rem:eqEig}
	In the assumption above, we do not need to treat the case $\lambda_k=0$ (i.e., $M_0\neq 0$ and $k\le M_0$). Indeed, recall that the manifold is connected, and that therefore $M_0\in\lbrace 0, 1\rbrace$.  Hence, if $\lambda_k=0$ arises, there is exactly one such eigenvalue to approximate, namely $\lambda_{1}=0$. And in this case, since the discretized operator $-\Delta_n$ is positive semi-definite, we  have $\lambda_1^{(n)}=0=\lambda_{1}$ for any $n\in\N$. The same conclusion can be derived for the eigenfunctions since in both cases, they can be taken equal to a constant function with value~$1$. 
\end{rem}

In \Cref{assum:Lh}, the requirement~\eqref{eq:ineqConv} states that eigenvalues and eigenfunctions of $-\Delta_n$ should asymptotically lie within a ball around the  eigenvalues and eigenfunctions of $-\Delta_{\mathcal{M}}$, where the radius of the ball may grow with the magnitude of the eigenvalue but, for a fixed index $k$, decreases as $n\rightarrow +\infty$.
The requirement~\eqref{eq:weylb} expresses that, asymptotically, the eigenvalues of $-\Delta_n$ should grow at the same rate as the eigenvalues of $-\Delta_{\mathcal{M}}$. This last requirement may seem redundant with the first one but ensures that, even for large indices $k\approx n$, the eigenvalues $\lambda_{k}^{(n)}$ do not stay too far away from $\lambda_k$ (which is not always ensured by the first requirement).

A straightforward example of a discretization space $V_n$ for which \Cref{assum:Lh} is satisfied is when $V_n$ is defined as the set containing the first $n$ eigenfunctions of the Laplace--Beltrami operator, since then $\lambda_{k}^{(n)}=\lambda_{k}$ and $e_k^{(n)}=e_k$ for any $k\in\bi 1, n\ei$. The resulting Galerkin--Chebyshev approximation of the field then amounts to a classical spectral method. In this case, one can use directly the Galerkin approximation of the random field for sampling purposes without requiring a Chebyshev polynomial approximation of the \psd (cf. \Cref{sec:trunc} for more details). However, considering this particular discretization space~$V_n$ implies that the eigenfunctions of the Laplace--Beltrami operator are known, which is seldom in practice.

An alternative to the spectral method consists in building the discretization space $V_n$ from basis functions of a finite element space. If the Riemannian manifold $(\mathcal{M}, g)$ is a bounded convex polygonal domain equipped with the Euclidean metric, and $V_n$ is the linear finite element space associated with a quasi-uniform triangulation of $\mathcal{M}$ with mesh size $h \lesssim n^{-1/d}$, then \Cref{assum:Lh} is satisfied for the exponents $r=s=\alpha=2/d$ and $q=2$ \cite[Theorems 6.1 \& 6.2]{strang1973analysis}. %

If now $\mathcal{M}$ is a smooth compact $2$-dimensional surface without boundary equipped with the metric $g$ induced by the Euclidean metric on $\R^3$ (and called pullback metric, see \cite[Chapter 13]{lee2013smooth} for more details), the surface finite element method (SFEM) provides a way to construct a finite element space on the surface $\mathcal{M}$ by \q{lifting} on $\mathcal{M}$ a linear finite element space defined on a polyhedral approximation of $\mathcal{M}$ that lies \q{close} to the surface (see \cite{dziuk1988finite} and \cite[Section 2.6]{demlow2009higher} for more details). The discretization space~$V_n$ can then be taken as the linear span of the lifted finite element basis functions defined on the polyhedral surface. One can show that, $\vert \lambda_{k}^{(n)}-\lambda_{k}\vert \lesssim \lambda_k^2 n^{-1}$ and that $\lambda_{k}^{(n)}\le \lambda_{k}$ (cf. Appendix~\ref{appen:sfem} for more details). Proving the eigenfunction inequality is open and ongoing work, but our numerical experiments in \Cref{sec:num} indicate that our error estimates hold.

\begin{rem}
	In practice, when using SFEM, it is usual to consider the eigenfunctions and eigenvalues of the discrete operator defined on the polyhedral approximation $\widehat{\mathcal{M}}$ of the surface $\mathcal{M}$ (as opposed to the original surface $\mathcal{M}$). In that case, $V_n$ is not a subset of functions of $\mathcal{M}$ but rather a subset of functions of $\widehat{\mathcal{M}}$, which is considered in the numerical experiments in \Cref{sec:num}. Then, the error on the approximation in $V_n$ of the eigenvalues and eigenvectors of the Laplace--Beltrami operator of $\mathcal{M}$ can be written as (see \cite{bonito2018priori}):
	\begin{equation*}
		\begin{aligned}
			\vert\lambda_{k}^{(n)} - \lambda_k \vert &\le ~\widehat C_1(\lambda_k) n^{-1}, \quad
			\Vert e_{k}^{(n)} - e_k\Vert_{0}^2 &\le ~\widehat C_2(\lambda_k) n^{-2}
		\end{aligned},
	\end{equation*}
	where the explicit dependence of the constants $\widehat{C}_1$ and $\widehat{C}_2$ on $\lambda_k$ is given in \cite{bonito2018priori}. 
	Hence, if one can write $C_1(\lambda_k)\lesssim \lambda_{k}^q$ and $C_2(\lambda_k)\lesssim \lambda_{k}^q$ for some $q\in [1,2]$, then \Cref{assum:Lh} is satisfied, which is ongoing work.
\end{rem}

We now state the main results of this section.

\begin{theorem}
Let Assumptions \ref{assum:alphabeta} and \ref{assum:Lh} be satisfied.
	Then, the approximation error of the random field $\mathcal{Z}$ by its Galerkin--Chebyshev  approximation $\widehat{\mathcal{Z}}_{n,K}$ of discretization order $n\in\N$ big enough and polynomial order $K\in\N$, satisfies
	\begin{equation*}
	\Vert \mathcal{Z} - \widehat{\mathcal{Z}}_{n,K} \Vert_{L^2(\Omega; H)} \le C_{\text{Galer}}\; n^{-\rho} 
	+ C_{\text{pol}}\; n^{\alpha\nu+1/2}(K-\nu)^{-\nu}
	\veq
	\end{equation*}
	where $C_{\text{Galer}}$ and $C_{\text{pol}}$  are constants independent of $n$ and $K$, $\rho=\min\left\lbrace s ;\; r;\; (\alpha\beta-1/2)\right\rbrace>0$, $\alpha>0$ is defined in \Cref{prop:weyl}, $r>0$ and $s>0$ are given in \Cref{assum:Lh}, and $\beta>0$ and $\nu\in\N$ as in \Cref{def:psd}. 	
	\label{th:err_all}
\end{theorem}

When the \psd $\gamma$ is known to be analytic over $[0,\lambda_{\max}]$ (meaning in particular that in \Cref{def:psd} any $\nu\in\N$ works), the polynomial approximation error can be shown to decrease at an exponential rate. The resulting overall error between the random field $\mathcal{Z}$ and its approximation $\widehat{\mathcal{Z}}_{n,K}$ can then be upper bounded as stated in the next result.

\begin{corol}
		Let Assumptions \ref{assum:alphabeta} and \ref{assum:Lh} be satisfied and  let $\gamma$ be a \psd such that there exists some $\chi >0$ such that the map $z \in\mathbb{C} \mapsto \gamma(z)$ is holomorphic inside the ellipse $E_\chi\subset \mathbb{C}$ centered at $z=\lambda_{\max}/2$, with foci $z_1=0$ and $z_2=\lambda_{\max}$, and semi-major axis $a_\chi=\lambda_{\max}/2+\chi$.  
	
	Then, the approximation error of the random field $\mathcal{Z}$ by its Galerkin--Chebyshev  approximation $\widehat{\mathcal{Z}}_{n,K}$ of discretization order $n\in\N$ big enough and polynomial order $K\in\N$, satisfies
	\begin{equation*}
		\Vert \mathcal{Z} - \widehat{\mathcal{Z}}_{n,K} \Vert_{L^2(\Omega; H)} \le C_{\text{Galer}}\;n^{-\rho} 
		+ \tilde C_{\text{pol}}\;
		n^{(\alpha+1)/2}\exp(-\widehat C_{\text{pol}}\; n^{-\alpha/2}K)
		\veq
	\end{equation*}
	where $C_{\text{Galer}}$. $\tilde C_{\text{pol}}$ and $\widehat C_{\text{pol}}$  are constants independent of $n$ and $K$, $\rho=\min\left\lbrace s ;\; r;\; (\alpha\beta-1/2)\right\rbrace>0$, $\alpha>0$ is defined in \Cref{prop:weyl}, $r>0$ and $s>0$ are given in \Cref{assum:Lh}, and $\beta>0$ as in \Cref{def:psd}.

	\label{th:err_all_ana}
\end{corol}

We prove these two error estimates by upper bounding the left-hand side by the sum of a discretization error and a polynomial approximation error, both of which are derived in the next two subsections. The discretization error is computed in the more general setting on spaces $\dot H^\sigma$ defined in \Cref{sec:func_lap} (with $\sigma=0$ giving the error on~$H$). We also provide an interpretation of the terms composing this error estimate, as well as a result on the convergence of the covariance of the discretization scheme.

\subsection{Error analysis of the discretized field}

In this section, a convergence result of the discretized field $\mathcal{Z}_n$ is derived in terms of a root-mean-squared error on the spaces $\dot H^\sigma$ defined in \Cref{sec:func_lap}.

\begin{theorem}
	Let Assumptions \ref{assum:alphabeta} and \ref{assum:Lh}  be satisfied.
	Then, there exists $N_1\in\N$ such that for any $n>N_1$, and 
	$\sigma \in [0, \alpha^{-1}(2\alpha\beta -1))$,
	the approximation error of the random field $\mathcal{Z}$ by its discretization $\mathcal{Z}_n$  satisfies
	\begin{equation}\label{eq:err_str}
	\Vert \mathcal{Z} - \mathcal{Z}_n \Vert_{L^2(\Omega; \dot{H}^\sigma)} \lesssim 
	\begin{cases}
	n^{-\min\left\lbrace s ;\; r;\; (2\alpha\beta-1-\alpha\sigma)/2\right\rbrace}(\log n)^{1/2} 
	& \text{if } (2\alpha\beta-1-\alpha\sigma)/2 = s,\\
	n^{-\min\left\lbrace s ;\; r;\; (2\alpha\beta-1-\alpha\sigma)/2\right\rbrace}(\log n)^{1/2} 
	& \text{if } (2\alpha\beta-1-\alpha\sigma)/2 = r \text{ and } q>1,\\
	n^{-\min\left\lbrace s ;\; r;\; (2\alpha\beta-1-\alpha\sigma)/2\right\rbrace}
	& \text{else},
	\end{cases}
	\end{equation}
	where $\alpha>0$ is defined in \Cref{prop:weyl}, $q\ge 1$, $r>0$ and $s>0$ are given in \Cref{assum:Lh}, and $\beta>0$ as in \Cref{def:psd}.
	\label{th:err_fem}
\end{theorem}
\vspace{-1em}

\begin{proof}
	Let $n>\max\lbrace M_0; N_0\rbrace$,  and let $\mathcal{Z}^{(n)}$ be the truncated random field of~$\mathcal{Z}$ given by
	\begin{equation*}
	\mathcal{Z}^{(n)}=\sum\limits_{k=1}^{n} W_j\gamma(\lambda_j) e_j \speq
	\end{equation*}
	We split the error with the triangle inequality into
	\begin{equation*}
	\Vert \mathcal{Z} - \mathcal{Z}_n \Vert_{L^2(\Omega; \dot{H}^\sigma)}
	\le \Vert \mathcal{Z} - \mathcal{Z}^{(n)} \Vert_{L^2(\Omega; \dot{H}^\sigma)} 
	+ \Vert \mathcal{Z}^{(n)} - \mathcal{Z}_n \Vert_{L^2(\Omega; \dot{H}^\sigma)}
	\end{equation*}
	and bound both terms in what follows.
	
	\vspace{0.5em}
	
	\noindent\underline{Truncation error term} $\Vert \mathcal{Z} - \mathcal{Z}^{(n)} \Vert_{L^2(\Omega; \dot{H}^\sigma)}$ :
	Note that
	\begin{equation*}
	\Vert \mathcal{Z} - \mathcal{Z}^{(n)} \Vert_{L^2(\Omega; \dot{H}^\sigma)}^2 
	=\e\bigg[\big\Vert \sum\limits_{j>n}W_j\gamma(\lambda_j) e_j \big\Vert_{\sigma}^2\bigg]
	=\sum\limits_{j>n}\lambda_j^\sigma\vert\gamma(\lambda_j)\vert^2 \sveq
	\end{equation*}
	which leads by \Cref{prop:weyl} and \Cref{assum:alphabeta} to
	\begin{equation}
	\left\Vert \mathcal{Z} - \mathcal{Z}^{(n)} \right\Vert_{L^2(\Omega; \dot{H}^\sigma)}^2 
	\lesssim \sum\limits_{j>n}  \lambda_j^{\sigma-2\beta} 
	\lesssim  \sum\limits_{j>n}  j^{-\alpha(2\beta-\sigma)} 
	\lesssim n^{-(2\alpha\beta-\alpha\sigma-1)}
	\sveq
	\label{eq:err_1_fem}
	\end{equation}
	where the last inequality is derived using a Riemann sum associated with the integration of the function $t\mapsto t^{-\alpha(2\beta-\sigma)}$ and using the assumption that $\alpha(2\beta-\sigma)>1$.
	
	·
	\vspace{0.5em}
	
	\noindent\underline{Discretization error} $\Vert \mathcal{Z}^{(n)} - \mathcal{Z}_n \Vert_{L^2(\Omega; \dot{H}^\sigma)}$ : 
	We split the error further by the triangle inequality into
	\begin{align*}
	&\Vert \mathcal{Z}^{(n)} -  \mathcal{Z}_n \Vert_{L^2(\Omega; \dot{H}^\sigma)}
	=\big\Vert \sum\limits_{j=1}^{n}W_j\gamma(\lambda_j)e_j  - \sum\limits_{j=1}^{n}W_j\gamma(\lambda_{j}^{(n)}) e_{j}^{(n)} \big\Vert_{L^2(\Omega; \dot{H}^\sigma)}  \\
	&\qquad \le \big\Vert \sum\limits_{j=1}^{n}W_j\gamma(\lambda_j) (e_j -e_{j}^{(n)} ) \big\Vert_{L^2(\Omega; \dot{H}^\sigma)}
	+\big\Vert \sum\limits_{j=1}^{n} W_j\left(\gamma(\lambda_j)-\gamma(\lambda_{j}^{(n)})\right)e_{j}^{(n)} \big\Vert_{L^2(\Omega; \dot{H}^\sigma)}\\
	& \qquad =(\text{I})+(\text{II})
	\speq
	\end{align*}
	%
	The first term satisfies
	\begin{align*}
	(\text{I})^2 
	&= \sum\limits_{j,k=1}^{n}\lambda_k^{\sigma(k)/2}\gamma(\lambda_k)\lambda_j^{\sigma(j)/2}\gamma(\lambda_j)\e\left[W_j{W_k}\right]\left( e_j-e_{j}^{(n)},  e_k-e_{k}^{(n)}\right)_{0} \sveq
	\end{align*}
	where for any $i\in\bi 1,n\ei$, $\sigma(i)=2$ if $\lambda_i=0$ and $\sigma(i)=\sigma$ otherwise. 
	Hence, using the independence of the Gaussian random weights $\lbrace W_j\rbrace_{j\in\N}$ and \Cref{assum:Lh},
	\begin{equation*}
	(\text{I})^2  
	= \vert\gamma(0)\vert^2\sum_{j=1}^{M_0}\Vert e_j-e_{j}^{(n)}\Vert_{0}^2 
	+ \sum_{j=M_0+1}^{n}\lambda_j^{\sigma}\vert\gamma(\lambda_j)\vert^2\Vert e_j-e_{j}^{(n)}\Vert_{0}^2 \speq
	\end{equation*}
	Following \Cref{rem:eqEig}, the first sum in $(\text{I})^2$ is $0$. It then follows from \Cref{assum:Lh}, \Cref{prop:weyl}, and \Cref{assum:alphabeta} that
	\begin{equation*}
	(\text{I})^2  \lesssim  n^{-2s}\bigg( 
	\sum_{j=M_0+1}^n\vert\lambda_j^{\beta}\gamma(\lambda_j)\vert^2\lambda_j^{q-2\beta+\sigma} \bigg)
	\lesssim  n^{-2s}\bigg( 
	\sum_{j=1}^nj^{\alpha(q-2\beta+\sigma)} \bigg)
	\speq
	\end{equation*}
	And using the fact that $\alpha q\le 2s$ (cf. \cref{eq:exp}), we finally obtain
	\begin{equation*}
	\begin{aligned}
	(\text{I})^2  
	&
	\lesssim  n^{-2s}\bigg( 
	\sum_{j=1}^nj^{2s-(2\alpha\beta-\alpha\sigma)} \bigg)
	\speq
	\end{aligned}
	\end{equation*}
	Bounding the sum again by the corresponding integral, we distinguish three cases:
	\begin{itemize}
		\item if $2s-(2\alpha\beta-\alpha\sigma)>-1$,  then $\sum_{j=1}^nj^{2s-(2\alpha\beta-\alpha\sigma)}  \lesssim n^{2s-(2\alpha\beta-\alpha\sigma)+1} $;
		\item if $2s-(2\alpha\beta-\alpha\sigma)=-1$,  then $\sum_{j=1}^nj^{2s-(2\alpha\beta-\alpha\sigma)}  \lesssim \log n $;
		\item if $2s-(2\alpha\beta-\alpha\sigma)<-1$,  then $\sum_{j=1}^nj^{2s-(2\alpha\beta-\alpha\sigma)}  \lesssim 1 $.
	\end{itemize}	
	Hence, we conclude
	\begin{equation*}
	(\text{I})^2 \lesssim
	\begin{cases}	
	n^{-2s}\log n &  \text{if } 2s = (2\alpha\beta-\alpha\sigma-1),\\
	n^{-\min\lbrace 2s ;\; (2\alpha\beta-\alpha\sigma-1)\rbrace} &  \text{otherwise},
	\end{cases} 
	\end{equation*}
	%
	%
	%
	and continue with bounding
	\begin{equation*}
	\begin{aligned}
	(\text{II})^2 
	=\sum\limits_{j=1}^{M_0}\vert\gamma(\lambda_j)-\gamma(\lambda_{j}^{(n)})\vert^2 
	+\sum\limits_{j=M_0+1}^{n}\lambda_j^\sigma\vert\gamma(\lambda_j)-\gamma(\lambda_{j}^{(n)})\vert^2 \speq
	\end{aligned}
	\label{eq:ineq_proof_2}
	\end{equation*}

	Following \Cref{rem:eqEig}, the first sum in $(\text{II})^2$ is $0$. We then focus on the terms composing the second sum.
	The mean value theorem gives for any $j\in\bi M_0+1,n\ei$,
	$$\vert\gamma(\lambda_j)-\gamma(\lambda_{j}^{(n)})\vert\le \vert\lambda_{j}^{(n)}-\lambda_j\vert  \sup\limits_{\theta \in (0,1)} \vert \gamma'(\theta\lambda_{j}^{(n)}+ (1-\theta)\lambda_{j})\vert 
	\speq$$

	We have, for $n>N_0$, $\min \big\lbrace\lambda_k; \lambda_{k}^{(n)}\big\rbrace \ge l_\lambda \lambda_k \ge l_\lambda c_\lambda k^{\alpha}$ as a consequence of \Cref{prop:weyl} and \Cref{assum:Lh}. We can therefore find $N_1 > N_0$ such that for any $n>N_1$ and any $j\in \bi N_1, n\ei$, $\min\big\lbrace\lambda_j ;\; \lambda_{j}^{(n)}\big\rbrace \ge l_\lambda c_\lambda k^{\alpha} \ge L_\gamma$, where $L_\gamma$ is defined in \Cref{def:psd}. Then, for any $j\in \bi N_1, n\ei$,
	$$\vert\gamma(\lambda_j)-\gamma(\lambda_{j}^{(n)})\vert
	\lesssim \vert\lambda_{j}^{(n)}-\lambda_j\vert  \left(\min\big\lbrace\lambda_j ;\; \lambda_{j}^{(n)}\big\rbrace \right)^{-(1+\beta)}
	\lesssim  \vert\lambda_{j}^{(n)}-\lambda_j\vert  j^{-\alpha(1+\beta)} .$$
	And for $j<N_1$, we can take
	$$\vert\gamma(\lambda_j)-\gamma(\lambda_{j}^{(n)})\vert\le S_\gamma'\vert\lambda_{j}^{(n)}-\lambda_j\vert, $$		 
	where $S_\gamma' = \sup_{ [0, L_\gamma]}\vert \gamma'\vert$.
	Therefore, using the last two inequalities (and applying again \Cref{prop:weyl} and \Cref{assum:Lh}), we get
	\begin{equation*}
	(\text{II})^2 
	\lesssim
	\sum\limits_{j=M_0+1}^{N_1-1} \lambda_j^\sigma\vert\lambda_{j}^{(n)}-\lambda_j\vert^2
	+\sum\limits_{j=N_1}^{n}  \lambda_j^\sigma\vert\lambda_{j}^{(n)}-\lambda_j\vert^2  j^{-2(1+\beta)}
	\lesssim
	n^{-2r}\bigg(1+\sum\limits_{j=N_1}^{n} j^{2\alpha(q-\beta-1)+\alpha\sigma}\bigg).
	\end{equation*}	 	
	If $q=1$, we  have $(\text{II})^2 \lesssim n^{-2r}$ since $2\alpha\beta-\alpha\sigma > 1$.
	If $q>1$, since $\alpha(q-1) \le r$, we obtain
	\begin{equation*}
	(\text{II})^2 
	\lesssim n^{-2r}\bigg(1+\sum\limits_{j=1}^{n} j^{2r-2\alpha\beta+\alpha\sigma}\bigg),
	\end{equation*}
	and using the same argument as for~$(\text{I})^2$, we conclude that
	\begin{equation*}
	(\text{II})^2 \lesssim
	\begin{cases}
	n^{-2r}\log n & \text{if } 2r = (2\alpha\beta-\alpha\sigma-1), \\
	n^{-\min\lbrace 2r ;\; (2\alpha\beta-\alpha\sigma-1)\rbrace} &  \text{else.} 
	\end{cases} 
	\end{equation*}
	
	%
	
	Combining the terms $(\text{I})$ and $(\text{II})$ finally gives, if $q>1$,
	\begin{equation}
	\begin{aligned}
	\Vert \mathcal{Z}^{(n)} &-  \mathcal{Z}_n \Vert_{L^2(\Omega; \dot{H}^\sigma)} \\
	&\lesssim
	\begin{cases}
	(\log n)^{1/2}( n^{-s} +n^{-r})
	& \text{if } \alpha(\beta-\sigma/2)-1/2 = s=r,\\
	(\log n)^{1/2}n^{-s} +n^{-\min\lbrace r, (\alpha(\beta-\sigma/2)-1/2)\rbrace}
	& \text{if } \alpha(\beta-\sigma/2)-1/2 = s\neq r,\\
	(\log n)^{1/2}n^{-r} +n^{-\min\lbrace s, (\alpha(\beta-\sigma/2)-1/2)\rbrace}
	& \text{if } \alpha(\beta-\sigma/2)-1/2 = r\neq s,\\
	n^{-\min\lbrace s, (\alpha(\beta-\sigma/2)-1/2)\rbrace} +n^{-\min\lbrace r, (\alpha(\beta-\sigma/2)-1/2)\rbrace}
	& \text{else.}
	\end{cases}
	\end{aligned}
	\label{eq:err_2_fem}
	\end{equation}
	and if $q=1$, 
	\begin{equation}
	\begin{aligned}
	\Vert \mathcal{Z}^{(n)} -  \mathcal{Z}_n \Vert_{L^2(\Omega; \dot{H}^\sigma)} 
	\lesssim
	\begin{cases}
	(\log n)^{1/2}n^{-s}+n^{-r}
	& \text{if } \alpha(\beta-\sigma/2)-1/2 = s\\
	n^{-\min\lbrace s, (\alpha(\beta-\sigma/2)-1/2)\rbrace} +n^{-r}
	& \text{else.}
	\end{cases}
	\end{aligned}
	\label{eq:err_2_fem1}
	\end{equation}	
	%
	The proof is concluded by bounding \Cref{eq:err_1_fem,eq:err_2_fem,eq:err_2_fem1} by the smallest exponents.\qed
\end{proof}

This error estimate~\eqref{eq:err_str} yields the same convergence rate as the one derived in \cite{BKK20,bolin2020rational} in their approximation of solutions to fractional elliptic SPDEs with spatial white noise, but our result differs from their result in three aspects. First, we defined our random fields on Riemannian manifolds. Then, the random fields covered by their result can be seen as those specific choices of $\gamma$ such that $\gamma$ is non-zero over $\R_+$. Finally, we use slightly different assumptions on the discretization space: in \Cref{assum:Lh}, we do not assume that $\lambda_{k}^{(n)}\ge \lambda_{k}$. This assumption holds in particular for finite element spaces associated with conforming triangulation and on domains of $\R^d$ \cite{strang1973analysis}, and dropping it allows to open the way to the use of non-conforming methods.

We conclude this subsection by investigating the overall error in the covariance between the random field~$\mathcal{Z}$ and its discretized counterpart $\mathcal{Z}_n$. This error is described in the next theorem and is derived using the same approach as in \Cref{th:err_fem}.

\begin{theorem}
	Let Assumptions \ref{assum:alphabeta} and \ref{assum:Lh}  be satisfied. Then, there exists some $N_2\in \N$ such that for any $n>N_2$, 
	the covariance error between the random field $\mathcal{Z}$ and its discretization $\mathcal{Z}_n$ satisfies, for any $\theta,\varphi \in H$,
	\begin{equation*}
	\begin{aligned}
	\big\vert \cov\left((\mathcal{Z}, \theta)_0 , (\mathcal{Z}, \varphi)_0\right) &-\cov\left((\mathcal{Z}_n, \theta)_0 , (\mathcal{Z}_n, \varphi)_0\right)\big\vert \\
	&\lesssim
	\begin{cases}
	n^{-\min\left\lbrace s ;\; r ;\; (2\alpha\beta-1)\right\rbrace}\log n  
	& \text{if } (2\alpha\beta-1) = s,\\
	n^{-\min\left\lbrace s ;\; r ;\; (2\alpha\beta-1)\right\rbrace}  \log n
	& \text{if } (2\alpha\beta-1) = r  \text{ and } q>1,\\
	n^{-\min\left\lbrace s ;\; r;\; (2\alpha\beta-1)\right\rbrace}
	& \text{else}.
	\end{cases}
	\end{aligned}
	\end{equation*}
	\label{th:err_cov}
\end{theorem}
\vspace{-1em}

\begin{proof}
	The proof of this theorem is similar to the proof of \Cref{th:err_fem}, and is available in Section~\ref{proof:err_cov} of the \sm.\qed
\end{proof}

\subsection{Error analysis of the polynomial approximation}
\label{sec:pol_approx_err}

The Chebyshev polynomial approximation boils down to replacing the \psd $\gamma$ by the polynomial  $P_{\gamma,K}$ defined in \eqref{eq:p_gamma}, which approximates $\gamma$ over a segment $[0, \lambda_n^{(n)}]$ containing all the eigenvalues of the discretized operator $-\Delta_n$ (or equivalently the eigenvalues of the matrix $\bm S$). Hence, we have according to~\eqref{eq:def_Zn}
\begin{equation*}
\widehat{\mathcal{Z}}_{n,K}=\sum\limits_{k=1}^{n}W_{k} P_{\gamma,K}(\lambda_{k}^{(n)})e_{k}^{(n)} \veq
\end{equation*}
where $\lbrace W_k\rbrace_{1\le k\le n}$ are the same random weights as the ones defining $\mathcal{Z}_n$ in~\eqref{eq:def_Zn}.
The next result gives the root-mean-squared error between $\mathcal{Z}_n$ and its approximation $\widehat{\mathcal{Z}}_{n,K}$.

\begin{theorem}
		Let  Assumption \ref{assum:Lh} be satisfied, and let $\nu\in\N$ be defined as in \Cref{def:psd}, and let $\lambda_{\max}=\lambda_n^{(n)}$. Then, there exists $N_{\text{Cheb}}\in\N$ such that for any $n> N_{\text{Cheb}}$, the root-mean-squared error between the discretized field ${\mathcal{Z}}_n$ and its polynomial approximation $\widehat{\mathcal{Z}}_{n,K}$ of order $K>\nu$ is bounded by

	\begin{equation*}
	\Vert {\mathcal{Z}}_n-\widehat{\mathcal{Z}}_{n,K}\Vert_{L^2(\Omega;H)}
	\le 
	2(C_\lambda)^\nu(\pi\nu)^{-1}\text{TV}(\gamma^{(\nu)})\; n^{\alpha\nu+1/2}(K-\nu)^{-\nu}, 
	\end{equation*}
	where $\text{TV}(\gamma^{(\nu)})$ denotes the total variation over $[0, \lambda_{\max}]$ of the $\nu$-th derivative of~$\gamma$ and $\alpha>0$ and $C_\lambda >0$ are defined in \Cref{prop:weyl}.
	
	 If $\gamma$ satisfies that there exists some $\chi >0$ such that the map $z \in\mathbb{C} \mapsto \gamma(z)$ is holomorphic inside the ellipse $E_\chi\subset \mathbb{C}$ centered at $z=\lambda_{\max}/2$, with foci $z_1=0$ and $z_2=\lambda_{\max}$ and semi-major axis $a_\chi=\lambda_{\max}/2+\chi$, then, there exists $M_{\text{Cheb}}\in\N$ such that for any $n> M_{\text{Cheb}}$,
\begin{equation}
	\begin{aligned}
		\Vert {\mathcal{Z}}_n-\widehat{\mathcal{Z}}_{n,K}\Vert_{L^2(\Omega;H)}
		\le  (2C_\lambda\chi^{-1})^{1/2} \big(\sup_{z\in E_\chi}\left\vert {\gamma}\left(z\right)\right\vert\big)
		n^{(\alpha+1)/2}\exp(-(2C_\lambda\chi^{-1})^{-1/2}\; n^{-\alpha/2}K).
	\end{aligned}\label{eq:exprate}
\end{equation}
	\label{prop:conv_cheb_ap}
\end{theorem}

\begin{proof}
	Let $\lambda_{\max}=\lambda_n^{(n)}$ and let $K\in\N$. We observe first that
	\begin{equation*}
	\begin{aligned}
	\Vert {\mathcal{Z}}_n-\widehat{\mathcal{Z}}_{n,K}\Vert_{L^2(\Omega;H)}^2
	=\e\left[\Vert {\mathcal{Z}}_n-\widehat{\mathcal{Z}}_{n,K}\Vert_{0}^2 \right] 
	&=\sum\limits_{k=1}^{n}(\gamma(\lambda_{k}^{(n)})-P_{\gamma,K}(\lambda_{k}^{(n)}) )^2
	\end{aligned}
	\end{equation*}
	using the definition of $\mathcal{Z}_n$ and~$\widehat{\mathcal{Z}}_{n,K}$.
	A rather crude upper bound of this quantity is given by
	\begin{equation*}
	\begin{aligned}
	\Vert {\mathcal{Z}}_n-\widehat{\mathcal{Z}}_{n,K}\Vert_{L^2(\Omega;H)}^2
	\le n \cdot \Vert \gamma - P_{\gamma,K}\Vert_{\infty} ^2 \sveq
	\end{aligned}
	\end{equation*}
	where
	\begin{equation*}
	\Vert \gamma - P_{\gamma,K}\Vert_{\infty}=\max\limits_{\lambda \in [0, \lambda_{\max}]} \vert \gamma(\lambda)-P_{\gamma,K}(\lambda) \vert
	=\max\limits_{t \in [-1, 1]} \vert \tilde\gamma(t)-\mathcal{S}_K[\tilde{\gamma}](t) \vert
	\end{equation*}
	with $\tilde{\gamma}$ defined in \eqref{eq:chang_gamma} and $\mathcal{S}_K[\tilde{\gamma}]$ denoting the Chebyshev series of $\tilde{\gamma}$ truncated at order~$K$. If we take $K>\nu$, the convergence properties of Chebyshev series (cf. \Cref{thm:cheb}) imply that
	\begin{equation*}
	\max\limits_{t \in [-1, 1]} \vert \tilde\gamma(t)-\mathcal{S}_K[\tilde{\gamma}](t) \vert
	\le 2(\pi\nu)^{-1}(K-\nu)^{-\nu}\text{TV}(\tilde\gamma^{(\nu)}) 
	= 2^{1-\nu}(\pi\nu)^{-1}(K-\nu)^{-\nu} \lambda_{\max}^\nu \text{TV}(\gamma^{(\nu)})
	\peq
	\end{equation*}

	Under \Cref{prop:weyl}, and Assumption~\ref{assum:Lh}, we have
	\begin{equation*}
		\lambda_{\max} \le \lambda_n (1+C_1\lambda_{n}^{q-1}n^{-r})\le C_\lambda n^{\alpha} (1+C_1C_\lambda n^{\alpha(q-1)-r})
	\end{equation*}
	 which yields $\lambda_{\max}=\mathcal{O}(n^{\alpha})$ (as $n\rightarrow +\infty$)  since $\alpha(q-1)\le r$. Hence, by defining $ N_{\text{Cheb}} = \min\lbrace n\in\N : C_1C_\lambda n^{\alpha(q-1)-r} <1\rbrace$, we obtain that for any $n> N_{\text{Cheb}}$, $\lambda_{\max}\le 2C_\lambda n^\alpha$, which in turn gives 
	 \begin{equation*}
	 	\Vert {\mathcal{Z}}_n-\widehat{\mathcal{Z}}_{n,K}\Vert_{L^2(\Omega;H)}
	 	\le n^{1/2} \cdot \Vert \gamma - P_{\gamma,K}\Vert_{\infty} 
	 	\le 2(C_\lambda)^\nu(\pi\nu)^{-1}\text{TV}(\gamma^{(\nu)})\; n^{\alpha\nu+1/2}(K-\nu)^{-\nu} .
	 \end{equation*}

	 For the second inequality, using a convergence result of Chebyshev series for analytic functions (cf.\ \Cref{thm:cheb}) and the same reasoning as above, we get for any $n, K\in\N$,
	 \begin{equation*}
	 	\begin{aligned}
	 		\Vert {\mathcal{Z}}_n-\widehat{\mathcal{Z}}_{n,K}\Vert_{L^2(\Omega;H)}
	 		\le   2\epsilon_\chi^{-1} \big(\sup_{z\in E_\chi}\left\vert {\gamma}\left(z\right)\right\vert\big)
	 		{n}^{1/2} 
	 		(1+\epsilon_\chi)^{-K}
	 		\veq 
	 	\end{aligned}
	 \end{equation*}
	 where  $\epsilon_\chi>0$ is given by $\epsilon_\chi=2\lambda_{\max}^{-1}\big( \chi
	 +\sqrt{\chi(\lambda_{\max}+\chi)}\big)=h(\chi\lambda_{\max}^{-1})$, and  for $x>0$, $h(x)=2(x+\sqrt{x(1+x)})$.  In particular, for $x\in(0,1)$, we have $ 2\sqrt{x}<h(x)< 2(1+\sqrt{2})\sqrt{x}$.

	 Following  \Cref{prop:weyl} and Assumption~\ref{assum:Lh}, $\lambda_{\max}=\lambda_{n}^{(n)}\ge l_\lambda \lambda_{n} \ge l_\lambda c_\lambda n^\alpha$, 	which gives in particular $\lambda_{\max}^{-1} \le (l_\lambda c_\lambda)^{-1} n^{-\alpha}$. 
	 Let $\widehat N_{\text{Cheb}} = \min\lbrace n\in\N : 4(1+\sqrt{2})^2\chi (l_\lambda c_\lambda)^{-1} n^{-\alpha} <1\rbrace$. Then, for any $n>  \widehat N_{\text{Cheb}}$, we have $\chi\lambda_{\max}^{-1} \in (0,1)$ and  
	 $$2\sqrt{\chi\lambda_{\max}^{-1}}<\epsilon_\chi < 2(1+\sqrt{2})\sqrt{\chi\lambda_{\max}^{-1}} \le 2(1+\sqrt{2})\sqrt{\chi} (l_\lambda c_\lambda)^{-1/2} n^{-\alpha/2}<1.$$
	 Taking $n > M_{\text{Cheb}}= \max\lbrace N_{\text{Cheb}}, \widehat N_{\text{Cheb}}\rbrace$, we obtain 
	 \begin{equation*}
	 	\sqrt{2\chi  (C_\lambda)^{-1}} n^{-\alpha/2} \le 2\sqrt{\chi\lambda_{\max}^{-1}} < \epsilon_\chi <1.
	 \end{equation*}
	 Using that $x\mapsto x^{-1}(1+x)^{-K}$ is decreasing for $x\in(0,1)$ and that $\log(1+x)\ge x/2$ yields for any $n> M_{\text{Cheb}}$,
	 \begin{equation*}
	 	\epsilon_\chi^{-1}(1+\epsilon_\chi)^{-K}  
	 	\le 	(2C_{\chi,\lambda} n^{-\alpha/2})^{-1}(1+2C_{\chi,\lambda} n^{-\alpha/2})^{-K} 
	 	\le (2C_{\chi,\lambda})^{-1}n^{\alpha/2}\exp(-C_{\chi,\lambda}K n^{-\alpha/2}),
	 \end{equation*}
	 where $C_{\chi,\lambda}=\sqrt{\chi(2C_\lambda)^{-1}}$. This in turn gives
	 \begin{equation*}
	 	\begin{aligned}
	 		\Vert {\mathcal{Z}}_n-\widehat{\mathcal{Z}}_{n,K}\Vert_{L^2(\Omega;H)}
	 		\le  (C_{\chi,\lambda})^{-1} \big(\sup_{z\in E_\chi}\left\vert {\gamma}\left(z\right)\right\vert\big)
	 		n^{(\alpha+1)/2}\exp(-C_{\chi,\lambda}K n^{-\alpha/2}).
	 	\end{aligned}
	 \end{equation*}	 
	 
	 \vspace*{-2\baselineskip}\qed
	
\end{proof}

For a fixed number of degrees of freedom~$n$ in \Cref{prop:conv_cheb_ap}, the approximation error $\Vert{\mathcal{Z}}_n-\widehat{\mathcal{Z}}_{n,K}\Vert_{L^2(\Omega;H)}$ converges to $0$ as the order of the polynomial approximation $K$ goes to infinity. Choosing $K$ as a function of $n$ that grows fast enough then allows to ensure the convergence of the approximation error as $n$ goes to infinity. For instance, let us assume that $\gamma$ is once differentiable with a derivative with bounded variations (i.e., $\nu=1$ in \Cref{def:psd}), and take for simplicity $\lambda_{\max}=\lambda_{n}^{(n)}$. Assuming that \Cref{assum:Lh} is satisfied, and following \Cref{prop:weyl} yields $\lambda_{\max}=\mathcal{O}(n^\alpha)$. Taking $K=K(n)=f(n)n^{\alpha+1/2}$, where $f$ denotes any function with $\lim_{n\rightarrow \infty} f(n) = + \infty$, ensures that the approximation error $\Vert {\mathcal{Z}}_n-\widehat{\mathcal{Z}}_{n,K}\Vert_{L^2(\Omega;H)}$ goes to~$0$ at least as fast as $f$ goes to infinity. In \Cref{sec:wnf}, we provide another example for the choice of~$K$ for an analytic \psd.

In practice though, the order~$K$ of the polynomial approximation is set differently, which allows to work with relatively small orders. It is suggested in \cite{pereira2019efficient} to set $K$ by controlling the deviation in distribution between the samples obtained with and without the polynomial approximation. We propose an approach based on the numerical properties of Chebyshev series, and show in the numerical experiments that it allows to limit the approximation order.

Observe that the random weights~\eqref{eq:approx_cheb_zn} defining the Chebyshev polynomial approximation~$\widehat{\mathcal{Z}}_{n,K}$ are obtained by summing the random vectors given by  
\begin{equation*}
c_k\; T_k((2/\lambda_{\max})\bm S-\bm I) \;\bm W , \quad 0\le k\le K,
\end{equation*}
where  $\bm W \sim \mathcal{N}(\bm 0, \bm I)$ and $c_0, \dots, c_K$ are the Chebyshev series coefficients of the function~$\tilde{\gamma}$ defined in~\eqref{eq:chang_gamma}.
The Chebyshev polynomials $\lbrace T_k\rbrace_{k\in\N}$ have values in $[-1,1]$, meaning in particular that the eigenvalues of the matrices $T_k((2/\lambda_{\max})\bm S-\bm I)$ lie in the same interval. Consequently, we have for any $k\in\bi 0, K\ei$,  
\begin{equation*}
\e\big[\Vert  c_{k}T_k((2/\lambda_{\max})\bm S-\bm I)\bm W \Vert_{2}^2\big]^{1/2} 
\le \vert c_{k}\vert \e\big[\Vert \bm W\Vert_2^2\big]^{1/2} \le  \vert c_k \vert n^{1/2}\peq
\end{equation*}
Let $c_{\max}=\max\lbrace \vert c_k\vert : 0\le k\le K\rbrace$.
Since the coefficients $c_k$ converge to $0$ at least linearly for \psds (cf. \Cref{thm:cheb}), the order $K$ can be chosen to ensure that the ratio $c_K/c_{\max} \ll 1$ or that the bound $\vert c_K \vert n^{1/2} \ll 1$. Then, in practice, adding more terms to the expansion only results in negligible perturbations of the solution.

\section{Complexity analysis}\label{sec:compl}

Recall that the Galerkin--Chebyshev approximation $\widehat{\mathcal{Z}}_{n,K}$ of discretization order $n\in\N$ and polynomial order $K\in\N$ of a random field $\mathcal{Z}$ is  defined as
\begin{equation}
	\widehat{\mathcal{Z}}_{n,K}=\sum_{k=1}^n \widehat{Z}_k \psi_k \veq
	\label{eq:fem_discr_approx}
\end{equation}
where $\widehat{\bm Z}=(\widehat Z_1, \dots, \widehat Z_n)^T$ is a  Gaussian random vector with mean $\bm 0$ and covariance matrix
\begin{equation*}
	\var[\widehat{\bm Z}]=\big(\sqrt{\bm C}\big)^{-T}P_{\gamma,K}^2(\bm S)\big(\sqrt{\bm C}\big)^{-1} \veq
\end{equation*}
which can be computed by solving the linear system
\begin{equation}
	\big(\sqrt{\bm C}\big)^{T} \widehat{\bm Z} =P_{\gamma,K}(\bm S)\bm W
	\label{eq:sample}
\end{equation}
for $\bm W \sim \mathcal{N}(\bm 0, \bm I)$.
We now discuss the computational and storage cost of sampling a GRF using this approximation. In a first part, we derive these costs for the the case where nothing further is assumed about the basis $\lbrace\psi_k\rbrace_{1\le k\le n}$ used to discretize the field. In a second part, we then show how some particular choices of this basis can help to drastically improve these costs. The computational and storage costs obtained in each case are summarized in \Cref{tab:cost}. Each time, we distinguish offline computational costs, linked to operations that can be reused to generate more samples, and online computational costs steps that are specific to the computation of a given sample. In particular, we observe that the spectral method seems to perform best, but as we will see this method is rarely applicable, and we will in practice prefer the method based on linear finite elements with a mass lumping approximation which still offers overall computational costs that grow linearly with the product $Kn$ (see \Cref{sec:trunc,sec:lfem} for more details).

\begin{table}
	\centering
	\begin{tabular}{ |C{8em}|c|c|c| } 
		\hline
		 & Offline computational costs & Online computational costs & Storage costs \\ 
		 \hline
		General case & $\mathcal{O}(n^3 + K\log K)$ & $\mathcal{O}(Kn^2)$ & $\mathcal{O}(n^2+K)$ \\ 
		\hline
		 Spectral method & $0$ & $\mathcal{O}(n)$ & $\mathcal{O}(n)$  \\ 
		 \hline
		 Linear finite elements + Cholesky & $\eta_{\text{Chol}}(\bm C)+\mathcal{O}(\mu n+ K\log K)$ & $\mathcal{O}(K\mu n)$ & $\mathcal{O}(\mu n+K)$ \\ 
		 \hline
		 Linear finite elements + Mass Lumping & $\mathcal{O}( K\log K)$ & $\mathcal{O}(K\mu n)$ & $\mathcal{O}(\mu n+K)$\\
		\hline
	\end{tabular}
\caption{Comparison of computational and storage costs for computing a GRF sample from a Galerkin--Chebyshev approximation of discretization order $n\in\N$ and polynomial order $K\in\N$,  for various choices of discretization basis.  The parameter $\mu$ is an upper bound for the mean number of nonzero entries in $\sqrt{\bm C}$ and $\bm R$, and $\eta_{\text{Chol}}(\bm C)$ the computational cost of a Cholesky factorization of~$\bm C$.}
\label{tab:cost}
\end{table}

\subsection{Efficient sampling: general case}\label{sec:samp_gen}

Generating samples of the weights $\widehat{\bm Z}$ in~\eqref{eq:sample} requires two steps:
\begin{itemize}
	\item first, one computes the vector $\widehat{\bm X}=P_{\gamma,K}(\bm S)\bm W$ for some $\bm W \sim \mathcal{N}(\bm 0, \bm I)$. Due to the fact that $P_{\gamma,K}$ is a polynomial, this step can be implemented as an iterative program involving at each step only one matrix-vector product between $\bm S$ and a vector;
	\item then, one solves the linear system $\big(\sqrt{\bm C}\big)^{T} \widehat{\bm Z} =\widehat{\bm X}$.
\end{itemize}
In order to execute these two steps, one only needs to implement the following two sub-algorithms:
\begin{itemize}
	\item an algorithm $\bm \Pi_{\bm S}$ taking as input a vector $\bm x$ and returning the product $\bm \Pi_{\bm S}(\bm x)=\bm S \bm x$;
	\item an algorithm $\bm{\Pi}_{(\sqrt{\bm C})^{-T}}$ taking as input a vector $\bm x$ and returning the solution $\bm y=\bm \Pi_{(\sqrt{\bm C})^{-T}}(\bm x)$ to the linear system
	$\big(\sqrt{\bm C}\big)^{T} \bm y = \bm x \peq$
\end{itemize}
We present in Algorithm 1 of the · the overall algorithm leading to sampling the weights of the decomposition defined in~\eqref{eq:fem_discr_approx}  using this approach.

Following the definition of $\bm S$ in \Cref{eq:def_S}, $\bm\Pi_{\bm S}$ does not require the matrix $\bm S$ to be computed explicitly and stored: a product by $\bm S$ boils down to solving a first linear system defined by $(\sqrt{\bm C})^{T}$, multiplying the obtained solution by $\bm R$ and then solving a second linear system defined by $\sqrt{\bm C}$.
Hence, both $\bm{\Pi}_{(\sqrt{\bm C})^{-T}}$ and $\bm\Pi_{\bm S}$ rely on solving linear systems involving a square-root of the mass matrix $\bm C$ (or its transpose). The cost associated with calls to $\bm{\Pi}_{(\sqrt{\bm C})^{-T}}$ and $\bm\Pi_{\bm S}$ should be kept minimal in order to reduce the overall  computational complexity of the sampling algorithm. 

Since the choice of this square-root is free, one could take it as the Cholesky factorization of $\bm C$ satisfying $\sqrt{\bm C} = \bm L$
for some lower-triangular matrix~$\bm L$. Solving a linear system involving $\bm L$ or $\bm L^T$ can be done at roughly the cost of a matrix-vector product using forward or backward substitution. The algorithms $\bm\Pi_{\bm S}$  and $\bm{\Pi}_{(\sqrt{\bm C})^{-T}}$ resulting from this choice are presented in Algorithms 2 and 3 of the \sm. 
Regarding the computational complexity of these algorithms, since solving a linear system using forward or backward substitution can be done with a computational cost of the same order as a matrix-vector product (namely $\mathcal{O}(n^2)$ operations), each call to $\bm\Pi_{\bm S}$ or $\bm{\Pi}_{(\sqrt{\bm C})^{-T}}$  amounts to $\mathcal{O}(n^2)$ operations. This means that, if implementations of these two algorithms are available, the cost of computing the weights $\widehat{\bm Z}$ in~\eqref{eq:sample} is of order $\mathcal{O}(Kn^2)$, where $K$ corresponds to the order of the polynomial approximation. 

Finally, recall that one needs an upper bound $\lambda_{\max}$ of the largest eigenvalue of~$\bm S$ in order to define the polynomial $P_{\gamma, K}$. This upper bound can be obtained with a limited computational cost (namely $\mathcal{O}(n^2)$ operations) by combining the Gershgorin circle theorem~\cite{gerschgorin1931uber} and a power iteration scheme (as described in Section~\ref{sec:upper_bound} of the \sm).

Overall, the computational cost of sampling the weights of the Galerkin--Chebyshev approximation $\widehat{\mathcal{Z}}_{n,K}$ in~\eqref{eq:fem_discr_approx} can be summarized as follows. We can distinguish between offline and online steps. 
The offline steps are as follows.  First, there is the computation of the coefficients of the Chebyshev approximation $P_{\gamma, K}$, which requires $\mathcal{O}(K\log K)$ operations as mentioned in the previous subsection. Then, there is  the Cholesky factorization of $\bm C$, which requires $\mathcal{O}(n^3)$ operations \citep[Chapter 2]{press2007numerical}. And finally, there is the computation of the upper bound of the eigenvalues of $\bm S$, which requires $\mathcal{O}(n^2)$ operations (dominated by the use of the power iteration scheme). 
The online step is the computation of the weights according to~\eqref{eq:sample}, which requires $\mathcal{O}(Kn^2)$ operations. Storage-wise, this workflow only requires enough space to store the Cholesky factorization of the mass matrix~$\bm C$, the stiffness matrix $\bm R$, the $K+1$ coefficients of the Chebyshev polynomial approximation, and a few vectors of size $n$.  In conclusion, the offline costs are of order $\mathcal{O}(K\log K+n^3)$, the online costs are of order $\mathcal{O}(Kn^2)$, and the storage needs are of order $\mathcal{O}(n^2+K)$. 
 As we will see in the next section, both computational and storage costs can be reduced for typical choices of the discretization space $V_n$.

\subsection{Efficient sampling: Particular cases}\label{sec:part_cases}

The choice of the space $V_n$ used to discretize the random fields impacts heavily the mass and stiffness matrices, and can in relevant cases be leveraged to speed up the sampling process. We provide here two examples, which will be considered later on in the numerical experiments.

\subsubsection{Spectral approximation} 
\label{sec:trunc}

If we assume that the eigenvalues of the Laplace--Beltrami operator are known, we can use spectral methods, which correspond to the case where $V_n$ is  the set of eigenfunctions associated with the first $n$ eigenvalues of the Laplace--Beltrami operator.
Then, since the eigenfunctions are orthonormal, the mass matrix $\bm C$ is equal to the identity matrix. Besides, using Green's theorem, we have that the stiffness matrix $\bm R$ is also diagonal, with entries equal to the operator eigenvalues. This gives that $\bm S=\bm R$ is diagonal.

Thus, sampling the weights of $\widehat{\mathcal{Z}}_{n,K}$  can be done without requiring any Cholesky factorization: calls to $\bm\Pi_{\bm S}$ are replaced by multiplication by the diagonal matrix $\bm R$ containing the eigenvalues of the operator, calls to $\bm\Pi_{(\sqrt{\bm C})^{-T}}$ are replaced by products with an identity matrix, and the upper bound $\lambda_{\max}$ is replaced by the maximal entry of $\bm R$. In particular, the offline costs are reduced to the computation of the coefficients of $P_{\gamma, K}$, and the online costs are reduced to $\mathcal{O}(n)$. As for the storage needs, they would now be reduced to $\mathcal{O}(n)$ (since both $\bm C$ and $\bm R$ are diagonal).

In practice though, the Chebyshev polynomial approximation is not necessary. One can directly use \Cref{th:cov_weights} to compute samples of $\mathcal{Z}_n$ (and therefore there is no need to approximate it by $\widehat{\mathcal{Z}}_{n,K}$): $\bm S$ being now diagonal, the matrix $\gamma^2(\bm S)$ is the diagonal matrix obtained by directly applying $\gamma^2$ to the diagonal entries of $\bm S$.  Samples of $\mathcal{Z}_n$ are then obtained by taking  the weights $\bm Z$  as a sequence of independent Gaussian random variables with variances given by the diagonal entries of $\gamma^2(\bm S)$ (since $\bm C$ is the identity matrix).
 In conclusion, no offline costs are needed for the spectral method, the online costs are of order $\mathcal{O}(n)$, and the storage needs are of order $\mathcal{O}(n)$. 
 
These computational costs might seem ideal, but one should remember that the spectral method is only applicable when the eigenfunctions and eigenvalues of the  Laplace--Beltrami operator are known. This is the case for instance when working on rectangular Euclidean domains, for which the eigenfunctions correspond to the Fourier basis, and we retrieve the classical spectral methods, or for the sphere, for which the eigenfunctions are the spherical harmonics, see \Cref{sec:num} for more details). For other choices of compact Riemannian manifolds, these are unknown, which is why we propose the next method relying on the finite element method.

\subsubsection{Linear finite element spaces}
\label{sec:lfem}

Consider the case where $V_n$ is taken to be a finite element space of (piecewise) linear functions associated with a simplicial mesh of the manifold $\mathcal{M}$. In this case, the basis functions composing $V_n$ have a support limited to a few elements of the mesh, and the matrices $\bm C$ and $\bm R$ are therefore sparse. Besides, for uniform meshes, one can bound the number of nonzero entries in each row of these matrices. Such sparsity can be leveraged to reduce the cost associated with sample generation. 

The cost $\eta_{\text{Chol}}(\bm C)$ of the Cholesky factorization now depends on the number of nonzero entries of $\bm C$, and adequate permutations can be found to ensure that the factors are themselves sparse. This cost is of course upper-bounded by the cost associated with the Cholesky factorization of a dense matrix, i.e., $\mathcal{O}(n^3)$, but in practice the sparsity of the matrix is leveraged to achieve a lower computational cost.
 Consequently, the costs associated with calling $\bm\Pi_{\bm S}$ or $\bm\Pi_{(\sqrt{\bm C})^{-T}}$ are reduced to an order $\mathcal{O}(\mu n)$, where $\mu \ll n$ denotes an upper bound for the mean number of nonzero entries in $\sqrt{\bm C}$ and $\bm R$. This means in particular that the computational cost of computing the weights through~\eqref{eq:sample} drops to $\mathcal{O}(K\mu n)$ operations.  Similarly, using the same approach as the one described in \Cref{sec:samp_gen}, the upper bound $\lambda_{\max}$ can be computed in $\mathcal{O}(\mu n)$ operations. In conclusion, the offline costs are of order $\eta_{\text{Chol}}(\bm C)+\mathcal{O}(\mu n +K\log K)$, the online costs are of order $\mathcal{O}(K\mu n)$, and the storage needs are of order $\mathcal{O}(\mu n+K)$.

In practice, an additional approximation can be made to further reduce the computational cost of the algorithm. As advocated by \citet{lindgren2011explicit}, the mass matrix $\bm C$ can be replaced by a diagonal approximation $\widehat{\bm C}$ whose entries are given by
	\begin{equation*}
		\widehat{C}_{ii}= \left(\psi_i, 1\right)_0, \quad i\in \bi 1, n\ei. 
	\end{equation*}
	This approach results in a Markovian approximation of the random field, and is inspired from the lumped mass approximation proposed by \citet{chen1985lumped} for parabolic PDEs. On Euclidean domains, this approach introduces an error in the covariance of the resulting field of order $\mathcal{O}(h^2)$ where $h$ is the mesh size, which, for a uniform mesh, is linked to the dimension $n$ of the finite element space as $n=\mathcal{O}(h^{-d})$. We show in the numerical experiments in \Cref{sec:num} that this additional error does not affect the theoretical convergence rates derived in \Cref{sec:err}.

Following the lumped mass approach, the square-root $\sqrt{\bm C}$ currently computed as a Cholesky factor, is replaced by the square-root $\widehat{\bm C}^{1/2}$ of $\widehat{\bm C}$, which is the diagonal matrix obtained by taking the square-root of the entries of $\widehat{\bm C}$. This completely eliminates the need for a Cholesky factorization. Also the linear system previously solved by substitution can be trivially solved in linear time since the matrix is diagonal. As for the upper bound $\lambda_{\max}$ it can be computed directly without requiring a power iteration method.
Then, the offline costs of our approach drop to $\mathcal{O}(K\log K)$ and the online costs are of order $\mathcal{O}(K\mu n)$.  As for the storage needs, they are reduced to $\mathcal{O}(\mu n)$ (since both $\bm C$ and $\bm R$ are sparse).  These costs are drastically reduced compared to the costs associated with the naive approach presented at the beginning of \Cref{sec:motiv_cheb}, which consisted of a storage need of $\mathcal{O}(n^2)$ and a computational complexity of $\mathcal{O}(n^3)$ operations. The storage costs now grow linearly with $n$, and the computational costs grow linearly with $K$ and $n$, hence rendering the algorithm much more scalable.

\subsection{Application: Simulation of Whittle--Matérn fields}\label{sec:wnf}

To conclude this section, we provide an application of the convergence results in \Cref{sec:err} and of the computational complexities derived in this section to the approximation of Whittle--Matérn random fields, i.e., fields with a \psd given by~\eqref{eq:gamma_beta}).

\begin{corol}
	Let  \Cref{assum:Lh} be satisfied,  and let $\gamma$
	be given by~\eqref{eq:gamma_beta}.
	Then, the approximation error of the random field $\mathcal{Z}$ by its Galerkin--Chebyshev polynomial approximation $\widehat{\mathcal{Z}}_{n,K}$ of order $K\in\N$, satisfies
	\begin{equation}
		\Vert \mathcal{Z} - \widehat{\mathcal{Z}}_{n,K} \Vert_{L^2(\Omega; H)} 
		\le  C_{\text{Galer}}\;n^{-\rho}+ C_{\kappa,\lambda}(2^{-1}\kappa^{2})^{-\beta} \;
		n^{(\alpha+1)/2}\exp(-(C_{\kappa,\lambda})^{-1}\; n^{-\alpha/2}K)
		\label{eq:err_matern}
	\end{equation}
	where $C_{\text{Galer}}$ is a constant independent of $n$ and $K$, $\rho=\min\left\lbrace s ;\; r;\; (\alpha\beta-1/2)\right\rbrace>0$ and  $C_{\kappa,\lambda}=2C_\lambda^{1/2}\kappa^{-1}$, $\alpha>0$ and $C_\lambda>0$ are defined in \Cref{prop:weyl}, $\kappa>0$ and  $\beta>0$ are as in~\eqref{eq:gamma_beta}, and $r>0$ and $s>0$ are given in \Cref{assum:Lh}.
	
	In particular, 
	there exist $\epsilon_0, C_1, C_2>0$ (depending only on $\gamma$, $C_{\text{Galer}}$ and the constants defined in  \Cref{prop:weyl} and Assumption \ref{assum:Lh}) such that for any $\epsilon\in (0,\epsilon_0)$, taking $n=(C_1)^{1/\rho}\;\epsilon^{-1/\rho}$ and $K=\big\lceil C_{2}\;\epsilon^{-\alpha/2\rho}\vert \log \epsilon\vert\big\rceil$ yields
	\begin{equation*}
		\Vert \mathcal{Z} - \widehat{\mathcal{Z}}_{n,K} \Vert_{L^2(\Omega; H)}
		\le  \epsilon.
	\end{equation*}
	\label{corol:conv_cheb_ap}
\end{corol}

\begin{proof}
	To ease the reasoning, let us consider once again that the upper-bound $\lambda_{\max}$ corresponds exactly to the maximal eigenvalue of $\bm S$. The inequality~\eqref{eq:err_matern} follows directly from~\Cref{th:err_all_ana}, after noting that $z\in\mathbb{C}\mapsto\gamma(z)$ is holomorphic in the ellipse centered at $z=\lambda_{\max}/2$, with foci $z_1=0$ and $z_2=\lambda_{\max}$, and semi-major axis $a=\lambda_{\max}/2+\kappa^2/2$ (i.e., $\chi=\kappa^2/2$ in \Cref{th:err_all_ana}), and that $\vert \gamma\vert$ can be bounded in this ellipse by $(\kappa^{2}/2)^{-\beta}$.

	 The error $\Vert \mathcal{Z} - \widehat{\mathcal{Z}}_{n,K} \Vert_{L^2(\Omega; H)}$ satisfies the inequality
	\begin{equation*}
		\Vert \mathcal{Z} - \widehat{\mathcal{Z}}_{n,K} \Vert_{L^2(\Omega; H)} \le E_{\text{Galer}} + E_{\text{Cheb}},
	\end{equation*}
	where $E_{\text{Galer}}=C_{\text{Galer}}\; n^{-\rho}$ denotes the contribution to the error estimate due to the Galerkin approximation, and $E_{\text{Cheb}}$ the contribution due to the Chebyshev approximation, i.e.,
	\begin{equation*}
		E_{\text{Cheb}}= C_{\kappa,\lambda}(2^{-1}\kappa^{2})^{-\beta} \;
		n^{(\alpha+1)/2}\exp(-(C_{\kappa,\lambda})^{-1}\; n^{-\alpha/2}K).
	\end{equation*}
	Let $\epsilon \in(0,\epsilon_{\max})$ where $\epsilon_{\max} =\sup\lbrace \epsilon \in (0,1): (C_{\text{Galer}})^{1/\rho}\;\epsilon^{-1/\rho} > M_{\text{Cheb}} \rbrace$ and $M_{\text{Cheb}}$ is defined in \Cref{prop:conv_cheb_ap}. Let $n=\lceil(C_{\text{Galer}})^{1/\rho}\;\epsilon^{-1/\rho}\rceil$.
Then, $n > M_{\text{Cheb}}$ and $E_{\text{Galer}}=\epsilon$. Let $r_{\rho,\alpha}=(1+(2\rho)^{-1}(\alpha+1))$, $C_{\kappa,\lambda,\text{Galer}}=C_{\kappa,\lambda} (C_{\text{Galer}})^{\alpha/2\rho}r_{\rho,\alpha}$, and take $K=\big\lceil C_{\kappa,\lambda,\text{Galer}}\;\epsilon^{-\alpha/2\rho}\vert \log \epsilon\vert\big\rceil$. Thus,
	\begin{equation*}
		\begin{aligned}
					E_{\text{Cheb}} 
			&\le C_{\text{Cheb}}\;\epsilon^{-(\alpha+1)/2\rho}\exp\big(-r_{\rho,\alpha}\vert \log \epsilon\vert\big) = C_{\text{Cheb}}\;\epsilon^{r_{\rho,\alpha}-(\alpha+1)/2\rho}=C_{\text{Cheb}}\;\epsilon\\
		\end{aligned}
	\end{equation*}
	 where $C_{\text{Cheb}}=C_{\kappa,\lambda}(2^{-1}\kappa^{2})^{-\beta} \;
	 (C_{\text{Galer}})^{(\alpha+1)/2\rho}$ and we used the fact the $\epsilon\le1$.
	 
	 	In conclusion, let $\tilde{\epsilon} \in(0, \epsilon_0)$ where $\epsilon_0=(1+C_{\text{Cheb}})\epsilon_{\max}$ and let $\epsilon=\tilde{\epsilon}(1+C_{\text{Cheb}})^{-1}$.
	 Then,  $\epsilon \in(0,\epsilon_{\max})$, and when taking $\lceil(C_{\text{Galer}})^{1/\rho}\;\epsilon^{-1/\rho}\rceil$ and $K=\big\lceil C_{\kappa,\lambda,\text{Galer}}\;\epsilon^{-\alpha/2\rho}\vert \log \epsilon\vert\big\rceil$, we end up with an error $\Vert \mathcal{Z} - \widehat{\mathcal{Z}}_{n,K} \Vert_{L^2(\Omega; H)}
	 \le  (1+ C_{\text{Cheb}})\epsilon = \tilde{\epsilon}$. \qed
\end{proof}

As a consequence, we can derive the computational cost required to sample a GRF with root-mean-squared error $\Vert \mathcal{Z} - \widehat{\mathcal{Z}}_{n,K} \Vert_{L^2(\Omega; H)}$ smaller than  some small $\epsilon$ by taking $n=(C_1)^{1/\rho}\;\epsilon^{-1/\rho}$ and $K=\big\lceil C_{2}\;\epsilon^{-\alpha/2\rho}\vert \log \epsilon\vert\big\rceil$ in the estimates in \Cref{tab:cost}. We end up with the bounds in \Cref{tab:cost_matern}. We also provide the computational cost associated with the choice of a linear finite element and mass lumping approximation. This method introduces an additional error term due to the mass lumping approximation but in practice does not seem to affect the theoretical convergence rates of the root-mean-squared error, which allows us to think that we can still carry out the analysis leading to \Cref{corol:conv_cheb_ap} (and therefore to the estimates in \Cref{tab:cost_matern}) in this case).  We finally observe that a Galerkin--Chebyshev approximation of a Whittle--Matérn field with a root-mean-squared error bounded by $\epsilon >0$ can be asymptotically obtained with a computational cost $\mathcal{O}(\mu\epsilon^{-(\alpha/2+1)/\rho}\vert\log\epsilon\vert)$ using linear finite elements with a mass lumping approximation.

\begin{table}
	\centering
	\begin{tabular}{ |C{8em}|c|c|c| } 
		\hline
		& Offline computational costs & Online computational costs & Storage costs \\ 
		\hline
		General case & $\mathcal{O}(\epsilon^{-3/\rho} )$ & $\mathcal{O}(\epsilon^{-(\alpha/2+2)/\rho}\vert\log\epsilon\vert)$ & $\mathcal{O}(\epsilon^{-2/\rho}  )$ \\ 
		\hline
		Spectral method & $0$ & $\mathcal{O}(\epsilon^{-1/\rho})$ & $\mathcal{O}(\epsilon^{-1/\rho})$  \\ 
		\hline
		Linear finite elements + Cholesky & $\eta_{\text{Chol}}(\bm C)+\mathcal{O}(\mu\epsilon^{-1/\rho} )$ & $\mathcal{O}(\mu\epsilon^{-(\alpha/2+1)/\rho}\vert\log\epsilon\vert)$ & $\mathcal{O}(\mu \epsilon^{-1/\rho})$ \\ 
		\hline
		Linear finite elements + Mass Lumping & $\mathcal{O}(\epsilon^{-\alpha/2\rho}\vert\log(\epsilon)\vert^2 )$ & $\mathcal{O}(\mu\epsilon^{-(\alpha/2+1)/\rho}\vert\log\epsilon\vert)$ & $\mathcal{O}(\mu \epsilon^{-1/\rho})$\\
		\hline
	\end{tabular}
\caption{Comparison of computational and storage costs for computing a Galerkin--Chebyshev approximation of a Whittle--Matérn field with a root-mean-squared error bounded by $\epsilon >0$. The parameter $\mu$ is an upper bound for the mean number of nonzero entries in $\sqrt{\bm C}$ and $\bm R$ and $\eta_{\text{Chol}}(\bm C)$ the computational cost of a Cholesky factorization of $\bm C$.}
	\label{tab:cost_matern}
\end{table}

\section{Numerical experiments}\label{sec:num}

In this section we confirm the convergence estimates derived in \Cref{sec:err} using numerical experiments. In a first subsection, we restrict ourselves to the specific case where the Riemannian manifold of interest $(\mathcal{M}, g)$ is the $2$-sphere endowed with its canonical metric, as in this case the eigenvalues and eigenvectors of the Laplace--Beltrami are known, and hence the exact solution can be computed and compared to the various approximations introduced in this work. In a second subsection, we investigate the case where the Riemannian manifold of interest is a hyperboloid, for which, even though the  the eigenvalues and eigenvectors of the Laplace--Beltrami are not known, we are still able retrieve the error estimate for the covariance.

\subsection{Numerical experiments on the sphere}

Recall that the Laplace--Beltrami operator $-\Delta_{\mathcal{M}}$ on the $2$-sphere has eigenvalues $\lambda_{l,m}$ given by $\lambda_{l,m}=l(l+1)$ for $l\in\mathbb{N}$, $m\in\bi -l, l\ei$,
with associated eigenfunctions given by the (real) spherical harmonics $Y_{l,m}$ defined in spherical coordinates $\theta\in [0,\pi], \phi\in[0,2\pi)$ by the expression
\begin{equation*}
Y_{l,m}(\theta, \phi)=(-1)^m\,2^{1/2}\bigg(\frac{2l+1}{4\pi}\frac{(l-\vert m\vert)!}{(l+\vert m\vert)!}\bigg)^{1/2}P_l^{\vert m \vert}(\cos\theta) Q_m(\phi) \times \begin{cases}
\sin(\vert m \vert \phi) & \text{if } m<0, \\
\cos(\vert m \vert \phi) & \text{if } m\ge 0,
\end{cases}
\end{equation*}
where for $l\in\N$, $m\in\bi 0, l\ei$, $P_l^m$ denotes the associated Legendre polynomial with indices $l$ and $m$. In the remainder of this section, we use {$\alpha = 2/d=1$ by} Weyl's asymptotic law in \Cref{prop:weyl}.

On the sphere, the Gaussian random fields defined using functions of the Laplacian $\gamma(-\Delta_\mathcal{M})$ as in~\eqref{eq:def_z} are particular instances of the class of isotropic random fields on the sphere described in~\cite{lang2015isotropic}. The covariance $C(\theta)$ of such fields between any two points on the sphere is linked to the spherical distance $\theta$ separating the points through the relation
\begin{equation}
C(\theta)=\sum_{l=0}^{\infty} \frac{2l+1}{4\pi}\gamma(l(l+1))^2 P_l(\cos\theta), \quad \theta\in [0,\pi],
\label{eq:covth}
\end{equation}
where $P_l$, $l\in\mathbb{N}_0$, denotes the Legendre polynomial of order~$l$.

Finally, we restrict our numerical experiments to Whittle--Mat\'ern fields by considering \psds of the form $\gamma(\lambda)=\vert \kappa^2 + \lambda\vert^{-\beta}$ for $\lambda \ge 0$ and
some parameters $\kappa >0$, $\beta>1/2$. %
We introduce an additional parameter $a$, which we call practical range, and which is defined from the parameters $\nu$ and $\kappa$ by $a = 3.6527\kappa^{-1}\nu^{0.4874}$.
In the remainder of this section, the \psds $\gamma$ will be characterized by choices of the parameters $\nu$ and $a$. The rationale behind the parameter $a$ comes from numerical experiments conducted in \cite{romary2008inversion} which showed that the correlation range of the Mat\'ern covariance function (on $\R^2$) is very-well approximated by $a$, thus yielding a rule-of-thumb for choosing~$\kappa$.

\begin{figure}
	\centering
	\includegraphics[width=0.45\textwidth]{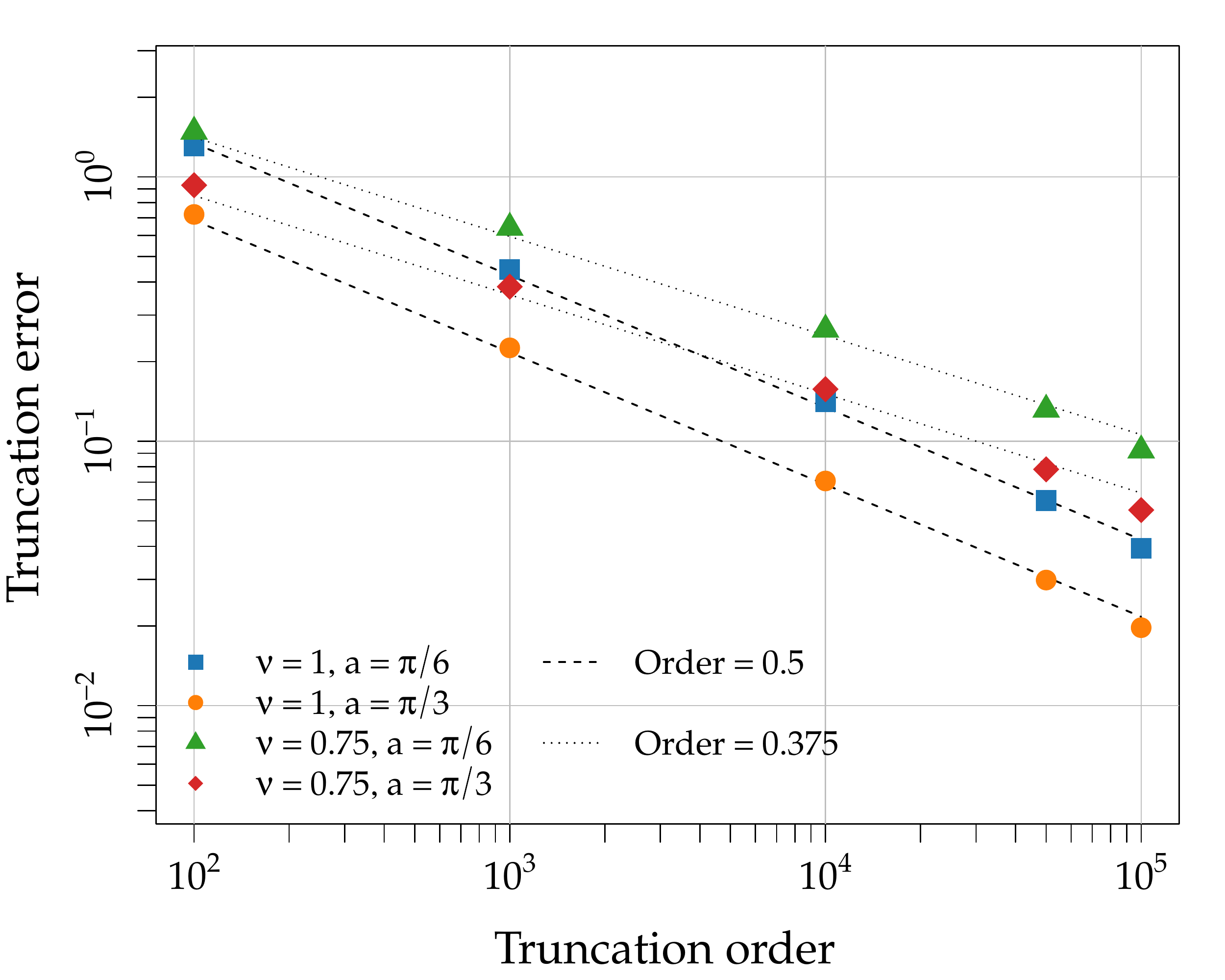}
	\caption{Truncation error $\left\Vert \mathcal{Z} - \mathcal{Z}^{(n)} \right\Vert_{L^2(\Omega; H)}$ on the sphere.\label{fig:trunc}}
\end{figure}

We now present the result obtained when computed numerically the truncation error, and the covariance error.  Results on the error due to the polynomial approximation can be found in Section~\ref{sec:pol_num} of the \sm.

\subsubsection{Truncation error}\label{sec:trunc_exp} 

We look at the truncation error $\Vert \mathcal{Z} - \mathcal{Z}^{(n)} \Vert_{L^2(\Omega; H)}$ between the full expansion $\mathcal{Z}$ and its truncation $\mathcal{Z}^{(n)}$ at order $n\in\N$, for various choices of~$n$. 
This error corresponds to the error term derived in \Cref{th:err_fem} when the discretization space $V_n$ is the set of the first $n$ eigenvalues of the Laplace--Beltrami operator (cf. \Cref{sec:trunc}).
In this case, \Cref{assum:Lh} holds for arbitrary large values of the exponents $r,s>0$ and we therefore expect a convergence of order $\alpha\beta-1/2=\nu/2$.

We compute truncation errors for the \psds $\gamma$ given by
\begin{itemize}
	\item $\nu=0.75$, $a\in\lbrace \pi/6, \pi/3\rbrace$, yielding an expected convergence of order $0.375$;
	\item $\nu=1$, $a\in\lbrace \pi/6, \pi/3\rbrace$, yielding an expected convergence of order~$0.5$;
\end{itemize}
and consider truncation orders $n\in\lbrace 10^2, 10^3, 10^4, 5\cdot 10^4, 10^5\rbrace$. Samples of the corresponding truncated fields are generated using the approach presented in \Cref{sec:trunc}.

The error  $\Vert \mathcal{Z} - \mathcal{Z}^{(n)} \Vert_{L^2(\Omega; H)}$ is approximated by a  Monte Carlo estimate taking the form
\begin{equation*}
\big\Vert \mathcal{Z} - \mathcal{Z}^{(n)} \big\Vert_{L^2(\Omega; H)} 
\approx \bigg({\frac{1}{N_{\text{simu}}}\sum_{k=1}^{N_{\text{simu}}}\big\Vert \mathcal{Z}^{(N_{\max})}_k - \mathcal{Z}^{(n)}_k \big\Vert_{0}^2\bigg)^{1/2}
},
\end{equation*}
where for any $k$,  $\mathcal{Z}^{(N_{\max})}_k$ is an independent realization of the truncation of $\mathcal{Z}$ at a very high order $N_{\max}=10^6$, and $\mathcal{Z}^{(n)}_k$ is a truncation of $\mathcal{Z}^{(N_{\max})}_k$ at order $n$. The number of samples used for this study is $N_{\text{simu}}=500$, which is sufficient as larger choices of~$N_{\max}$ have little impact on the results.  
The results are presented in \Cref{fig:trunc} and show that the theoretical orders of convergence are systematically retrieved.

\subsubsection{Covariance error and computational cost}\label{sec:cov_num}

The covariance error refers to the absolute error in covariance between the model random field and its approximation used in practice. We take here the discretization space $V_n$ to be the finite element space of piecewise linear functions defined on a polyhedral approximation of the sphere with triangular faces, hence following the surface finite element (SFEM) approach \cite{dziuk1988finite}.

We generate $10^6$ samples of the random field $\widehat{\mathcal{Z}}_{n,K}$ while considering finite element spaces defined on gradually refined polyhedral approximations of the sphere. For each choice of parameter defining the spectral density $\gamma$, we set the order $K$ of the polynomial approximation using the approach described in \Cref{sec:pol_approx_err}, with a criterion $\vert c_K/c_{\max}\vert<10^{-12}$. 
The covariance error we compute is given as an error between the covariance functions of the field $\mathcal{Z}$ and its approximation $\widehat{\mathcal{Z}}_{n,K}$. The former is given in~\eqref{eq:covth} and the latter is approximated by a Monte Carlo estimator. The overall error between both covariance functions is then evaluated as the maximum absolute error between their evaluations on a grid of $500$ equi\-spaced points in $(0,\pi)$ along a great circle. The covariance errors are presented in \Cref{fig:cov}, and show that the theoretical convergence rate $2\alpha\beta -1 = \nu$ is confirmed.

Finally, we present the order of polynomial approximation in \Cref{fig:poly_approx} and the associated computation time needed to generate the samples used to compute the covariance errors in \Cref{fig:comp_time}. We observe that although the order of the polynomial approximation grows, the computation time remains small with less than half a second.

\begin{figure}[H]
	\centering
	\begin{subfigure}[t]{.45\textwidth}
		\centering
		\includegraphics[width=\textwidth]{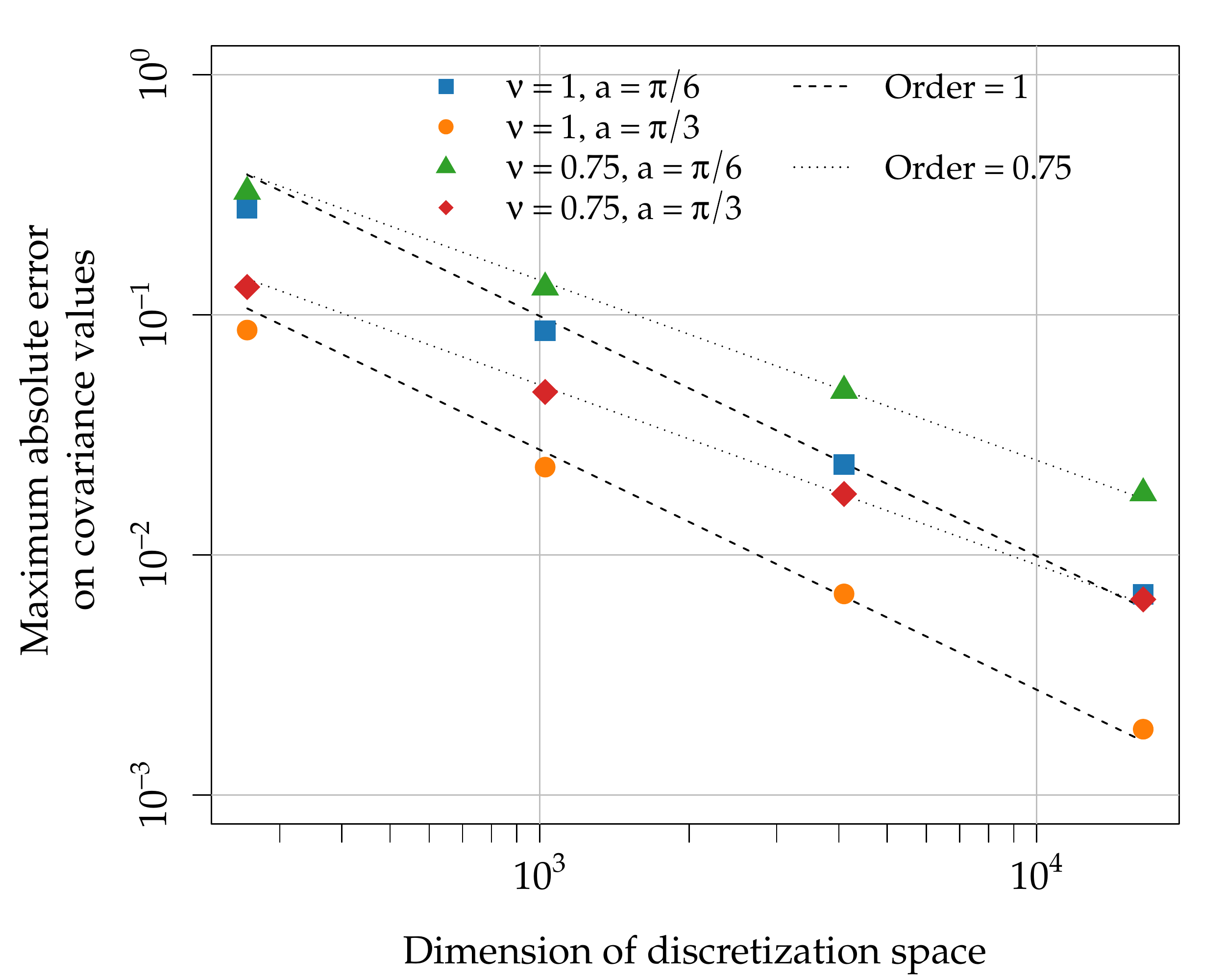}
		\caption{Maximum absolute error between the covariance functions of $\mathcal{Z}$ and  $\widehat{\mathcal{Z}}_{n,K}$.}
		\label{fig:cov}
	\end{subfigure} \\
	\begin{subfigure}{.45\textwidth}
		\centering
		\includegraphics[width=\textwidth]{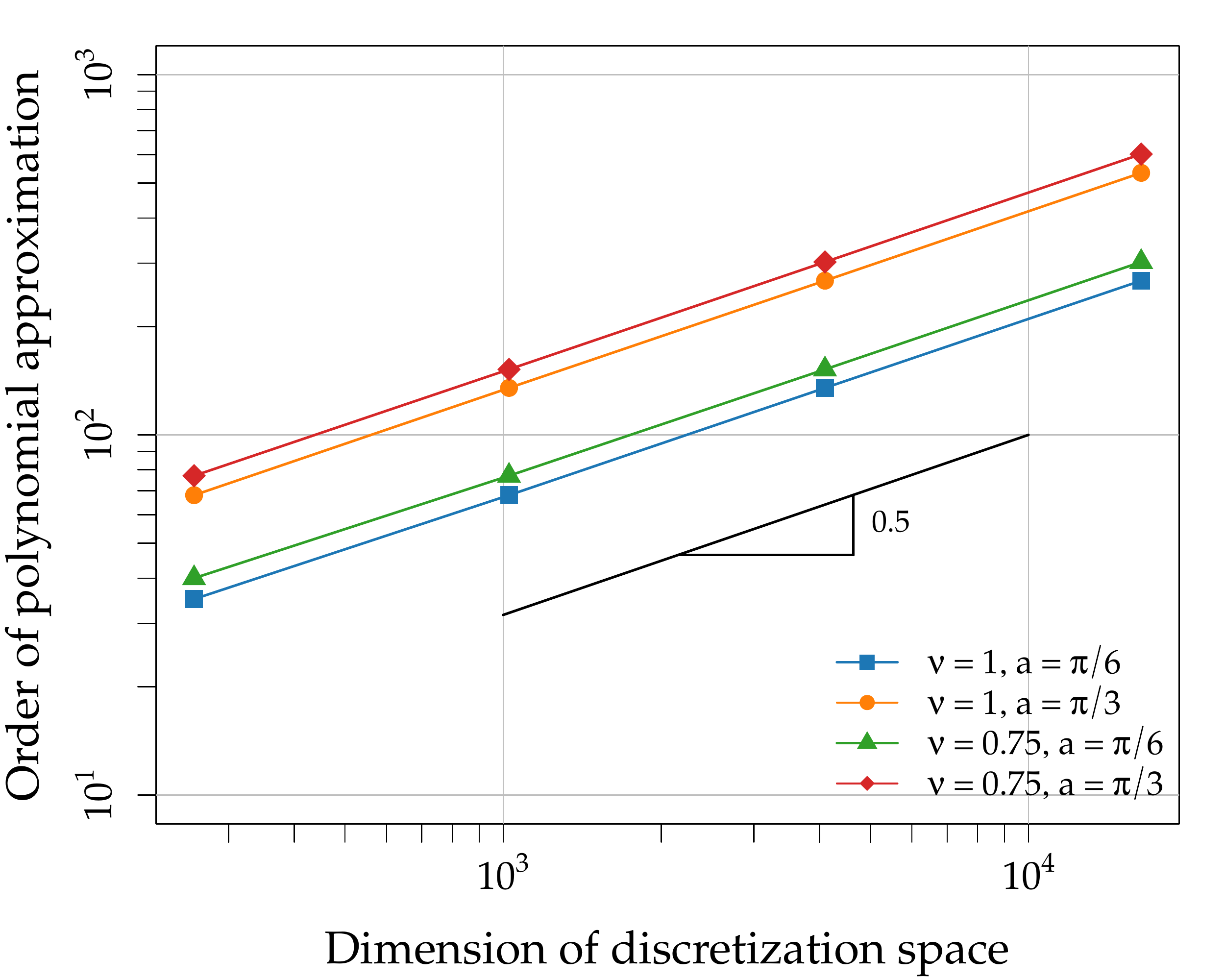}
		\caption{Order of the polynomial approximation used.\label{fig:poly_approx}}
	\end{subfigure}%
	\hfill%
	\begin{subfigure}{.45\textwidth}
		\centering
		\includegraphics[width=\textwidth]{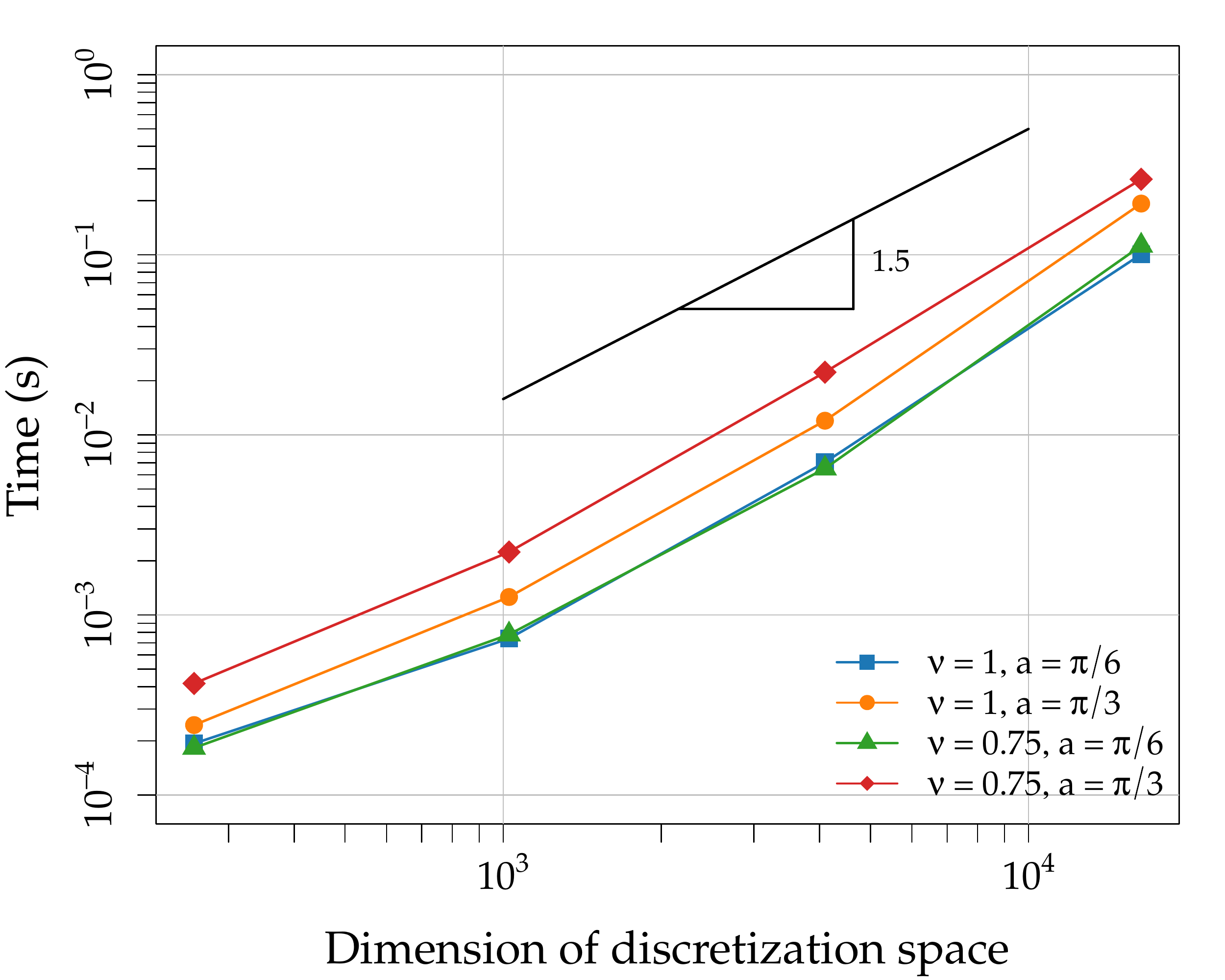}
		\caption{Time needed.\label{fig:comp_time}}
	\end{subfigure}
	\caption{Approximation order used and time needed to compute one sample of the (approximated) Gaussian random field.\label{fig:effort}}	
\end{figure}

\subsection{Numerical experiment on a hyperboloid}

In this section we confirm the error estimate from \Cref{th:err_cov} numerically on a hyperboloid surface. We consider the two-dimensional surface defined implicitly by the equation
\begin{equation*}
	\mathcal{M} = \lbrace (x,y,z)\in\R^3  : x^2+y^2-z^2=1 \text{ and } z\in [-2,2]\rbrace.
\end{equation*}
We equip $\mathcal{M}$ with its canonical metric to turn it into a compact Riemannian manifold of dimension~$2$ and consider once again the sampling of Whittle--Mat\'ern fields using the Galerkin--Chebyshev approach. In particular,  we take again the discretization space $V_n$ to be the finite element space of piecewise linear functions defined on a polyhedral approximation of the surface with triangular faces. 

As in \Cref{sec:cov_num}, we consider the covariance error between the random field and its approximation. More specifically, we evaluate the covariance of the field along the curve $\mathcal{C}= \lbrace (x,y,z)\in\mathcal{M} : y=0 \text{ and } x>0\rbrace$. To do so, we generate samples of the field using the Galerkin--Chebyshev approach and compute the covariance between the point $(1,0,0)\in\mathcal{C}$ and the points $\mathcal{P}=\lbrace (\sqrt{1+z^2}, 0 , z) : z=-2+0.04i, \quad i\in\bi 0,100\ei\rbrace\subset\mathcal{C}$. We generate $2.5\times 10^6$ samples on these points and use a Monte Carlo estimator to estimate the covariances. Note that for each sample the order $K$ of the polynomial approximation is set in the same way as in \Cref{sec:cov_num} and the mass lumping approximation is applied.
We repeat the experience with finite element spaces defined on gradually refined polyhedral approximations of the surface. An example of a sample of the Whittle--Matérn field on~$\mathcal{M}$ along with the sampled points $\mathcal{P}$ is presented in \Cref{fig:ex_cov_hyper}.

\begin{figure}
	\centering
	\includegraphics[height=0.25\textheight]{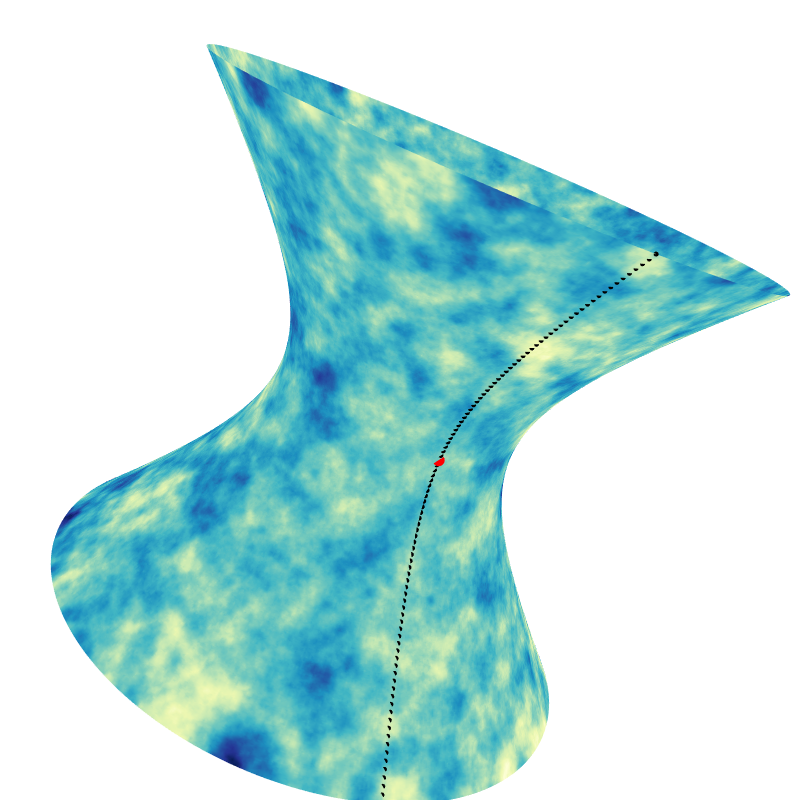}
	\caption{Whittle--Matérn field on the hyperboloid $\mathcal{M}$ along with the sampled points $\mathcal{P}$ used to compute the covariances (in black). The point $(1,0,0)$ is colored in red, and the colors on the surface stand for the value of taken by the field.\label{fig:ex_cov_hyper}}
\end{figure}

 Finally, we compute the covariances with this same approach on a very fine polyhedral approximation of $\mathcal{M}$ (with $540900$ nodes) and use these values as the reference solution. We then compute, for each level of discretization of $\mathcal{M}$, the maximal absolute error between the covariance values and the ground truth.
The result of the numerical experiment is presented in \Cref{fig:error_cov_hyper}. The parameters defining the  \psd $\gamma$ are $\nu=1$ and $a=0.5$ (defined as in \Cref{sec:num}) meaning that we expect convergence of rate $\nu=1$. As can be observed, we retrieve that the maximal absolute error in the covariance decreases as $n ^{-1}$.

\begin{figure}
	\centering
	\includegraphics[height=0.25\textheight]{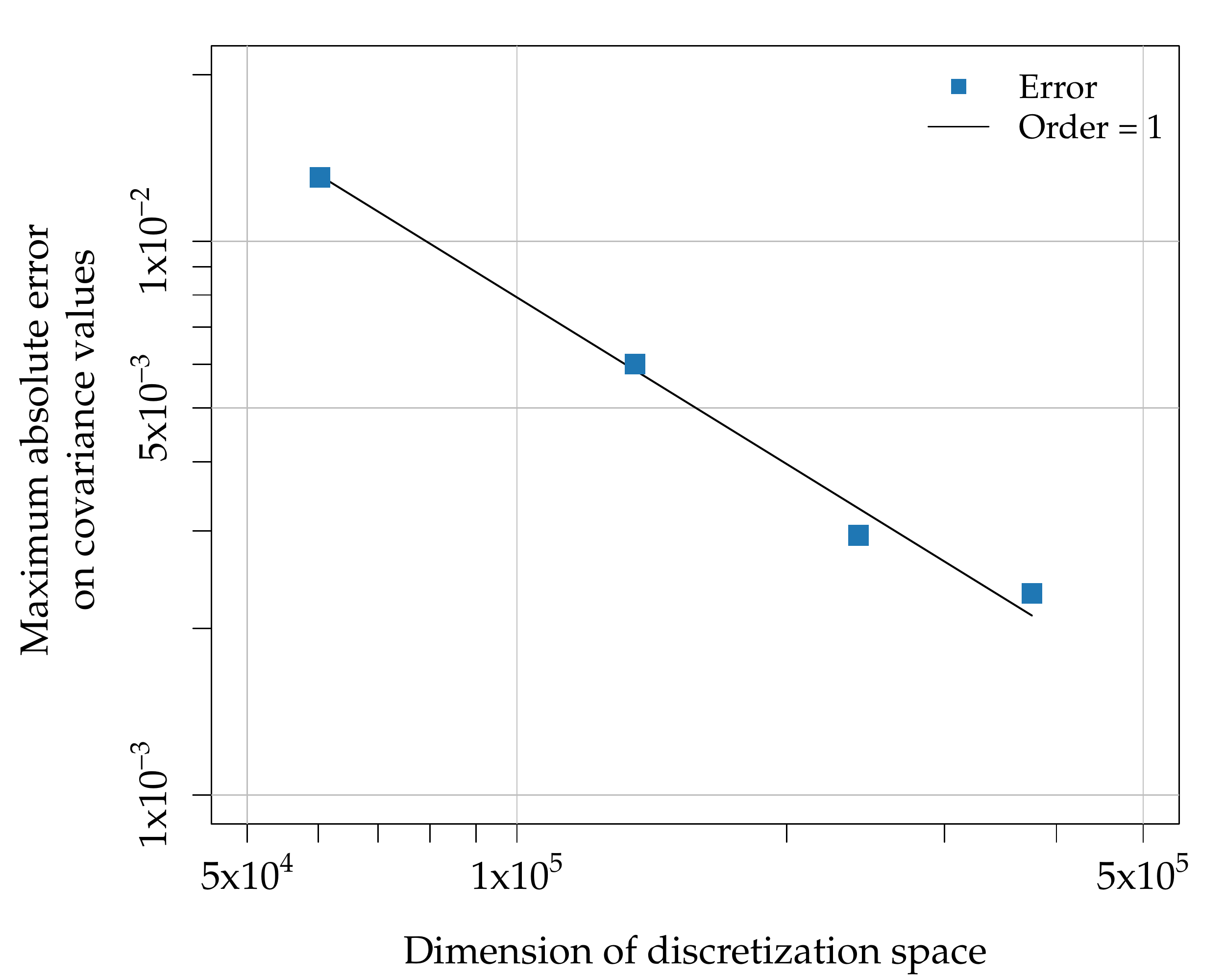}
	\caption{Maximum absolute error in covariance between the Whittle--Matérn field on the hyperboloid and its Galerkin--Chebyshev approximation.\label{fig:error_cov_hyper}}
\end{figure}


\bibliographystyle{plainnat}
\bibliography{allbiblio}

\begin{thebibliography}{52}
\providecommand{\natexlab}[1]{#1}
\providecommand{\url}[1]{\texttt{#1}}
\expandafter\ifx\csname urlstyle\endcsname\relax
  \providecommand{\doi}[1]{doi: #1}\else
  \providecommand{\doi}{doi: \begingroup \urlstyle{rm}\Url}\fi

\bibitem[Adler and Taylor(2009)]{adler2009random}
Robert~J Adler and Jonathan~E Taylor.
\newblock \emph{Random Fields and Geometry}.
\newblock Springer, 2009.

\bibitem[Axelsson and Barker(2001)]{axelsson2001finite}
Owe Axelsson and Vincent~Allan Barker.
\newblock \emph{Finite Element Solution of Boundary Value Problems: Theory and
  Computation}.
\newblock SIAM, 2001.

\bibitem[Bachmayr and Djurdjevac(2020)]{bachmayr2020multilevel}
Markus Bachmayr and Ana Djurdjevac.
\newblock Multilevel representations of isotropic {G}aussian random fields on
  the sphere.
\newblock arXiv:2011.06987, 2020.

\bibitem[B{\'e}rard(2006)]{berard2006spectral}
Pierre~H B{\'e}rard.
\newblock \emph{Spectral Geometry: Direct and Inverse Problems}.
\newblock Springer, 2006.

\bibitem[Bolin and Kirchner(2020)]{bolin2020rational}
David Bolin and Kristin Kirchner.
\newblock The rational {SPDE} approach for {G}aussian random fields with
  general smoothness.
\newblock \emph{Journal of Computational and Graphical Statistics}, 29\penalty0
  (2):\penalty0 274--285, 2020.

\bibitem[Bolin et~al.(2018)Bolin, Kirchner, and Kov{\'a}cs]{bolin2018weak}
David Bolin, Kristin Kirchner, and Mih{\'a}ly Kov{\'a}cs.
\newblock Weak convergence of {G}alerkin approximations for fractional elliptic
  stochastic {PDE}s with spatial white noise.
\newblock \emph{BIT Numerical Mathematics}, 58\penalty0 (4):\penalty0 881--906,
  2018.

\bibitem[Bolin et~al.(2020)Bolin, Kirchner, and Kov\'acs]{BKK20}
David Bolin, Kristin Kirchner, and Mih\'aly Kov\'acs.
\newblock Numerical solution of fractional elliptic stochastic {PDE}s with
  spatial white noise.
\newblock \emph{IMA Journal of Numerical Analysis}, 40\penalty0 (2):\penalty0
  1051--1073, April 2020.

\bibitem[Bonito et~al.(2018)Bonito, Demlow, and Owen]{bonito2018priori}
Andrea Bonito, Alan Demlow, and Justin Owen.
\newblock A priori error estimates for finite element approximations to
  eigenvalues and eigenfunctions of the {L}aplace--{B}eltrami operator.
\newblock \emph{SIAM Journal on Numerical Analysis}, 56\penalty0 (5):\penalty0
  2963--2988, 2018.

\bibitem[Bonito et~al.(2022)Bonito, Guignard, and Lei]{bonito2022numerical}
Andrea Bonito, Diane Guignard, and Wenyu Lei.
\newblock Numerical approximation of {G}aussian random fields on closed
  surfaces.
\newblock \emph{arXiv preprint arXiv:2211.13739}, 2022.

\bibitem[Borovitskiy et~al.(2020)Borovitskiy, Terenin, Mostowsky, and
  Deisenroth]{borovitskiy2020mat}
Viacheslav Borovitskiy, Alexander Terenin, Peter Mostowsky, and Marc~Peter
  Deisenroth.
\newblock Matérn {G}aussian processes on {R}iemannian manifolds.
\newblock arXiv:2006.10160, 2020.

\bibitem[Borovitskiy et~al.(2021)Borovitskiy, Azangulov, Terenin, Mostowsky,
  Deisenroth, and Durrande]{borovitskiy21a}
Viacheslav Borovitskiy, Iskander Azangulov, Alexander Terenin, Peter Mostowsky,
  Marc Deisenroth, and Nicolas Durrande.
\newblock Mat{é}rn {G}aussian processes on graphs.
\newblock In Arindam Banerjee and Kenji Fukumizu, editors, \emph{Proceedings of
  The 24th International Conference on Artificial Intelligence and Statistics},
  volume 130 of \emph{Proceedings of Machine Learning Research}, pages
  2593--2601. PMLR, 13--15 Apr 2021.

\bibitem[Bouclet(2012)]{bouclet2012introduction}
Jean-Marc Bouclet.
\newblock An {I}ntroduction to {P}seudo-{D}ifferential {O}perators.
\newblock {L}ecture Notes, 2012.
\newblock URL \url{http://www.math.univ-toulouse.fr/~bouclet}.

\bibitem[Chen and Thom{\'e}e(1985)]{chen1985lumped}
Chuan~Miao Chen and Vidar Thom{\'e}e.
\newblock The lumped mass finite element method for a parabolic problem.
\newblock \emph{The ANZIAM Journal}, 26\penalty0 (3):\penalty0 329--354, 1985.

\bibitem[Chil{\`e}s and Delfiner(2012)]{chiles1999geost}
Jean-Paul Chil{\`e}s and Pierre Delfiner.
\newblock \emph{Geostatistics : Modeling Spatial Uncertainty. 2nd Edition}.
\newblock Wiley Series In Probability and Statistics, 2012.

\bibitem[Cleanthous et~al.(2020)Cleanthous, Georgiadis, Lang, and
  Porcu]{CGLP20}
Galatia Cleanthous, Athanasios Georgiadis, Annika Lang, and Emilio Porcu.
\newblock Regularity, continuity and approximation of isotropic {G}aussian
  random fields on compact two-point homogeneous spaces.
\newblock \emph{Stochastic Processes and their Applications}, 130\penalty0
  (8):\penalty0 4873--4891, August 2020.

\bibitem[Cooley and Tukey(1965)]{cooley1965algorithm}
James~W Cooley and John~W Tukey.
\newblock An algorithm for the machine calculation of complex {F}ourier series.
\newblock \emph{Mathematics of Computation}, 19\penalty0 (90):\penalty0
  297--301, 1965.

\bibitem[Creasey and Lang(2018)]{creasey2018fast}
Peter~E. Creasey and Annika Lang.
\newblock Fast generation of isotropic {G}aussian random fields on the sphere.
\newblock \emph{Monte Carlo Methods and Applications}, 24\penalty0
  (1):\penalty0 1--11, 2018.

\bibitem[Demlow(2009)]{demlow2009higher}
Alan Demlow.
\newblock Higher-order finite element methods and pointwise error estimates for
  elliptic problems on surfaces.
\newblock \emph{SIAM Journal on Numerical Analysis}, 47\penalty0 (2):\penalty0
  805--827, 2009.

\bibitem[Dziuk(1988)]{dziuk1988finite}
Gerhard Dziuk.
\newblock Finite elements for the {B}eltrami operator on arbitrary surfaces.
\newblock In \emph{Partial Differential Equations and Calculus of Variations},
  pages 142--155. Springer, 1988.

\bibitem[Dziuk and Elliott(2013)]{dziuk2013finite}
Gerhard Dziuk and Charles~M Elliott.
\newblock Finite element methods for surface {PDE}s.
\newblock \emph{Acta Numerica}, 22:\penalty0 289--396, 2013.

\bibitem[Emery and Porcu(2019)]{Emery2019}
Xavier Emery and Emilio Porcu.
\newblock Simulating isotropic vector-valued {G}aussian random fields on the
  sphere through finite harmonics approximations.
\newblock \emph{Stochastic Environmental Research and Risk Assessment},
  33\penalty0 (8):\penalty0 1659--1667, 2019.

\bibitem[Estrade et~al.(2019)Estrade, Fari{\~n}as, and
  Porcu]{estrade2019covariance}
Anne Estrade, Alessandra Fari{\~n}as, and Emilio Porcu.
\newblock Covariance functions on spheres cross time: Beyond spatial isotropy
  and temporal stationarity.
\newblock \emph{Statistics \& Probability Letters}, 151:\penalty0 1--7, 2019.

\bibitem[Friedberg et~al.(2003)Friedberg, Insel, and
  Spence]{friedberg2003linear}
Steven~H. Friedberg, Arnold~J. Insel, and Lawrence~E. Spence.
\newblock \emph{Linear Algebra}.
\newblock Featured Titles for Linear Algebra (Advanced) Series. Pearson
  Education, 2003.
\newblock ISBN 9780130084514.

\bibitem[Gerschgorin(1931)]{gerschgorin1931uber}
S~Gerschgorin.
\newblock \"{U}ber die {A}bgrenzung der {E}igenwerte einer {M}atrix.
\newblock \emph{Bulletin de l'Acad\'emie des Sciences de l'URSS. Classe des
  sciences math\'ematiques et na}, 6:\penalty0 749--754, 1931.

\bibitem[Gneiting(2013)]{gneiting2013strictly}
Tilmann Gneiting.
\newblock Strictly and non-strictly positive definite functions on spheres.
\newblock \emph{Bernoulli}, 19\penalty0 (4):\penalty0 1327--1349, 2013.

\bibitem[Harbrecht et~al.(2021)Harbrecht, Herrmann, Kirchner, and
  Schwab]{harbrecht2021multilevel}
Helmut Harbrecht, Lukas Herrmann, Kristin Kirchner, and Christoph Schwab.
\newblock Multilevel approximation of {G}aussian random fields: Covariance
  compression, estimation and spatial prediction.
\newblock arXiv:2103.04424, 2021.

\bibitem[Herrmann et~al.(2018)Herrmann, Lang, and
  Schwab]{herrmann2018numerical}
Lukas Herrmann, Annika Lang, and Christoph Schwab.
\newblock Numerical analysis of lognormal diffusions on the sphere.
\newblock \emph{Stochastics and Partial Differential Equations: Analysis and
  Computations}, 6\penalty0 (1):\penalty0 1--44, 2018.

\bibitem[Herrmann et~al.(2020)Herrmann, Kirchner, and
  Schwab]{herrmann2020multilevel}
Lukas Herrmann, Kristin Kirchner, and Christoph Schwab.
\newblock Multilevel approximation of {G}aussian random fields: fast
  simulation.
\newblock \emph{Mathematical Models and Methods in Applied Sciences},
  30\penalty0 (01):\penalty0 181--223, 2020.

\bibitem[Huang et~al.(2011)Huang, Zhang, and Robeson]{Huang2011}
Chunfeng Huang, Haimeng Zhang, and Scott~M. Robeson.
\newblock On the validity of commonly used covariance and variogram functions
  on the sphere.
\newblock \emph{Mathematical Geosciences}, 43\penalty0 (6):\penalty0 721--733,
  Aug 2011.

\bibitem[Jansson et~al.(2022)Jansson, Kov{\'a}cs, and Lang]{jansson2021surface}
Erik Jansson, Mih{\'a}ly Kov{\'a}cs, and Annika Lang.
\newblock Surface finite element approximation of spherical
  {W}hittle--{M}at\'ern {G}aussian random fields.
\newblock \emph{SIAM Journal of Scientific Computing}, 2022.

\bibitem[Jones(1963)]{jones1963stochastic}
Richard~H. Jones.
\newblock Stochastic processes on a sphere.
\newblock \emph{The Annals of Mathematical Statistics}, 34\penalty0
  (1):\penalty0 213--218, 1963.

\bibitem[Jost(2008)]{jost2008riemannian}
J{\"u}rgen Jost.
\newblock \emph{Riemannian Geometry and Geometric Analysis}.
\newblock Springer, 2008.

\bibitem[Knyazev and Osborn(2006)]{knyazev2006new}
Andrew~V Knyazev and John~E Osborn.
\newblock New a priori {FEM} error estimates for eigenvalues.
\newblock \emph{SIAM Journal on Numerical Analysis}, 43\penalty0 (6):\penalty0
  2647--2667, 2006.

\bibitem[Labl{\'e}e(2015)]{lablee2015spectral}
Olivier Labl{\'e}e.
\newblock \emph{Spectral Theory in Riemannian Geometry}.
\newblock EMS textbooks in Mathematics. European Mathematical Society, 2015.

\bibitem[Lang and Schwab(2015)]{lang2015isotropic}
Annika Lang and Christoph Schwab.
\newblock Isotropic {G}aussian random fields on the sphere: regularity, fast
  simulation and stochastic partial differential equations.
\newblock \emph{The Annals of Applied Probability}, 25\penalty0 (6):\penalty0
  3047--3094, 2015.

\bibitem[Lantu{\'e}joul et~al.(2019)Lantu{\'e}joul, Freulon, and
  Renard]{Lantuejoul2019}
Christian Lantu{\'e}joul, Xavier Freulon, and Didier Renard.
\newblock Spectral simulation of isotropic {G}aussian random fields on a
  sphere.
\newblock \emph{Mathematical Geosciences}, 51\penalty0 (8):\penalty0 999--1020,
  2019.

\bibitem[Lee(2013)]{lee2013smooth}
John~M Lee.
\newblock Smooth manifolds.
\newblock In \emph{Introduction to Smooth Manifolds}, pages 1--31. Springer,
  2013.

\bibitem[Lindgren et~al.(2011)Lindgren, Rue, and
  Lindstr{\"o}m]{lindgren2011explicit}
Finn Lindgren, H{\aa}vard Rue, and Johan Lindstr{\"o}m.
\newblock An explicit link between {G}aussian fields and {G}aussian {M}arkov
  random fields: the stochastic partial differential equation approach.
\newblock \emph{Journal of the Royal Statistical Society: Series B},
  73\penalty0 (4):\penalty0 423--498, 2011.

\bibitem[Lototsky et~al.(2017)Lototsky, Rozovsky,
  et~al.]{lototsky2017stochastic}
Sergey~V Lototsky, Boris~L Rozovsky, et~al.
\newblock \emph{Stochastic Partial Differential Equations}.
\newblock Springer, 2017.

\bibitem[Marinucci and Peccati(2011)]{marinucci2011random}
Domenico Marinucci and Giovanni Peccati.
\newblock \emph{Random Fields on the Sphere: Representation, Limit Theorems and
  Cosmological Applications}.
\newblock Cambridge University Press, 2011.

\bibitem[Mason and Handscomb(2002)]{mason2002chebyshev}
John~C. Mason and David~C. Handscomb.
\newblock \emph{Chebyshev Polynomials}.
\newblock CRC Press, 2002.

\bibitem[Parlett(1998)]{parlett1998symmetric}
Beresford~N. Parlett.
\newblock \emph{The Symmetric Eigenvalue Problem}.
\newblock SIAM, 1998.

\bibitem[Pereira(2019)]{pereira2019generalized}
Mike Pereira.
\newblock \emph{Generalized Random Fields on {R}iemannian Manifolds: Theory and
  Practice}.
\newblock PhD thesis, Universit{\'e} Paris Sciences et Lettres, 2019.

\bibitem[Pereira and Desassis(2019)]{pereira2019efficient}
Mike Pereira and Nicolas Desassis.
\newblock Efficient simulation of {G}aussian {M}arkov random fields by
  {C}hebyshev polynomial approximation.
\newblock \emph{Spatial Statistics}, 31:\penalty0 100359, 2019.

\bibitem[Porcu et~al.(2016)Porcu, Bevilacqua, and Genton]{porcu2016spatio}
Emilio Porcu, Moreno Bevilacqua, and Marc~G. Genton.
\newblock Spatio-temporal covariance and cross-covariance functions of the
  great circle distance on a sphere.
\newblock \emph{Journal of the American Statistical Association}, 111\penalty0
  (514):\penalty0 888--898, 2016.

\bibitem[Press et~al.(2007)Press, Teukolsky, Vetterling, and
  Flannery]{press2007numerical}
William~H Press, Saul~A Teukolsky, William~T Vetterling, and Brian~P Flannery.
\newblock \emph{Numerical Recipes: The Art of Scientific Computing}.
\newblock Cambridge university press, 3rd edition, 2007.

\bibitem[Romary(2008)]{romary2008inversion}
Thomas Romary.
\newblock \emph{Inversion des Mod{\`e}les Stochastiques de Milieux
  H{\'e}t{\'e}rog{\`e}nes}.
\newblock PhD thesis, Universit{\'e} Pierre et Marie Curie-Paris VI, 2008.

\bibitem[Sauter and Schwab(2011)]{SSch11}
Stefan~A. Sauter and Christoph Schwab.
\newblock \emph{Boundary Element Methods}.
\newblock Springer Series in Computational Mathematics. Springer, 2011.

\bibitem[Strang and Fix(1973)]{strang1973analysis}
Gilbert Strang and George~J. Fix.
\newblock \emph{An Analysis of the Finite Element Method}.
\newblock Prentice-Hall, 1973.

\bibitem[Strichartz(1983)]{strichartz1983analysis}
Robert~S. Strichartz.
\newblock Analysis of the {L}aplacian on the complete {R}iemannian manifold.
\newblock \emph{Journal of Functional Analysis}, 52\penalty0 (1):\penalty0
  48--79, 1983.

\bibitem[Taylor(1996)]{taylor1996partial}
Michael~E. Taylor.
\newblock \emph{Partial Differential Equations I: Basic Theory}.
\newblock Springer, 1996.

\bibitem[Trefethen(2019)]{trefethen2019approximation}
Lloyd~Nicholas Trefethen.
\newblock \emph{Approximation Theory and Approximation Practice. Extended
  Edition}.
\newblock SIAM, 2019.

\end{thebibliography}



\appendix


%

	\section{Uniform convergence of Chebyshev series}\label{sec:cheb_pres}
	
	The next theorem is proven in \cite[Theorems 7.1, 7.2, 8.1, 8.2]{trefethen2019approximation} and gives conditions for the uniform convergence of Chebyshev series. 
	\begin{theorem}
		Let $\nu \in\N$. If $f:[-1, 1] \rightarrow\R$ is such that its derivatives $f, f', \dots, f^{(\nu-1)}$ are continuous and that $f^{(\nu)}$ is of bounded variation, then the coefficients of the Chebyshev series of $f$ satisfy for any $k>\nu$,
		\begin{equation*}
			\vert c_k\vert \le \frac{2}{\pi(k-\nu)^{\nu+1}} \text{TV}(f^{(\nu)}) \veq
		\end{equation*}
		and for any $K>\nu$, the error of the Chebyshev approximation is bounded by
		\begin{equation*}
			\quad \Vert f - \mathcal{S}_K[f]\Vert_\infty \le \frac{2}{\pi\nu(K-\nu)^\nu} \text{TV}(f^{(\nu)}) \veq
		\end{equation*}
		where $\text{TV}(f^{(\nu)})$ denotes the total variation of $f^{(\nu)}$ over $[-1, 1]$ and $\Vert\cdot\Vert_{\infty}$ denotes the $L^{\infty}$-norm on the segment $[-1, 1]$.
		
		Besides if there exists $\rho >1$ such that the complex function $z\in\C \mapsto f(z)$ is holomorphic inside the ellipse $E_\rho$ centered at $0$, with foci $z=\pm 1$, and semi-major (resp.\ semi-minor) axis of length $(\rho+\rho^{-1})/2$  (resp.\ $(\rho-\rho^{-1})/2$), then, for any $K\ge 0$,
		\begin{equation*}
			\vert c_K\vert \le 
			\frac{2}{\rho^K} \sup_{z\in E_\rho}\vert f(z)\vert
		\end{equation*}
		and
		\begin{equation*}
			\Vert f - \mathcal{S}_K[f]\Vert_\infty \le 
			\frac{2}{\rho^K(\rho-1)} \sup_{z\in E_\rho}\vert f(z)\vert\peq
		\end{equation*}
		\label{thm:cheb}
	\end{theorem}

	\section{Proof of \Cref{prop:diag_lapl_discrg}}\label{appen:diag_lapl_discrg}
	
	\begin{proof}
		
		Take an eigenvalue $\lambda$ of the GEP defined by the matrix pencil $(\bm R, \bm C)$, and denote by $\bm w \neq \bm 0$ an associated eigenvector. Using~\eqref{eq:def_CG}, we have for any $k \in[\![1, n]\!]$,
		$$\sum_{l=1}^n( \nabla_{\mathcal{M}}\psi_k, \nabla_{\mathcal{M}}\psi_l)_0w_l=\lambda\sum_{l=1}^n \left(\psi_k, \psi_l\right)_0w_l \veq$$
		which, by definition of $E_0$, gives for any $k \in[\![1, n]\!]$,
		\begin{equation}
			(\nabla_{\mathcal{M}}\psi_k, \nabla_{\mathcal{M}} E_0(\bm w))_0= \lambda(\psi_k, E_0(\bm w))_0 \peq
			\label{eq:eq_psi}
		\end{equation}
		Note that $\lbrace \psi_k\rbrace_{1\le k\le n}$ is also a basis of $V_n$ as it is a family of linearly independent functions spanning $V_n$. Denote by $\bm A\in\R^{n\times n}$ the invertible change-of-basis matrix between $\lbrace \psi_k\rbrace_{1\le k\le n}$ and the orthonormal basis $\lbrace f_k \rbrace_{1\le k\le n}$ of $V_n$ in~\eqref{eq:proj_lap}. In particular, $\bm A$ satisfies, for any $k \in[\![1, n]\!]$,
		\begin{equation*}
			\psi_k = \sum_{l=1}^n A_{kl} f_l \peq
		\end{equation*}
		Injecting this last equality in~\eqref{eq:eq_psi} gives
		$$\bm A 
		\left[
		( \nabla_{\mathcal{M}}f_k, \nabla_{\mathcal{M}}E_0(\bm w))_0 
		\right]_{1\le k\le n}
		=\lambda\bm A 
		\left[
		(f_k, E_0(\bm w))_0 
		\right]_{1\le k\le n}.$$
		Multiplying both members of this equality by $\bm A^{-1}$ yields that for any $k\in [\![1, n]\!]$,
		$$( \nabla_{\mathcal{M}}f_k, \nabla_{\mathcal{M}}E_0(\bm w))_0=\lambda( f_k, E_0(\bm w))_0 \peq$$
		And so, given that $E_0(\bm w)\in V_n$, 
		\begin{align*}
			-\Delta_n E_0(\bm w) = \sum\limits_{k=1}^n ( \nabla_{\mathcal{M}}f_k,\nabla_{\mathcal{M}}E_0(\bm w))_0 f_k
			=\lambda\sum\limits_{k=1}^n ( f_k, E_0(\bm w))_0 f_k
			=\lambda E_0(\bm w) \peq
		\end{align*}
		Therefore $\lambda$ is an eigenvalue of $-\Delta_n$ and $E_0$ maps the eigenvectors of $(\bm R, \bm C)$  to the eigenfunctions of $-\Delta_n$.\\
		Observe that for any $\bm x\in\mathbb{R}^n$, 
		$$\Vert E_0(\bm x)\Vert_0^2
		=\sum_{k=1}^n\sum_{l=1}^n x_k( \psi_k, \psi_l)_0x_l
		=\bm x^T \bm C \bm x=\Vert\bm x\Vert_{\bm C}^2 \speq$$
		Hence, given that it is also linear, $E_0$ is an isometry between $(\R^n, \Vert\cdot\Vert_{\bm C})$ and $(V_n, \Vert\cdot\Vert_0)$.
		Consequently, $E_0$ is injective: for any $\bm x\in\R^n$, $E_0(\bm x)=0$ implies that $\Vert\bm x\Vert_{\bm C}^2=\Vert E_0(\bm x)\Vert_0^2=0$ and so that $\bm x=0$. Finally, using the rank–nullity theorem \cite{friedberg2003linear}, $E_0$ is bijective (as an injective linear mapping between two vector spaces with the same dimension).\qed
	\end{proof}

	\section{Proof of the eigenvalue estimates for SFEM}\label{appen:sfem}

		Assume that $\mathcal{M}$ is a smooth compact $2$-dimensional surface without boundary  equipped with the metric $g$ induced by the Euclidean metric on $\R^3$. Following the SFEM approach, we consider a polyhedral approximation $\mathcal{M}_h$ of $\mathcal{M}$ with mesh size $h$ such that the vertices of $\mathcal{M}_h$ lie on $\mathcal{M}$. 
		Let $V_h$ be the finite-dimensional space of functions obtained by \q{lifting} on $\mathcal{M}$ the linear finite element space defined on the polyhedral mesh $\mathcal{M}_h$. Note in particular that $V_h$ is geometrically consistent in the sense that $V_h \subset \dot{H}^1$. Denote then by $(\lambda_{k}, e_k)_{k\in\N}$ the eigenpairs of the Laplace--Beltrami operator $-\Delta_{\mathcal{M}}$ and by $(\lambda_{k}^{(n)}, e_k^{n})_{1\le k\le n}$ the eigenpairs of the Galerkin approximation of $-\Delta_{\mathcal{M}}$ on $V_h$, as defined in \Cref{sec:rg_discr} (where $n=\dim V_h$).

		Following \cite[Theorem 3.1]{knyazev2006new} and the smoothness of the eigenfunctions of $-\Delta_{\mathcal{M}}$, there exists $C$ (independent of $h$) such that for any $k\in\lbrace 1, \dots, n\rbrace$, 
		\begin{equation*}
			0 \le \lambda_{k}^{(n)} - \lambda_k \le C \lambda_{k}^{(n)}\lambda_{k}h^2,
		\end{equation*}
		(see \cite[Lemma 4.1]{bonito2022numerical} for a complete proof).
		Reinserting this bound and using the growth of the eigenvalues yield
		\begin{equation*}
			0 \le \lambda_{k}^{(n)} - \lambda_k \le C \lambda_{k}^2h^2 + C \lambda_{k}h^2(\lambda_{k}^{(n)} - \lambda_{k}) \le C \lambda_{k}^2h^2(1+C\lambda_{k}^{(n)}h^2) \le C \lambda_{k}^2h^2(1+C\lambda_{n}^{(n)}h^2).
		\end{equation*}
		Note then that by an inverse inequality \cite[Proposition 2.7]{demlow2009higher}, there exists $C_{\text{INV}}>0$ independent of $h$  such that,  for $h$ small enough,
		\begin{equation*}
			\Vert \nabla_\mathcal{M} e_n^{(n)}\Vert_0 \le C_{\text{INV}} h^{-1} \Vert e_n^{(n)} \Vert_0= C_{\text{INV}} h^{-1}.
		\end{equation*}
		Since $\Vert \nabla_\mathcal{M} e_n^{(n)}\Vert_0^2=(\nabla_\mathcal{M} e_n^{(n)}, \nabla_\mathcal{M} e_n^{(n)})_0=\lambda_{n}^{(n)}(e_n^{(n)}, e_k^{(n)})_0=\lambda_{n}^{(n)}$, we get $\lambda_{n}^{(n)} \le (C_{\text{INV}})^2 h^{-2}$. Hence we can conclude that
		\begin{equation*}
			0 \le \lambda_{k}^{(n)} - \lambda_k 
			\le C' \lambda_{k}^2h^2,
		\end{equation*}
		where $C'=C(1+C(C_{\text{INV}})^2)$ is a constant independent of $k$ and $h$.	
		Finally, assuming that the polyhedral approximations $\mathcal{M}_h$ for different values of $h$ are built from uniform refinements of an initial polyhedral surface, the size $n$ of $V_h$ can be linked to the mesh size $h$ by $nh^2 \lesssim 1$, which in turn gives
		\begin{equation*}
			0 \le \lambda_{k}^{(n)} - \lambda_k 
			\le C' \lambda_{k}^2n^{-1}.
		\end{equation*}


\pagebreak

\part*{}
\addcontentsline{toc}{part}{\sm}
\begin{center}
	\begin{Large}
		{\scshape \textbf{\sm}}
	\end{Large}
\end{center}
\medskip

\setcounter{section}{0}
\renewcommand{\sectionname}{}
\renewcommand{\thesection}{SM\arabic{section}} 


\section{Series bounds}
\label{sec::proof_fem_gerf}

\begin{lemma}
	Let $m\in\R$, $m\neq -1$ and $n\in\N$. Then,
	\begin{equation*}
	\frac{1}{m+1}\left(1-\frac{1}{n^{m+1}}\right) +\frac{1}{n^{\max\lbrace 1, m+1\rbrace}} 
	\le \frac{1}{n}\sum_{k=1}^n  \left(\frac{k}{n}\right)^m 
	\le \frac{1}{m+1}\left(1-\frac{1}{n^{m+1}}\right)+\frac{1}{n^{\min\lbrace 1, m+1\rbrace}}
	\end{equation*}
	\label{lem:ineq_sum_int}
\end{lemma}

\begin{proof}
	Let $n\ge 1$ and let $S_n$ denote the sum $S_n=\sum_{k=1}^n  \left(\frac{k}{n}\right)^m$. \\
	First, assume that $m\le 0$. Then, for any $k\in \bi 1, n-1\ei$ and any $t\in [k/n, \, (k+1)/n]$,
	\begin{equation*}
	\left(\frac{k+1}{n}\right)^m \le t^m \le \left(\frac{k}{n}\right)^m.
	\end{equation*}
	Integrating both inequalities over $[k/n, \, (k+1)/n]$ and summing them for $k\in \bi 1, n-1\ei$ gives:
	\begin{equation*}
	\frac{1}{n} \left( S_n - \frac{1}{n^m}\right) \le I_n \le \frac{1}{n} (S_{n}-1),
	\end{equation*}
	where $I_n=\int_{1/n}^1 t^m \dd t=\left(1-n^{-(m+1)}\right)/(m+1)$.
	Hence, we have
	\begin{equation*}
	I_n +\frac{1}{n} \le \frac{1}{n}S_n \le I_n+\frac{1}{n^{m+1}}.
	\end{equation*}
	Similarly, if $m\ge 0$ we get
	\begin{equation*}
	I_n +\frac{1}{n^{m+1}} \le \frac{1}{n}S_n \le I_n+\frac{1}{n}.
	\end{equation*}
	So, for any $ m \neq -1$, we have
	\begin{equation*}
	I_n +\frac{1}{n^{\max\lbrace 1, m+1\rbrace}} \le \frac{1}{n}S_n \le I_n+\frac{1}{n^{\min\lbrace 1, m+1\rbrace}}. \qedhere
	\end{equation*}
\end{proof}

\begin{lemma}
	Let $m\in\R, m>1$ and let $J\in\N$, $J\ge 1$. Then,
	\begin{equation*}
	\frac{1}{(m-1)(J+1)^{m-1}}\le \sum_{j= J+1}^\infty j^{-m} \le \frac{1}{(m-1)J^{m-1}}
	\end{equation*}
	\label{eq:encadr_sum}
\end{lemma}

\begin{proof}
	This result is obtained straightforwardly by upper-bounding  and lower-bounding the integrals $\int_{1}^{J} t^{-m}\di t$, $\int_{1}^{J+1} t^{-m}\di t$ and $\int_{J+1}^{+\infty} t^{-m}\di t$.
\end{proof}

\section{Additional properties of the Galerkin discretization}

\begin{lemma}
	Let $\bm C$ and $\bm R$ be the mass and stiffness matrices defined in~\eqref{eq:def_CG}.
	Then,  $\bm C$ is a symmetric positive definite matrix and $\bm R$ is a symmetric positive semi-definite matrix.
	\label{prop:prop_CG}
\end{lemma}

\begin{proof}
	On one hand, note that $\bm C$ is symmetric since the functions $\lbrace \psi_k\rbrace_{1\le k\le n}$ are real-valued. Also, for any $\bm x\in\R^n$, 
	\begin{equation*}
	\bm x^T\bm C\bm x=\sum_{k=1}^n\sum_{l=1}^n x_kx_l (  \psi_k,   \psi_l)_0 =\bigg\Vert\sum_{k=1}^n x_k \psi_k\bigg\Vert_{0}^2\ge 0 \speq
	\end{equation*}
	Given that the functions $\lbrace \psi_k\rbrace_{1\le k\le n}$ are linearly independent, this quantity is zero if and only if $\bm x=\bm 0$. Hence, $\bm C$ is positive definite.\\
	On the other hand, $\bm R$ is by definition symmetric.
	And, for any $\bm x\in\R^n$, 
	\begin{equation*}
	\bm x^T\bm R\bm x=\bigg( \sum_{k=1}^n x_k \nabla_{\mathcal{M}}\psi_k,\, \sum_{l=1}^n x_l\nabla_{\mathcal{M}} \psi_l\bigg)_0=\bigg\Vert\sum_{k=1}^n x_k \nabla_{\mathcal{M}}\psi_k\bigg\Vert_0^2\ge 0.
	\end{equation*}
	Hence $\bm R$ is positive semi-definite.
\end{proof}

\begin{prop}
	The definition of $\gamma(-\Delta_n)$ in \eqref{eq:discr_pseudo_diff} does not depend on the choice of orthonormal basis $\lbrace e_{k}^{(n)}\rbrace_{1\le k\le n}$ of eigenfunctions of $-\Delta_n$ satisfying for any $k\in\bi 1,n\ei$, $-\Delta_n e_{k}^{(n)}=\lambda_{k}^{(n)}e_{k}^{(n)}$.
	\label{prop:onb}
\end{prop}

\begin{proof}
	Let $\lbrace e_{k}^{(n)}\rbrace_{1\le k\le n}$ and $\lbrace \tilde{e}_{k}^{(n)}\rbrace_{1\le k\le n}$ denote two orthonormal bases of $V_n$ such that for any $k\in\bi 1,n\ei$, $-\Delta_n e_{k}^{(n)}=\lambda_{k}^{(n)}e_{k}^{(n)}$ and $-\Delta_n \tilde{e}_{k}^{(n)}=\lambda_{k}^{(n)}\tilde{e}_{k}^{(n)}$. Assume that $\gamma(-\Delta_n)$ is defined by~\eqref{eq:discr_pseudo_diff}.
	
	Let $\bm A$ be the change-of-basis matrix between $\lbrace e_{k}^{(n)}\rbrace_{1\le k\le n}$ and $\lbrace \tilde{e}_{k}^{(n)}\rbrace_{1\le k\le n}$, i.e., for any $k\in\bi 1,n\ei$, $e_{k}^{(n)} = \sum_{l=1}^nA_{kl}\tilde{e}_{l}^{(n)}$.
	
	Note that, since $\lbrace \tilde{e}_{k}^{(n)}\rbrace_{1\le k\le n}$ is orthonormal, we have for any $ k,k' \in\bi 1,n\ei$,
	\begin{align*}
	( e_{k}^{(n)},e_{k'}^{(n)})_0
	&=\sum_{l,l'=1}^nA_{kl}A_{k'l'}(\tilde{e}_{l}^{(n)},\tilde{e}_{l'}^{(n)})_0
	=\sum_{l=1}^nA_{kl}A_{k'l}
	=[\bm A\bm A^T]_{kk'}\peq
	\end{align*}
	Therefore, since $\lbrace e_{k}^{(n)}\rbrace_{1\le k\le n}$ is also orthonormal, we have  $\bm A\bm A^T=\bm I_n=\bm A^T\bm A$. 
	
	Then, recall that $\lbrace e_{k}^{(n)}\rbrace_{1\le k\le n}$ and $\lbrace \tilde{e}_{k}^{(n)}\rbrace_{1\le k\le n}$ are eigenfunctions of $-\Delta_n$. Hence, for any $k\in\bi 1, n\ei$,
	$$-\Delta_n e_{k}^{(n)}=\lambda_{k}^{(n)}e_{k}^{(n)}=\sum_{l=1}^n\lambda_{k}^{(n)}A_{kl}\tilde{e}_{l}^{(n)}\veq$$
	and, by linearity of $-\Delta_n$, 
	\begin{equation*}
	-\Delta_n e_{k}^{(n)}= \sum_{l=1}^nA_{kl}(-\Delta_n \tilde{e}_{l}^{(n)})=\sum_{l=1}^n\lambda_{l}^{(n)}A_{kl} \tilde{e}_{l}^{(n)} \peq
	\end{equation*}
	Consequently, by identification of both formulas, for any $k,l\in\bi 1,n\ei$,
	\begin{equation*}
	\lambda_{k}^{(n)}A_{kl}=\lambda_{l}^{(n)}A_{kl} \peq
	\end{equation*}
	A proof by contradiction then gives that for any $k,l\in\bi 1,n\ei$, the following also holds:
	$$\gamma(\lambda_{k}^{(n)})A_{kl}=\gamma(\lambda_{l}^{(n)})A_{kl},$$
	and therefore,
	\begin{equation*}
	\gamma(\bm\Lambda)\bm A=\bm A\gamma(\bm\Lambda), \quad \text{where}\quad
	\gamma(\bm\Lambda):=
	\Diag\left(\gamma\big(\lambda_{1}^{(n)}\big) , \dots, \gamma\big(\lambda_{n}^{(n)}\big)\right)
	\peq
	\end{equation*}
	
	Finally, note that, by definition of $\gamma(-\Delta_n)$, we have for every $\varphi\in V_n$ , 
	\begin{equation*}
	\begin{aligned}
	\gamma(-\Delta_n)\varphi
	&= \sum\limits_{k=1}^n\gamma(\lambda_{k}^{(n)})\big( \varphi,e_{k}^{(n)}\big)_0\; e_{k}^{(n)}
	=\sum\limits_{k,l,l'=1}^n\gamma(\lambda_{k}^{(n)})A_{kl} A_{kl'}\big( \varphi,\tilde{e}_{l}^{(n)}\big)_0\tilde{e}_{l'}^{(n)}\\
	&=\sum_{l,l'=1}^n [\bm A^T\gamma(\bm\Lambda)\bm A]_{ll'}\big( \varphi,\tilde{e}_{l}^{(n)}\big)_0\tilde{e}_{l'}^{(n)} \veq
	\end{aligned}
	\end{equation*}
	and since we proved that $\gamma(\bm\Lambda)\bm A=\bm A\gamma(\bm\Lambda)$, 
	\begin{equation*}
	\begin{aligned}
	\gamma(-\Delta_n)\varphi
	&=\sum_{l,l'=1}^n [\bm A^T\bm A\gamma(\bm\Lambda)]_{ll'}\big( \varphi,\tilde{e}_{l}^{(n)}\big)_0\tilde{e}_{l'}^{(n)}
	=\sum_{l,l'=1}^n [\bm I_n\gamma(\bm\Lambda)]_{ll'}\big( \varphi,\tilde{e}_{l}^{(n)}\big)_0\tilde{e}_{l'}^{(n)}\\
	&=\sum_{l=1}^n \gamma(\lambda_{l}^{(n)})\big( \varphi,\tilde{e}_{l}^{(n)}\big)_0\tilde{e}_{l}^{(n)} \sveq
	\end{aligned}
	\end{equation*}
	which proves the result.
\end{proof}

\section{Proof of Theorem 5.7 of the main article}\label{proof:err_cov}

We know provide a proof of \Cref{th:err_cov} of the main article, which we first recall.

\begin{thm*}
	Let Assumptions \ref{assum:alphabeta} and \ref{assum:Lh}  be satisfied. Then, there exists some $N_2\in \N$ such that for any $n>N_2$, 
the covariance error between the random field $\mathcal{Z}$ and its discretization $\mathcal{Z}_n$ satisfies, for any $\theta,\varphi \in H$,
\begin{equation*}
	\begin{aligned}
		\big\vert \cov\left((\mathcal{Z}, \theta)_0 , (\mathcal{Z}, \varphi)_0\right) &-\cov\left((\mathcal{Z}_n, \theta)_0 , (\mathcal{Z}_n, \varphi)_0\right)\big\vert \\
		&\lesssim
		\begin{cases}
			n^{-\min\left\lbrace s ;\; r ;\; (2\alpha\beta-1)\right\rbrace}\log n  
			& \text{if } (2\alpha\beta-1) = s,\\
			n^{-\min\left\lbrace s ;\; r ;\; (2\alpha\beta-1)\right\rbrace}  \log n
			& \text{if } (2\alpha\beta-1) = r  \text{ and } q>1,\\
			n^{-\min\left\lbrace s ;\; r;\; (2\alpha\beta-1)\right\rbrace}
			& \text{else}.
		\end{cases}
	\end{aligned}
\end{equation*}
\end{thm*}

\begin{proof}
	
	Let $n>\max\lbrace M_0; N_0\rbrace$,  $\theta,\varphi \in H$ and let $R(\theta,\varphi)$ be defined by
	\begin{align*}
	R(\theta,\varphi)
	& =\left\vert \cov\left((\mathcal{Z}, \theta)_0 , (\mathcal{Z}, \varphi)_0\right)-\cov\left((\mathcal{Z}_n, \theta)_0 , (\mathcal{Z}_n, \varphi)_0\right)\right\vert\\
	& = \vert \e\left[(\mathcal{Z}, \theta)_0 (\mathcal{Z}, \varphi)_0\right]
	- \e\left[(\mathcal{Z}_n, \theta)_0 (\mathcal{Z}_n, \varphi)_0\right] \vert \sveq
	\end{align*}
	where the last equality follows from the fact that $\mathcal{Z}$ and $\mathcal{Z}_n$ are centered.
	
	To prove the error estimate of this theorem, we proceed in the same way as in \Cref{th:err_fem} by splitting
	\begin{equation*}
	\begin{aligned}
	R(\theta,\varphi) & \le 
	\big\vert \e\left[(\mathcal{Z}, \theta)_0 (\mathcal{Z}, \varphi)_0\right] - \e[(\mathcal{Z}^{(n)}, \theta)_0 (\mathcal{Z}^{(n)}, \varphi)_0] \big\vert \\
	& \quad\quad + \big\vert \e[(\mathcal{Z}^{(n)}, \theta)_0 (\mathcal{Z}^{(n)}, \varphi)_0] 
	- \e\left[(\mathcal{Z}_n, \theta)_0 (\mathcal{Z}_n, \varphi)_0\right] \big\vert \\
	& = R_T(\theta,\varphi) + R_D(\theta,\varphi) \sveq
	\end{aligned}
	\end{equation*}
	where $\mathcal{Z}^{(n)}$ denotes the truncation of $\mathcal{Z}$ after $n$ terms.
	
	\vspace{1em}
	
	\underline{Truncation error term} $R_T(\theta,\varphi)$ :
	Note that
	\begin{equation*}
	\e\left[(\mathcal{Z}, \theta)_0 (\mathcal{Z}, \varphi)_0\right]
	=\e\bigg[\sum_{k,l\in\N} \gamma(\lambda_k)\gamma(\lambda_l) W_kW_l (e_k, \theta)_0 (e_l, \varphi)_0\bigg]
	=\sum_{k\in\N} \gamma(\lambda_k)^2  (e_k, \theta)_0 (e_k, \varphi)_0
	\end{equation*}
	which gives
	\begin{equation*}
	R_T(\theta,\varphi)
	\le \sum_{k>n} \gamma(\lambda_k)^2  \vert (e_k, \theta)_0 (e_k, \varphi)_0\vert \peq
	\end{equation*}
	Using the Cauchy--Schwartz inequality on the terms $(e_k, \theta)_0$ and $(e_k, \varphi)_0$, the orthonormality of $\lbrace e_k\rbrace_{k\in\N}$, and \Cref{def:psd}, we obtain
	\begin{equation*}
	R_T(\theta,\varphi)\le  \Vert \theta\Vert_0 \Vert \varphi\Vert_0 \sum_{k>n} \gamma(\lambda_k)^2  \lesssim \Vert \theta\Vert_0 \Vert \varphi\Vert_0 \sum_{k>n} \lambda_k^{-2\beta} \peq
	\end{equation*}
	Finally, \Cref{prop:weyl} yields
	\begin{equation*}
	R_T(\theta,\varphi) \lesssim \Vert \theta\Vert_0 \Vert \varphi\Vert_0 \sum_{k>n} k^{-2\alpha\beta} \lesssim \Vert \theta\Vert_0 \Vert \varphi\Vert_0\;  n^{-(2\alpha\beta-1)}\peq
	\end{equation*}

	\vspace{0.5em}
	
	\underline{Discretization error} $R_D(\theta,\varphi)$ : 
	From the triangle inequality,
	\begin{equation*}
	\begin{aligned}
	R_D(\theta,\varphi) & \le 
	\big\vert \e[(\mathcal{Z}^{(n)}, \theta)_0 (\mathcal{Z}^{(n)}, \varphi)_0] - \e[(\widetilde{\mathcal{Z}}^{(n)}, \theta)_0 (\widetilde{\mathcal{Z}}^{(n)}, \varphi)_0] \big\vert \\
	& \quad\quad + \big\vert \e[(\widetilde{\mathcal{Z}}^{(n)}, \theta)_0 (\widetilde{\mathcal{Z}}^{(n)}, \varphi)_0]
	- \e\left[(\mathcal{Z}_n, \theta)_0 (\mathcal{Z}_n, \varphi)_0\right] \big\vert \\
	& = R_D^{(1)}(\theta,\varphi) + R_D^{(2)}(\theta,\varphi) \sveq
	\end{aligned}
	\end{equation*}
	where $\widetilde{\mathcal{Z}}^{(n)}$ is defined as
	\begin{equation*}
	\widetilde{\mathcal{Z}}^{(n)}=\sum_{k=1}^n\gamma(\lambda_k)W_k e_{k}^{(n)} \peq
	\end{equation*}
	%
	The first term can be bounded by
	\begin{equation*}
	\begin{aligned}
	& R_D^{(1)}(\theta,\varphi) \\
	& \quad \le
	\sum_{k=1}^n  \gamma(\lambda_k)^2 \big\vert  (e_k, \theta)_0(e_k, \varphi)_0 -  (e_k^{(n)}, \theta)_0(e_k^{(n)}, \varphi)_0 \big\vert
	\\
	& \quad =
	\sum_{k=1}^n \gamma(\lambda_k)^2  \big\vert (e_k-e_k^{(n)}, \theta)_0(e_k, \varphi)_0  +  (e_k, \theta)_0(e_k-e_k^{(n)}, \varphi)_0 
	-  (e_k-e_k^{(n)}, \theta)_0(e_k-e_k^{(n)}, \varphi)_0 \big\vert
	\end{aligned}
	\end{equation*}
	and satisfies further by the triangle and the Cauchy--Schwartz inequality:
	\begin{equation*}
	\begin{aligned}
	R_D^{(1)}(\theta,\varphi) \lesssim \Vert\theta\Vert_0 \Vert \varphi\Vert_0 \bigg(
	\sum_{k=1}^n \gamma(\lambda_k)^2 \Vert e_k - e_k^{(n)}\Vert_0 + \sum_{k=1}^n \gamma(\lambda_k)^2 \Vert e_k - e_k^{(n)}\Vert_0^2
	\bigg) \speq
	\end{aligned}
	\end{equation*}
	Using \Cref{prop:weyl}, \Cref{def:psd} \Cref{assum:Lh},  and the fact that $\alpha q \le 2s$, the first sum can be bounded by
	\begin{equation*}
	\begin{aligned}
	\sum_{k=1}^n \gamma(\lambda_k)^2 \Vert e_k - e_k^{(n)}\Vert_0 
	&\lesssim 
	n^{-s}\bigg(
	\sum_{k=1}^n k^{-2\alpha\beta+\alpha q/2} \bigg) 
	\lesssim 
	n^{-s}\bigg(
	\sum_{k=1}^n k^{s-2\alpha\beta}\bigg)\\
	&	
	\lesssim 
	\begin{cases}
	n^{-s}\log n & \text{if } 2\alpha\beta-1=s, \\
	n^{-\min\lbrace s, 2\alpha\beta-1\rbrace} &  \text{else.} 
	\end{cases}
	\end{aligned}
	\end{equation*}
	Similarly, we prove that
	\begin{equation*}
	\begin{aligned}
	\sum_{k=1}^n \gamma(\lambda_k)^2 \Vert e_k - e_k^{(n)}\Vert_0^2
	\lesssim 
	\begin{cases}
	n^{-2s}\log n & \text{if } 2\alpha\beta-1=2s, \\
	n^{-\min\lbrace 2s, 2\alpha\beta-1\rbrace} &  \text{else.} 
	\end{cases}
	\end{aligned}
	\end{equation*}
	We conclude then by considering the term with the slowest convergence that
	\begin{equation*}
	\begin{aligned}
	R_D^{(1)}(\theta,\varphi) 
	\lesssim \Vert\theta\Vert_0 \Vert \varphi\Vert_0
	\begin{cases}
	n^{-s}\log n & \text{if } 2\alpha\beta-1=s, \\
	n^{-\min\lbrace s, 2\alpha\beta-1\rbrace} &  \text{else.} 
	\end{cases}
	\end{aligned}
	\end{equation*}
	
	%
	%
	For $R_D^{(2)}(\theta,\varphi)$ we observe that
	\begin{align*}
	R_D^{(2)}(\theta,\varphi) & \le
	\sum_{k=1}^n \big\vert \gamma(\lambda_{k})^2 - \gamma(\lambda_k^{(n)})^2\big\vert \big\vert (e_k^{(n)}, \theta)_0 (e_k^{(n)}, \varphi)_0 \big\vert \\
	& \le
	\Vert\theta\Vert_0\Vert\varphi\Vert_0\sum_{k=1}^n \big\vert \gamma(\lambda_{k})^2 - \gamma(\lambda_k^{(n)})^2\big\vert
	\sveq
	\end{align*}
	where we used that $\lbrace e_k^{(n)} \rbrace_{1\le k \le  n}$ is orthonormal.
	Applying the mean value theorem we get, for any $k\in\bi 1, n\ei$, 
	\begin{equation*}
	\begin{aligned}
	\vert\gamma(\lambda_k)^2-\gamma(\lambda_{k}^{(n)})^2\vert
	&=\vert\gamma(\lambda_k)+\gamma(\lambda_{k}^{(n)})\vert
	\vert\gamma(\lambda_k)-\gamma(\lambda_{k}^{(n)})\vert\\
	&\le \vert\gamma(\lambda_k)+\gamma(\lambda_{k}^{(n)})\vert
	\vert\lambda_{k}^{(n)}-\lambda_k\vert 
	\sup\limits_{\theta \in (0,1)} \vert \gamma'(\theta\lambda_{k}^{(n)}+ (1-\theta)\lambda_{k})\vert.
	\end{aligned}
	\end{equation*}

	Using the same arguments as the ones used in the proof for \Cref{th:err_fem}, we can find some $N_1\in \N$ such that for any $n>N_1$,  and any $k\in \bi N_1, n\ei$, $\min\big\lbrace\lambda_k ;\; \lambda_{k}^{(n)}\big\rbrace \ge l_\lambda c_\lambda k^\alpha\ge L_\gamma$, where $L_\gamma$ is defined in \Cref{def:psd}. For any such $k$, we then have, still according to \Cref{def:psd}, $\vert\gamma(\lambda_k)+\gamma(\lambda_{k}^{(n)})\vert \lesssim (\min\big\lbrace\lambda_k ;\; \lambda_{k}^{(n)}\big\rbrace)^{-\beta} \lesssim k^{-\alpha\beta}$ and for any $\theta\in (0,1)$, 
	$$\vert \gamma'(\theta\lambda_{k}^{(n)}+ (1-\theta)\lambda_{k})\vert \lesssim \left(\min\big\lbrace\lambda_k ;\; \lambda_{k}^{(n)}\big\rbrace\right)^{-(1+\beta)}
	\lesssim k^{-\alpha(1+\beta)},$$ which in turn gives
	\begin{equation*}
	\vert\gamma(\lambda_k)^2-\gamma(\lambda_{k}^{(n)})^2\vert \lesssim \vert\lambda_{k}^{(n)}-\lambda_k\vert 
	k^{-(1+2\beta)}.
	\end{equation*}
	And if $k\in\bi M_0+1, N_1-1\ei$, we can simply take  $\vert\gamma(\lambda_k)^2-\gamma(\lambda_{k}^{(n)})^2\vert \lesssim \vert\lambda_{k}^{(n)}-\lambda_k\vert$ as the other terms can be bounded by constants independent of $n$. 
	In conclusion, using \Cref{assum:Lh} and \Cref{prop:weyl}, we get
	\begin{align*}
	R_D^{(2)}(\theta,\varphi)
	& \lesssim
	\Vert\theta\Vert_0\Vert\varphi\Vert_0\left(
	\sum_{k=M_0+1}^{N_1-1} \vert\lambda_{k}^{(n)}-\lambda_k\vert +
	\sum_{k=N_1}^n \vert\lambda_{k}^{(n)}-\lambda_k\vert k^{-(1+2\beta)}
	\right)\sveq \\
	& \lesssim
	\Vert\theta\Vert_0\Vert\varphi\Vert_0\, n^{-r}\left(
	1 +
	\sum_{k=N_1}^n k^{\alpha(q-(1+2\beta))}
	\right)\sveq
	\end{align*}
	If $q=1$, since $2\alpha\beta > 1$, we get $R_D^{(2)}(\theta,\varphi) \lesssim \Vert\theta\Vert_0\Vert\varphi\Vert_0\, n^{-r}$. If now $q>1$, since $\alpha(q-1)\le r$,
	\begin{equation*}
	R_D^{(2)}(\theta,\varphi)
	\lesssim
	\Vert\theta\Vert_0\Vert\varphi\Vert_0\, n^{-r}\left(1 +\sum_{k=N_1}^n k^{r-2\alpha\beta}\right) 
	\lesssim
	\Vert\theta\Vert_0\Vert\varphi\Vert_0
	\begin{cases}
	n^{-r}\log n & \text{if } 2\alpha\beta-1=r, \\
	n^{-\min\lbrace r;\; 2\alpha\beta-1\rbrace} &  \text{else.} 
	\end{cases}
	\end{equation*}

	\vspace{0.5em}
	
	\underline{Total error} : Combining the three error terms $R_T(\theta,\varphi)$, $R_D^{(1)}(\theta,\varphi)$, and $R_D^{(2)}(\theta,\varphi)$,  and keeping the terms with the slowest convergence then gives the claim for the total error. 
\end{proof}

\section{Sampling a Galerkin--Chebyshev approximation of a random field}

In \Cref{sec:cheb} of the main article, we present an approach to generate samples of the Galerkin--Chebyshev approximation of a random field defined on a Riemannian manifold. We provide here additional implementation details and pseudo-code for this approach.

\subsection{An upper-bound for the eigenvalues of the stiffness matrix}\label{sec:upper_bound}

In order to define the polynomial $P_{\gamma, K}$ used to approximate the power spectral density $\gamma$ defining the random field, one needs to provide an upper-bound $\lambda_{\max}$ of the largest eigenvalue of the stiffness matrix~$\bm S$. Let us denote by $\lambda_{n}^{(n)}$ this maximal eigenvalue. Recall from \Cref{prop:diag_lapl_discrg} and \Cref{prop:diag_lapl_discr} of the main article that the eigenvalues of $\bm S$ are exactly those of the stencil $(\bm R, \bm C)$. As such, they can be upper-bounded by the maximum of the associated Rayleigh quotient, thus giving
\begin{equation*}
\lambda_{n}^{(n)} 
\le \max_{\bm x\in\R^n, \Vert \bm x\Vert_2=1} \frac{\bm x^T \bm R \bm x }{\bm x^T \bm C \bm x} 
\le\frac{\max\limits_{\bm x\in\R^n, \Vert \bm x\Vert_2=1} \bm x^T \bm R \bm x }{\min\limits_{\bm x\in\R^n, \Vert \bm x\Vert_2=1} \bm x^T \bm C \bm x} \peq
\end{equation*}
We recognize on the right-hand side of the last inequality the ratio between two Rayleigh quotients. Hence, we can conclude that an upper-bound $\lambda_{\max}$ of the eigenvalues of $\bm S$ is obtained by taking the ratio
\begin{equation*}
\lambda_{\max}=\frac{\lambda_{\max}(\bm R)}{\lambda_{\min}(\bm C)},
\end{equation*}
where $\lambda_{\max}(\bm R)$ (resp.\ $\lambda_{\min}(\bm C)$) is an upper-bound (resp.\ lower-bound) of the eigenvalues of the stiffness matrix $\bm R$ (resp.\ mass matrix $\bm C$). On the one hand, $\lambda_{\max}(\bm R)$ can be obtained using the Gershgorin circle theorem, thus requiring only to sum the entries of $\bm R$ row-wise (or column-wise) to get the bound. On the other hand, $\lambda_{\min}(\bm C)$ can be taken to be the inverse of an upper-bound $\lambda_{\max}(\bm C^{-1})$ of the eigenvalues of the inverse of $\bm C$. This upper-bound can in turn be obtained using a power iteration scheme which would require to solve linear systems defined by $\bm C$.

\subsection{Workflow and pseudo-code}

We now present the workflow used to generate samples of the Galerkin--Chebyshev approximation and some pseudo-code associated with the different steps of this workflow. 

The overall workflow is presented in \Cref{workflow_sim}. The weights of the Galerkin--Chebyshev approximation can be sampled using \Cref{alg:sim_gen}. This algorithm relies on the following two sub-algorithms:
\begin{itemize}
	\item an algorithm $\bm \Pi_{\bm S}$ taking as input a vector $\bm x$ and returning the product $\bm \Pi_{\bm S}(\bm x)=\bm S \bm x$;
	\item an algorithm $\bm{\Pi}_{(\sqrt{\bm C})^{-T}}$ taking as input a vector $\bm x$ and returning the solution $\bm y=\bm \Pi_{(\sqrt{\bm C})^{-T}}(\bm x)$ to the linear system
	$\big(\sqrt{\bm C}\big)^{T} \bm y = \bm x \peq$
\end{itemize}
In the most general case, we proposed implementations for these two algorithms that are recalled in \Cref{alg:piS_chol,alg:piC_chol}.

\begin{workflow}{Generate a sample of the discretized field $\widehat{\mathcal{Z}}_{n,K}$ in~\eqref{eq:fem_discr_approx}}\label{workflow_sim}
	\SetAlgoLined
	
	\vspace{1ex}
	
	\begin{enumerate}
		\item Compute the Cholesky factorization $\bm C=\bm L\bm L^T$ of the mass matrix $\bm C$;\label{item:chol}
		\vspace{1ex}
		\item Compute
		\begin{itemize}
			\item an upper-bound $\lambda_{\max}(\bm R)$ of the eigenvalues of $\bm R$ (using Gershgorin circle theorem),
			\item an upper-bound $\lambda_{\max}(\bm C^{-1})$ of the eigenvalues of the inverse of $\bm C$ (using the Cholesky factors of $\bm C$ to solve the linear systems in a power iteration scheme);
		\end{itemize}\label{item:lm}
		\vspace{1ex}
		\item Run \Cref{alg:sim_gen} using the implementations of $\bm\Pi_{\bm S}$, and  $\bm{\Pi}_{(\sqrt{\bm C})^{-T}}$ given in \Cref{alg:piS_chol,alg:piC_chol}, and taking $\lambda_{\max}=\lambda_{\max}(\bm R)\lambda_{\max}(\bm C^{-1})$.\label{item:sim}
	\end{enumerate}
\end{workflow}

Note that in the two particular cases presented in \Cref{sec:part_cases} of the main article, the first step of \Cref{workflow_sim} is no longer required, and the second step can be performed without requiring a power iteration scheme. Besides, the implementations of $\bm \Pi_{\bm S}$ and $\bm{\Pi}_{(\sqrt{\bm C})^{-T}}$ can be replaced by single products with sparse or diagonal matrices. This speeds up greatly the time needed to generate samples. To illustrate this,  we gave in \Cref{fig:effort} of the main article the computational time (and corresponding orders of polynomial approximation) needed to generate the samples used in the numerical experiment presented in \Cref{sec:cov_num} of the main article.

\begin{algorithm}
	\caption{Compute a sample of the discretized random field $\widehat{\mathcal{Z}}_{n,K}$ in~\eqref{eq:fem_discr_approx} }\label{alg:sim_gen}
	\begin{algorithmic}[1]
		\KwInputs{ \Statex Algorithms $\bm \Pi_{\bm S}$ and $\bm{\Pi}_{(\sqrt{\bm C})^{-T}}$ to compute products by $\bm S$ and $\big(\sqrt{\bm C}\big)^{-T}$,
			\Statex Estimate $\lambda_{\max}$ of the largest eigenvalue of $\bm S$, 
			\Statex Order of polynomial approximation $K$.}
		\KwOut{Gaussian random weights $\widehat{\bm Z}$ of the discretized random field~\eqref{eq:fem_discr_approx}.}
		
		\vspace{1ex}
		\hrule
		\vspace{1ex}
		
		\State Set $\widehat{\bm X}\leftarrow \bm 0$;
		\State Sample $\bm W \sim \mathcal{N}(\bm 0, \bm I)$;	
		
		\vspace{1ex}
		\State Compute the first $(K+1)$ coefficients $c_0,\dots,c_K$ of the Chebyshev series~\eqref{eq:cheb_sum} of the function $\displaystyle t\in [-1,1] \mapsto \gamma\left(\frac{\lambda_{\max}}{2} (1+t)\right)$ using the FFT algorithm;

		\vspace{1ex}
		\State Set $\bm y^{(-2)} \leftarrow \bm W$;
		\State Update $\widehat{\bm X}\leftarrow\widehat{\bm X} + c_0 \bm y^{(-2)}$;
		\If{$(K=0)$}
		\State Jump to Line \ref{algstep:solve};
		\EndIf
		\vspace{1ex}
		
		\State Set $\bm y^{(-1)} \leftarrow (2/\lambda_{\max})\bm \Pi_{\bm S}(\bm W) - \bm W$;
		\State Update $\widehat{\bm X}\leftarrow\widehat{\bm X} + c_1 \bm y^{(-1)}$;
		\If{$(K=1)$}
		\State Jump to Line \ref{algstep:solve};
		\EndIf
		
		\vspace{1ex}
		\For{$k=2$ \textbf{to} $K$}
		\State Compute $\bm y = (4/\lambda_{\max})\bm \Pi_{\bm S}(\bm y^{(-1)}) - 2\bm y^{(-1)}-\bm y^{(-2)}$;
		\State Update $\widehat{\bm X}\leftarrow\widehat{\bm X} + c_k \bm y$;
		\State Set $\bm y^{(-2)} \leftarrow \bm y^{(-1)}$;
		\State Set $\bm y^{(-1)} \leftarrow \bm y$;
		\EndFor
		\vspace{1ex}
		
		\State Compute $\widehat{\bm Z}  = \bm{\Pi}_{(\sqrt{\bm C})^{-T}}\left(\widehat{\bm X}\right)$;\label{algstep:solve}
		
		\KwRet{$\widehat{\bm Z}$.}
	\end{algorithmic}
\end{algorithm}

\begin{algorithm}
	\caption{Implementation of $\bm{\Pi}_{(\sqrt{\bm C})^{-T}}$}\label{alg:piC_chol}
	\begin{algorithmic}[1]
		
		\KwDepend{Cholesky factor $\bm L^T$ of $\bm C$.}
		\KwIn{Vector $\bm x$.}
		\KwOut{Vector $\bm y$ solution of $\big(\sqrt{\bm C}\big)^{T}\bm y=\bm x$.}
		
		\vspace{1ex}
		\hrule
		\vspace{1ex}
		
		\State Solve $\bm L^T \bm y  = \bm x$ by backward substitution;
		
		\KwRet{$\bm y$.}
	\end{algorithmic}
\end{algorithm}

\begin{algorithm}
	\caption{Implementation of $\bm\Pi_{\bm S}$}\label{alg:piS_chol}
	\begin{algorithmic}[1]
		\KwDepend{\Statex $\quad$ Cholesky factor $\bm L$ of $\bm C$, \Statex $\quad$ Stiffness matrix $\bm R$.}
		\KwIn{Vector $\bm x$.}
		\KwOut{Vector $\bm y=\bm S\bm x$.}
		
		\vspace{1ex}
		\hrule
		\vspace{1ex}
		
		\State Solve $\bm L^T \bm u  = \bm x$ by backward substitution;
		\State Update $\bm u \leftarrow \bm R \bm u$;
		\State Solve $\bm L \bm y  = \bm u$ by forward substitution;
		
		\KwRet{$\bm y$.}
	\end{algorithmic}
\end{algorithm}

\subsection{Numerical experiment: Polynomial approximation error}\label{sec:pol_num}

The polynomial approximation error refers to the error due to the fact that the \psd $\gamma$ defining the random field is in practice approximated by a polynomial. This polynomial is defined as a truncated Chebyshev series of order $K$ chosen by the practitioner. \Cref{prop:conv_cheb_ap} of the main article ensures that this error converges to $0$ as $K\rightarrow \infty$. 

To get a feeling of how fast this convergence can be, we consider the same setting as the one described for the truncation error study (cf. \Cref{sec:trunc_exp} of the main article). Let $n$ be a fixed truncation order. We compute for various choices of~$K$, the approximation error
\begin{equation*}
\Vert {\mathcal{Z}}_n-\widehat{\mathcal{Z}}_{n,K}\Vert_{L^2(\Omega;H)}^2
\end{equation*}
between the truncated expansion $\mathcal{Z}_n=\mathcal{Z}^{(n)}$ and its approximation $\widehat{\mathcal{Z}}_{n,K}$ obtained by replacing $\gamma$ by a Chebyshev series of order $K$. Following the proof of \Cref{prop:conv_cheb_ap}, this error can in particular be computed without requiring any simulations since it has a closed form given by
\begin{equation*}
\begin{aligned}
\Vert {\mathcal{Z}}_n-\widehat{\mathcal{Z}}_{n,K}\Vert_{L^2(\Omega;H)}^2
=\e\left[\Vert {\mathcal{Z}}_n-\widehat{\mathcal{Z}}_{n,K}\Vert_{0}^2 \right] 
&=\sum\limits_{k=1}^{n}(\gamma(\lambda_{k})-P_{\gamma,K}(\lambda_{k}) )^2.
\end{aligned}
\end{equation*}

These approximation errors are computed for four scenarios corresponding to truncation orders $n\in\lbrace 1024$, $10^4$, $100489$, $10^6\rbrace$ and for the \psd given by the parameters $\nu=1$ and $a=\pi/6$.  The results are presented in \Cref{fig:pol} and show that even for large truncation orders, a Chebyshev series of order of around $1000$ is enough to reach very small errors. Moreover, considering very large orders of approximation results in the error stagnating at the machine precision level.

\begin{figure}[h]
	\centering
	\includegraphics[width=0.5\textwidth]{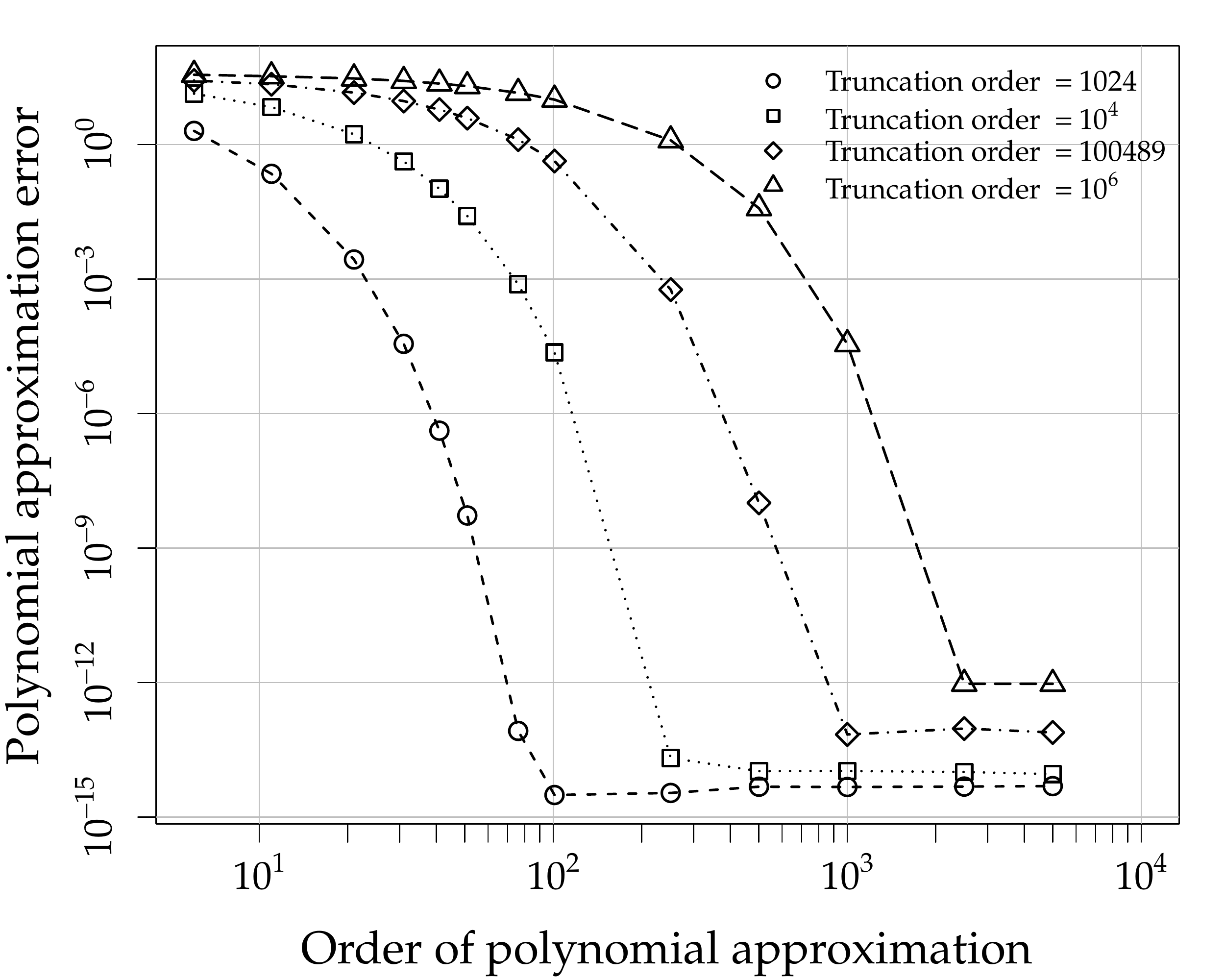}
	\caption{Polynomial approximation error $\Vert {\mathcal{Z}}_n-\widehat{\mathcal{Z}}_{n,K}\Vert_{L^2(\Omega;H)}^2$ on the sphere for the Matérn \psd with $\nu=1$ and $a=\pi/6$.}
	\label{fig:pol}
\end{figure}


\end{document}